\documentclass[11pt,a4paper,fleqn,leqno]{article}
\usepackage[utf8]{inputenc}
\usepackage[english]{babel}
\usepackage{latexsym,amssymb,amsthm,xcolor}
\usepackage[centertags]{amsmath}
\usepackage{mathtools}
\usepackage{scrtime}
\usepackage{graphicx}
\usepackage{color}
\usepackage{bbm}
\usepackage{framed}
\usepackage{fancyhdr}
\usepackage{xspace} 

\usepackage{appendix}                    

\usepackage{hyperref}

\usepackage{accents}           

\usepackage{thmtools}

\usepackage{scrtime}          
\newcommand{\Rand}[1]{\marginpar{#1}}
\renewcommand{\Rand}[1]{}

\newcommand{\be}[1]{\Rand{\vspace{0,6cm}\tt #1}\begin{equation}\label{#1}}
\newcommand{\beL}[1]{\Rand{\vspace{0,6cm}\tt #1}\begin{lemma}\label{#1}}
\newcommand{\belC}[2]{\Rand{\vspace{0,6cm}\tt #1}\begin{lemma}[#2]\label{#1}}
\newcommand{\beP}[1]{\Rand{\vspace{0,6cm}\tt #1}\begin{proposition}\label{#1}}
\newcommand{\bePC}[2]{\Rand{\vspace{0,6cm}\tt #1}\begin{proposition}[#2]\label{#1}}
\newcommand{\beD}[1]{\Rand{\vspace{0,6cm}\tt #1}\begin{definition}\label{#1}}
\newcommand{\beT}[1]{\Rand{\vspace{0,6cm}\tt #1}\begin{theorem}\label{#1}}
\newcommand{\beC}[1]{\Rand{\vspace{0,6cm}\tt #1}\begin{corollary}\label{#1}}
\newcommand{\beA}[1]{\Rand{\vspace{0,6cm}\tt #1}\begin{assumption}\label{#1}}
\newcommand{\bea}[1]{\Rand{\vspace{0,7cm}\tt #1\vspace{-0,7cm}}\begin{eqnarray}\label{#1}}

\newtheorem{xx}{\bf xxx}

 \newcommand{\eps}{\varepsilon}

 \newcommand{\R}{\mathbb{R}}
 \newcommand{\N}{\mathbb{N}}
 
 \newcommand{\Z}{\mathbb{Z}}

 \newcommand{\bigtimes}{\prod}

 \newcommand{\E}{\mathbb{E}}
 \renewcommand{\P}{\mathbb{P}}
\newcommand{\1}{\mathbbm{1}}

\newcommand{\sm}{\smallskip}

\newcommand{\wt}{\widetilde}
\newcommand{\dr}{\underline{\underline{r}}}

\def\CA{\mathcal{A}}
\def\CB{\mathcal{B}}

\def\CM{\mathcal{M}}

\def\CL{\mathcal{L}}

\def\E{\mathbb{E}}

\def\N{\mathbb{N}}
\def\P{\mathbb{P}}

\def\R{\mathbb{R}}

\def\U{\mathbb{U}}

\def\Z{\mathbb{Z}}

\DeclareMathSymbol{\varNu}{\mathord}{letters}{78}

\newcommand{\wh}{\widehat}
\newcommand{\ee}{\end{equation}}
\newcommand{\eea}{\end{eqnarray}}
\newcommand{\bean}{\begin{eqnarray*}}
\newcommand{\eean}{\end{eqnarray*}}
\newcommand{\noi}{\noindent}
\newcommand{\intl}{\int\limits}
\newcommand{\liml}{\lim\limits}
\newcommand{\suml}{\sum\limits}

\newcommand{\ve}{\varepsilon}

\newtheorem{theorem}{Theorem}[section]
\newtheorem{proposition}[theorem]{Proposition}
\newtheorem{corollary}[theorem]{Corollary}
\newtheorem{lemma}[theorem]{Lemma}
\newtheorem{assumption}{Assumption}

\makeatletter                                           
\def\th@newremark{\th@remark\thm@headfont{\bfseries}}   
\makeatletter                                           

\theoremstyle{definition}
\newtheorem{definitionx}[theorem]{Definition}

\theoremstyle{newremark}
\newtheorem{remarkx}[theorem]{Remark}
\newtheorem{examplex}[theorem]{Example}

\newenvironment{example}
  {\pushQED{\qed}\examplex}
  {\popQED\endexamplex}
\newenvironment{remark}
  {\pushQED{\qed}\remarkx}
  {\popQED\endremarkx}
\newenvironment{definition}
  {\pushQED{\qed}\definitionx}
  {\popQED\enddefinitionx}

\usepackage{mathtools}

\DeclarePairedDelimiter\abs{\lvert}{\rvert}%
\DeclarePairedDelimiter\norm{\lVert}{\rVert}%

\makeatletter
\let\oldabs\abs
\def\abs{\@ifstar{\oldabs}{\oldabs*}}
\let\oldnorm\norm
\def\norm{\@ifstar{\oldnorm}{\oldnorm*}}
\makeatother

\usepackage{enumerate}

\newcommand{\vphi}{\varphi}
\renewcommand{\epsilon}{\varepsilon}
\newcommand{\dx}{{\textup{d}}}                                   
\newcommand{\independent}{\perp\!\!\!\perp}                    
\DeclareMathOperator{\supp}{supp}                              
\DeclareMathOperator{\diam}{diam}                              
\DeclareMathOperator{\Poiss}{Poiss}                            
\newcommand{\eqd}{\overset{\textup{d}}{=}}                     
\newcommand{\dunderline}[1]{\underline{\underline{#1}}}        
\newcommand{\ubar}[1]{\underaccent{\bar}{#1}}                  

\allowdisplaybreaks[4]

\newcommand{\mcA}{\mathcal{A}}
\newcommand{\mcB}{\mathcal{B}}

\newcommand{\mcF}{\mathcal{F}}

\newcommand{\mcL}{\mathcal{L}}
\newcommand{\mcM}{\mathcal{M}}
\newcommand{\mcN}{\mathcal{N}}

\newcommand{\mcS}{\mathcal{S}}

\newcommand{\mfe}{\mathfrak{e}}

\newcommand{\mfP}{\mathfrak{P}}

\newcommand{\mfS}{\mathfrak{S}}

\newcommand{\mfu}{\mathfrak{u}}
\newcommand{\mfU}{\mathfrak{U}}
\newcommand{\mfv}{\mathfrak{v}}
\newcommand{\mfV}{\mathfrak{V}}
\newcommand{\mfw}{\mathfrak{w}}

\newcommand{\mfx}{\mathfrak{x}}
\newcommand{\mfX}{\mathfrak{X}}
\newcommand{\mfy}{\mathfrak{y}}

\newcommand{\bbD}{\mathbb{D}}
\newcommand{\bbE}{\mathbb{E}}

\newcommand{\bbK}{\mathbb{K}}

\newcommand{\bbM}{\mathbb{M}}
\newcommand{\bbN}{\mathbb{N}}

\newcommand{\bbP}{\mathbb{P}}

\newcommand{\bbR}{\mathbb{R}}

\newcommand{\bbU}{\mathbb{U}}

\newcommand{\pfd}{\Theta}

\newcommand{\lan}{\langle}
\newcommand{\ran}{\rangle}

\newcommand{\hu}{h}
\newcommand{\hd}{\ubar{h}}
\newcommand{\Levy}{L{\'e}vy }
\DeclareMathOperator{\anz}{\#}
\newcommand{\nn}{\nonumber}
\newcommand{\cadlag}{c\`adl\`ag\xspace}
\newcommand{\dGPr}{d_{\textup{GPr}}}

\usepackage{color}
\definecolor{darkgreen}{rgb}{0.0,0.6,0.1}

\begin{document}

\title
{Branching trees I: Concatenation and infinite divisibility}
\author{P. Gl\"ode$^{1,4,5}$, A. Greven$^{2,4}$, T. Rippl$^{3,4}$}
\date{{\today}\\
}
\maketitle

\begin{abstract}
The goal of this work is to decompose random populations with a genealogy in subfamilies of a given degree of kinship and to obtain a notion of infinitely divisible genealogies.
We model the genealogical structure of a population by (equivalence classes of) ultrametric measure spaces (um-spaces) as elements of the  Polish space $\U$ which we recall.
In order to then analyze the family structure in this coding we introduce an algebraic structure on um-spaces (a consistent collection of semigroups).
This allows us to obtain a path of decompositions of subfamilies of fixed kinship $ h $ (described as ultrametric measure spaces), for every depth $ h $ as a measurable functional of the genealogy.

Technically the elements of the semigroup are those um-spaces which have diameter less or equal to $2h$ called \emph{$h$-forests} ($h> 0$).
They arise from a given ultrametric measure space by applying maps called $ h-$truncation.
We can define a concatenation of two $ h$-forests as binary operation.
The corresponding semigroup is a Delphic semigroup and any $h$-forest has a unique prime factorization in $h$-trees (um-spaces of diameter  less than $2h$).
Therefore we have a nested $\R^+$-indexed consistent (they arise successively by truncation) collection of Delphic semigroups with unique prime factorization. 

Random elements in the semigroup are studied, in particular infinitely divisible random variables.
Here we define infinite divisibility of random genealogies as the property that the $h$-tops can be represented as
concatenation of independent identically distributed h-forests for every $h$ and obtain a L\'evy-Khintchine 
representation of this object and a corresponding representation via a concatenation of points of a Poisson point process of h-forests. 

Finally the case of discrete and marked um-spaces is treated  allowing to apply the results to both the individual based and most important  spatial populations.

The results have various applications. 
In particular the case of the genealogical  ($\U$-valued) Feller diffusion and genealogical ($\U^V$-valued) super random walk is treated based on the present work in \cite{ggr_tvF14} and \cite{ggr_GeneralBranching}.

In the part II of this paper we go in a different direction and refine the study in the case of continuum branching populations, give a refined analysis
of the Laplace functional and give a representation in terms of a Cox process on h-trees, rather than forests.
\end{abstract}

\noi

\noi

\footnoterule
\noi
\hspace*{0.3cm}\\
{\footnotesize $^{1)}$ The Faculty of Industrial Engineering and Management, Technion - Israel Institute of Technology, Technion City, Haifa 3200003, Israel, gkarl@technion.ac.il}\\
{\footnotesize{}}\\
{\footnotesize $^{2)}$ Department Mathematik, Universit\"at Erlangen-N\"urnberg, Cauerstr. 11,
D-91058 Erlangen, Germany, greven@math.fau.de}\\
{\footnotesize $^{3)}$ Institut für Mathematische Stochastik, Goldschmidtstrasse 7, D-37077 G\"ottingen, Germany, trippl@uni-goettingen.de}\\
{\footnotesize $^{4)}$ Supported by DFG-Grant GR 876/15-1,2 of AG}\\
{\footnotesize $^{5)}$ Supported with a postdoc fellowship from Technion, Haifa.}

 \newpage

\tableofcontents
\newpage
\vspace*{15mm}


\newcounter{secnum}
\setcounter{secnum}{\value{section}}
\setcounter{section}{0}
\setcounter{secnumdepth}{3}
\section{Introduction}
\setcounter{equation}{0}
\renewcommand{\theequation}{\mbox{\arabic{secnum}.\arabic{equation}}}
\numberwithin{equation}{section}

We are interested in random genealogies for example those arising from an evolving branching population.
We use here a concept of genealogy which is suited to consider the \emph{evolution in time} of this structure described by martingale problems aiming eventually at spatial situations and which is based on \emph{ultrametric measure spaces} (more precisely their equivalence classes).
In particular do we follow a \emph{different point of view} then in the literature describing genealogies of populations by \emph{labeled} trees, see for example \cite{Neveu86},\cite{NP89},\cite{LG89},\cite{LJ91},\cite{Abr92} continuing up to the present, or genealogies are modeled as measure $\R$-trees see \cite{EPW06,Gro99}, we comment later on relations.
For more information on our approach to genealogies see the survey article \cite{DG18evolution}.

An important role in this research project will be played by branching processes, for example the genealogy of a continuous state Feller branching diffusion respectively spatial versions thereof as super random walk and here in this paper we lay the foundations to study such objects.
For more on these processes see \cite{ggr_tvF14,ggr_GeneralBranching}
The question is to consider for varying depths of kinship decompositions of the population in subfamilies specifying their genealogy as well as their size and to
 see whether we can give a cluster representation of the genealogy, i.e. can we view the  genealogy of the sub-populations as a \emph{Cox point process} on the space of genealogies \emph{modeled as ultra-metric measure spaces}?
 
The term family decomposition appears in the literature of branching processes frequently, see e.g.~\cite{DG96}, \cite{D93} based on the {\em historical process} \cite{DP91}, but not always corresponding exactly to genealogies, but here we will get in general an interpretation in terms of genealogies, using a specific concept of "genealogy", even for continuum state processes as the limit of the obvious meaning it has in discrete Galton-Watson processes.

In order to study this type of questions following \cite{EPW06,GPW09}, we code the genealogy as equivalence classes of \emph{ultrametric measure spaces} (and  in case of a population distributed in space, where individuals have a location a {\em marked} one \cite{DGP11}) turning the genealogy of the time-$t$ population into a random variable with values in a
Polish space. 
A survey on this approach is found in \cite{DG18evolution}.
Here the distance describes the degree of kinship and is twice the time back to the
most recent common ancestor giving in fact an ultrametric. Now we can fix a degree of kinship, say $h >0$ and decompose the space into 
{\em open 2h-balls} and this induces a number of ultrametric measure spaces, each describing the
genealogy of a subfamily and we can obtain one ultrametric measure space by connecting the 
subspace to a {\em forest}, by giving distance $2h$ to points in different balls. This $h$-{\em concatenation}
of the ultrametric measure spaces describes then the \emph{$2h$-family decomposition}.
In particular, we can use the algebraic structure of the {\em concatenation} and the framework of 
ultrametric measure spaces to characterize the decomposition.
This allows to carry out calculations to obtain  properties of the whole {\em path of 
$2h$-family decompositions} for $h \in (0,\diam (\textup{space})/2)$.
However since we take \emph{equivalence classes} of such objects some care is needed.

The $h$-concatenation is a binary operation and we show that it defines a {\em topological 
semigroup}. Further topological and algebraic properties are established and we notice a close 
relationship in structure of our results and arguments to those in a Cartesian semigroup $(\bbM,\boxplus)$ on metric 
measure spaces, introduced by Evans and Molchanov \cite{EM14}, even though the binary operations used in their and in our paper are \emph{completely different} (but there the algebraic properties are similar).

One of the most important questions for semigroups is whether it is atomic, i.e. every element can be
written as product of irreducible elements, and furthermore if it is a {\em unique factorization domain},
i.e. the representation is unique up to order.
We show the answer to both questions is yes, getting for fixed depths a well-defined \emph{family decomposition} of an ultrametric measure space and finally also a well-defined {\em path} of family decompositions of varying depth. 
Starting with a fixed ultrametric measure space this defines a {\em \cadlag{} path of family decompositions}.
This structure is {\em much richer and more informative} than having just the semigroup structure of $ \U $.
If we decrease the parameter starting from the diameter, then we obtain a succession of refinements of the family  decomposition as $h$ decreases to zero and as limit the original ultrametric measure space.
There are more applications of this structure we exploit in \cite{ggr_GeneralBranching}.

We establish that the association of the path of decompositions  is measurable and hence is indeed suitable to give rise to a legitimate random variable.
The next task is to turn to {\em random} genealogies.
The goal of the research program is to understand better the probabilistic structure of this random path for genealogies of evolving populations.

\sm

We begin by giving a concept of {\em infinite divisibility} on the level of {\em random genealogies} in our coding.
This lifts the concept one has for the population size process (an $\R_+$-valued process) or
if we are interested for example in the Dawson-Watanabe process that of measure-valued processes.
Since this should be related to properties of the family decomposition we want to use the 
{\em algebraic structure} from above (h-forests with the concatenation operation).
There is a quite general concept of infinite divisibility for random variables with {\em values in semigroups} which we follow here (\cite{BCR84}) but now {\em lift it} to a whole \emph{consistent  collection of semigroups}. We shall discuss in Section (\ref{ss.compDMZ08}) the subtle relation to the classical theory of semigroups and negative definite functions i.e. harmonic analysis. 

\sm

We are then establishing a \emph{L\'evy-Khintchine formula} for the Laplace functional implying a representation of the forests
of concatenated 2h-forests corresponding to the random state of the genealogy as a concatenation
of independent {\em subfamily forests} which are generated by a Poisson point process of $2h$-forests
for every $h$ in $(0,t]$ for random ultrametric measure space of diameter $2t$ and which are {\em truncation consistent}.
As example we conclude showing the state of genealogies of branching processes are infinitely divisible if the initial state has this property.

The results hold also for the genealogies of \emph{spatial} models and similarly models where individuals carry a type.
Examples are super random walks or Dawson-Watanabe processes as well as multitype branching processes, provided they exist as um-measure space valued processes (an issue addressed in forth-coming work \cite{ggr_tvF14}).
This allows to establish here a \emph{\Levy-Khintchine formula for genealogies} for this very important class of spatial population models. 
\sm

{\bf Perspectives} \quad
The key results in this paper can be used to get information about concrete stochastic systems.
In \cite{ggr_tvF14} we apply our results to discuss at length the example of the \emph{$\U$-valued Feller diffusion} and its spatial version the \emph{genealogical} ($\U$-valued) \emph{super random walk}, the genealogy of this classical model represented by ultra-metric measure spaces.
We also use results in \cite{ggr_GeneralBranching} to study criteria for generators allowing to obtain a "branching property" from the form of the generator working very well in our context of genealogical and/or historical information.

Finally in part II of this paper \cite{ggr_MBT} we apply the developed techniques to the study of continuum mass branching
populations and give a \emph{more refined} analysis of the present  \Levy-Khintchine representation for that special case. The key point there will be to move
to a representation of the $2h$-tops in terms of the {\em Cox process on $2h$-trees} rather than
forests as in this part I, i.e. splitting into the descending ancestors at depth $h$, that is a {\em representation by the prime}
 elements of $(\U(h))^\sqcup$, i.e. elements of $\U(h)$,  rather than Poisson
point processes on forests. This is related to the {\em Cox cluster representation} of
historical processes associated with spatial branching processes, see \cite{DP91} which is developed in \cite{ggr_GeneralBranching}. 
This will also allow there better to understand the structure of the random genealogy by generating it dynamically.

Further tasks are to apply the present theory also to genealogies of $\alpha$-stable processes and to modify concepts and theory to deal with genealogies appear in the form of algebraic trees as developed by  Löhr and Winter \cite{LW,LMW}.

\sm

{\bf Outline } 
In Section~(\ref{s.basic}) we collect all of the important concepts and results.
 Section~(\ref{s.tt}) contains all of the proofs concerning the structure of the state spaces and semigroups.
The Section~\ref{s.proofsprobm} proves statements on random variables with values in $\U$.
 Proofs for infinite divisibility results can be found in Section~(\ref{s.lkfproof}).
 Finally an appendix contains basic facts on ultrametric measure spaces and on boundedly finite measures.

\section{Basic concepts and Results}
\label{s.basic}

In this section we introduce first in Subsection~(\ref{ss.umms}) the {\em state space} of our genealogy-valued random variables.
Subsection~(\ref{ss.ttfd}) introduces decompositions in disjoint open balls with weights
providing the precise meaning of the concept of {\em family decomposition}.
Subsection~(\ref{ss.anatools}) gives some key analytical instruments, namely {\em(truncated) polynomials} and properties
of quantities related to family decompositions.  
In Subsection~(\ref{ss.anatoolsrandom}) we pass to \emph{random} ultrametric measure spaces and 
introduce {\em Laplace functionals}.
In Subsection~(\ref{ss.infdiv}) we relate {\em infinite divisible random} ``trees'' with {\em Poisson point processes} of ``forests'' giving a version of the L\'evy-Khintchine formula and then we discuss particular cases with additional properties, in Subsection (\ref{ss.branchproc}) an application to {\em branching} with an example of a concrete branching process in Subsection~\ref{ss.exam} and in Subsection (\ref{ss.marked}) to individual based (discrete) populations but more important to our main goal to understand \emph{spatial} populations by  generalizations to genealogies of {\em spatial} populations modeled by {\em marked} metric measure spaces.
Section~\ref{ss.compDMZ08} discusses the relation to harmonic analysis and~\ref{ss.outps} gives an outline for the proof section.

\subsection{Ultrametric measure spaces}
\label{ss.umms}

A building block of our description of genealogies are ultrametric measure spaces. An ultrametric measure
space is a triple $(U,r,\mu)$, with $U$ a set, $r$ an ultrametric on $U$ such that $(U,r)$ is a Polish space
and with $\mu$ a {\em finite} Borel measure on $(U,\CB(U))$ with $\CB$ denoting the Borel $\sigma$-algebra.
Here $U$ describes the individuals of a population, $r$ the genealogical distance between individuals 
and $\mu$ is the multiple (the population
size) of the sampling measure, a probability measure on $(U,\CB(U))$. 

We say that two such triples $(U,r,\mu)$ and $(U',r',\mu')$ are equivalent if there is a {\em measure preserving isometry} between the \emph{two supports} of $\mu$ resp. $\mu^\prime$. 
The {\em equivalence class} of a triple $(U,r,\mu)$ and the total mass is denoted by
\begin{equation}\label{rg0}
\mfu=[U,r,\mu] \mbox{  and  } \bar \mfu=\mu(U).
\end{equation}

Extending the space of {\em ultrametric probability measures spaces} from \cite{GPW09} to \emph{finite} measures see in particular Section 2.4. in \cite{Gl12} or \cite{ALW14a}, the basic state space of our random variables is the following.

\begin{definition}[Ultrametric measure spaces]
\label{D.umsp}

Define the space
\be{rg1}
  \bbU:=\text{set of \emph{isomorphy classes} of ultrametric measure spaces with finite measure}
\end{equation}
endowed with the {\em Gromov-weak topology}.
This topology on $\U$ is metrized by the {\em Gromov-Prokhorov metric $d_{\text{GPr}}$} such that $\U$ is a {\em Polish space}, see \cite{GPW09}, \cite{Gl12}, \cite{DG18evolution}.

Recall that in this topology sequences converge iff all distance matrix measures (see \eqref{rg6}) converge weakly to such a measure of a limit element in $\U$. 
The null measure on a measure space and the null tree $[\{1\},r,0]$ are here both denoted simply by $0$.
\end{definition}
We refer the reader to the Appendix~(\ref{s.umspaces}) for detailed definitions and and to \cite{D93} for facts on spaces of measures and corresponding weak topologies.

\begin{remark}[Ultrametric spaces and trees]\label{R.umsptr}
Recall that any ultrametric space $(U,r)$ (with finite diameter) can be imbedded isometrically 
into an (rooted) $\bbR$-tree 
such that the {\em leaves} of the $\bbR$-tree correspond to the elements of $U$.
In particular we are then able to talk about a {\em most recent common ancestor} of two points of $U$,
which is the unique point in the $\R$-tree such that the distance to each of two elements of $U$ is
half their distance in $(U,r)$.
\end{remark}

\begin{remark}[Equivalent ultrametrics]\label{R.um.transform}
Under certain conditions transformations of the metric result in equivalent 
topologies that is topologies with the \emph{same converging sequences}. Let 
$\tau:[0,\infty] \to [0,\infty]$ with $\tau(0)=0$ and $\tau((0,\infty]) \subset (0,\infty]$ continuous at zero and (not necessarily strictly) increasing. Let $\mfu = [U,r,\mu] \in \U$ and define a new ultrametric measure space:
\begin{equation}\label{grx52}
 \tau^*(\mfu) = [U, 2 \tau \circ \frac{r}2, \mu] \, .
\end{equation}
By the above remark this can be viewed as transforming the time back to the most recent
ancestor.
It has the {\em same topology} as $\mfu$ but a {\em different geometry}. A particular example is $\tau_a(r)=ar$ written 
\begin{equation}\label{e.ar2}
 a \circledast \mfu := \tau^*_a(\mfu) \mbox{ for } a\geq 0 \mbox{ and } \mfu \in \U. 
\end{equation}
\end{remark}

\subsection{Family decomposition: the semigroup of  \texorpdfstring{$h$}{h}-forests}
\label{ss.ttfd}
 
We want to decompose an ultrametric measure space into disjoint open balls of a fixed radius, say $h>0$, 
corresponding to subfamilies which are descendants of a single MRCA time $h$ back, recall Remark (\ref{R.umsptr}). To do this we need some notation.

For a measurable $A \subset [0,\infty)$ we set 
\begin{equation}\label{e756}
\U(A) = \{\mfu=[U,r,\mu] \in \U:\, \mu^{\otimes 2}(\{(u,v)\in U^2:r(u,v)\not\in A\})=0 \} 
\end{equation}
as the set of ultrametric measure spaces which only realize distances in $A$. 
We give more definitions of sets and a binary operation:
\begin{definition}[Forests, trees and concatenation]\label{D.concat}\mbox{}\\
 Let $h\geq 0$.
 \begin{enumerate}
  \item Define the subset of \emph{$h$-forests}
  \begin{equation}\label{e.tr46}
  \U(h)^\sqcup := \U([0,2h]) = \{[U,r,\mu]\in\bbU: 
\mu^{\otimes2}(\{(u,v)\in U^2:r(u,v)\in(2h,\infty)\})=0\} \, .
  \end{equation}
  \item For $\mfu, \mfv \in \U(h)^\sqcup$ with $\mfu = [U,r_U,\mu], \mfv = [V,r_V,\nu]$ define the 
\emph{$h$-concatenation}:
  \begin{equation}\label{e.tr47}
  \mfu \sqcup^{h} \mfv := [U \uplus V, r_U \sqcup^{h} r_V, \mu + \nu ] \, ,
  \end{equation}
  where $\uplus$ is the disjoint union of sets and $r_U \sqcup^{h} r_V |_{U \times U} = r_U$, $r_U \sqcup^{h} r_V |_{V \times V} = r_V$ and for $x \in U$, $y \in V$:
  \begin{equation}\label{grx53}
   \left( r_U \sqcup^{h} r_V \right) (x,y) =    2h
   \, .
  \end{equation}
  We write $\sqcup$ if $h>0$ has been fixed.
  \item Define the subset of \emph{$h$-trees}
  \begin{equation}\label{e.tr48}
  \U(h) := \U([0,2h)) =
  \{[U,r,\mu]\in\bbU: \mu^{\otimes2}(\{(u,v)\in U^2:r(u,v)\in[2h,\infty)=0\}\,.
  \end{equation}
 \end{enumerate}
 Forests and trees are endowed with the {\em relative topology}  from $\U$.
\end{definition}

\begin{remark}\label{r.765}
For an ultrametric measure space with diameter $2h$ we have always a decomposition into countably many open $h$-balls.
Each of these balls generates an ultrametric measure space by restriction. This decomposition is unique.
Then we expect that the equivalence classes of ultrametric measure spaces generated by this decomposition induces a unique decomposition of the corresponding equivalence class of the full space.
Here we based this however on a \emph{selection of representatives} for all these objects which have a unique decomposition.
We want to lift this to the equivalence classes.
To see that this is feasible the simpled way is to proceed with purely algebraic and topological arguments.
\end{remark}

\begin{remark}\label{r790} 
Recalling Definition \eqref{D.umsp} we observe:
 \begin{enumerate}
  \item Note $\U(h)^\sqcup$ and $\U(h)$  are separable metric spaces in their restricted topologies and the former is the completion of the latter hence Polish, see Proposition (\ref{l.frsts.clsd}).
  \item Observe that $\U \neq \bigcup_{h>0} \U(h)^\sqcup$, since trees of infinite diameter are not included on the r.h.s.
  \item For $h,h' >0$ it is easy to see that the semigroups $(\U(h)^\sqcup, \sqcup^h)$ and 
$(\U(h')^\sqcup, \sqcup^{h'})$ are isomorphic as topological groups via the mapping 
$\tau^\ast_{h'/h}:\bbU(h)^\sqcup\to\bbU(h')^\sqcup$.
\end{enumerate}
\end{remark}

\begin{remark}\label{r800}
We note that the distinction between forests and trees is not preserved under monotone transformations as in Remark~(\ref{R.um.transform}) of
the ultrametric, but as a geometric property only under {\em strictly} monotone ones.
\end{remark}

With each forest we associate a path of decompositions in subfamilies.
We define for that purpose the {\em $h$-truncation}. 

 \begin{definition}[$h$-truncation]\label{D.h-trunc}
  Let $h\geq 0$.
  The mapping 
  \begin{equation}\label{e.tr334} \tau(h): \begin{cases} \U &\to \U(h)^\sqcup,\\
                               \mfu= [U,r,\mu] &\mapsto \lfloor\mfu\rfloor(h) = [U,r \wedge 2h, \mu]
                             \end{cases}
  \end{equation}
  is called $h$-truncation.
 \end{definition}
\begin{remark}\label{r.778}
Note that for $ h \in [0,t], \; \tau(h) \left( (\U(t)^\sqcup) \right) \subseteq (\U(h))^\sqcup $ and this is a homomorphism of these topological semigroups.
\end{remark}
We can now associate with every element of $ \U $ or $ \U(t)^\sqcup $ an object which is a key object in applications but also contains the key mathematical structure of $ \U $ to study infinite divisibility of $ \U-$valued random variables and which allows to make precise the concept of family decompositions.
\begin{definition}[Associated family decomposition: abstract version]\label{D.assfam}
For $ \mfu \in \U $ there is an associated path of family decompositions if $ t \in [0,\infty] $ is the diameter of $ \mfu $:
\begin{equation}\label{e780}
\left(\tau (h)(\mfu)\right)_{h \in (0,t)}.
\end{equation}
\end{definition}
The first result tells us that any forest can be seen as a unique (at most countable) {\em concatenation of trees} i.e. elements of $\U$. 

The result clarifies also the algebraic structure of the binary operation $\sqcup$. Recall that an element $ x $ not the identity $ e $ in a semigroup is called \emph{irreducible} if {\em x=yz} implies {\em that x=y or y=e} and {\em prime} if it divides a concatenation only if it divides one of the factors, in general a stronger property.
We will need the concept of Delphic semigroups which we recall.
\begin{remark}\label{R.delphic}
 \emph{Delphic semigroups} were introduced by \cite{Kendall68} and further studied in \cite{Davidson68}. They are commutative, topological semigroups where the set of divisors of any element is compact and there is a continuous homomorphism into $([0,\infty),+)$ with trivial kernel which is here given by $ \mfu \to \bar \mfu $. 
\end{remark}
\begin{remark}\label{r.796}
Note that we can define for \emph{countable} index sets concatenations by taking limits of the finite concatenations of finitely many elements from the index set and requiring that the limit exists for all choices of the finite subsets.
\end{remark}
We prove in Section \eqref{ss.concsemi} the following:

\begin{theorem}[Semigroup structure, truncation consistency]\label{p.delphic}
~
\begin{enumerate}
\item The algebraic structure
$(\U(h)^\sqcup, \sqcup^{h})$ is a Delphic semigroup, see Remark (\ref{R.delphic}).
The set of irreducible elements is $\U(h)$ and 
any $\mfu \in \U(h)^\sqcup$ has a unique (up to order) decomposition:
 \begin{equation}\label{e.tr49}
 \mfu = \sideset{}{^{h}}\bigsqcup_{i\in I} \mfu_i \, ,
 \end{equation}
 where $I$ is a countable index set and $\mfu_i \in \U(h) \setminus \{0\}$.
 We can associate via $ \tau(h^\prime)$ for each $h^\prime \in [0,h]$ with $\mfu$ a path of $ h^\prime-$ decompositions. 
 
\item The $h$-decompositions are consistent w.r.t. truncation, i.e. $ \mfu \in \U (h)^\sqcup$ and for $h' \in [0,h)$
 \begin{align}
  \label{e.ag112} 
\tau(h')\;(\mfu^{(h)}_1 \sqcup...\sqcup \mfu^{(h)}_\ell) =\mfu^{(h^\prime)}_1 \sqcup...\sqcup \mfu^{(h^\prime)}_m,\\
\nonumber \mbox {where} \quad \lbrace\mfu^{(h')}_1,...\mfu^{(h')}_m \rbrace, \lbrace\mfu^{(h)}_1,...\mfu^{(h)}_\ell \rbrace 
\end{align}
 are the $ h^\prime $ respectively $ h $-decomposition of $\mfu$ and similarly for countable decompositions.
\end{enumerate}
 \end{theorem}

In other words there exists for every $h$ for a forest of diameter $t$ a unique family decomposition
with kinship degree $h$. As a consequence we can associate with every forest with diameter $t$
a path of family decompositions, where the "time-index" is degree of kinship $h \in [0,t]$.
The different decompositions are \emph{consistent} 
in the sense that they are successive truncations.

\begin{remark}\label{r.hypoth}
The semigroup $(\bbM_1, \boxplus)$ from \cite{EM14} has the property that an element can be decomposed into countably many  prime elements. 
Since both semigroups allow a unique factorization in prime elements, they can be understood as free semigroups with a certain set of generators, the prime elements.
We do not know of a ``natural'' isomorphism between the two semigroups which is continuous.
However in the sequel we only use that many arguments carry over since they only use the algebraic structure.
\end{remark}

\begin{figure}
  \includegraphics[width=14cm]{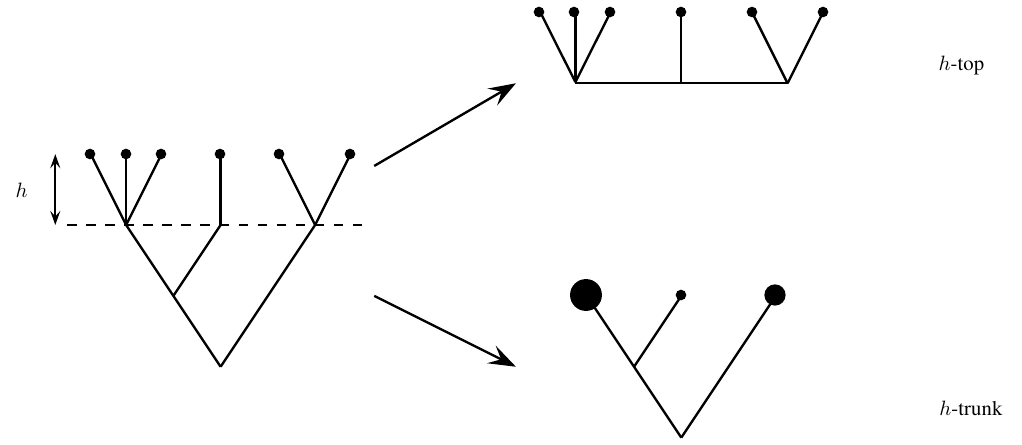}\\
  \caption{Example of $h$-top and $h$-trunk.}\label{f2607131839}
\end{figure}

The previous result allows us next to formalize the idea of a family decomposition for an 
ultrametric measure space and is the key object for the analysis of the genealogies of 
branching populations.

For that purpose we want to extract the $h$-top consisting of the decomposition in the open $2h$-balls of an element $\mfu \in \U$ and in addition an object containing the ancestral relations of the different $2h$-balls by a second element of $\U(h)^\sqcup$ which is complementary to the $h$-top namely the so called $h$-trunk.

\begin{definition}[Tops and trunks]\label{d.tops.trunks}
Let $h>0$ and $\mfu = [U,r,\mu] \in \U$:
 \begin{enumerate}
  \item Define the \emph{$h$-top} $\lfloor\mfu\rfloor(h) := [U,r \wedge 2h, \mu] \in \U(h)^\sqcup$.
  \item Suppose $\lfloor\mfu\rfloor(h) = \bigsqcup^{h}_{i\in I} \mfu_i$ as in \eqref{e.tr49} with at most countable index set $I$ and $\mfu_i \in \U(h)\setminus\{0\}$
  and write $\mfu_i = [U_i,r_i,\mu_i]$ for $i \in I$. The \emph{$h$-trunk} of $\mfu$ 
is defined as the ultrametric space
\begin{equation}\label{e.ar34a} 
\lceil \mfu\rceil(h) = [I, r^\ast, \mu^\ast],
\end{equation}
with (recall $r$ is ultrametric)
\be{ar1a}
r^\ast (i, i^\prime) = \inf \{ r(u,v) -2h | u \in U_i,  v \in U_{i^\prime} \} \mbox{  for  } i,i^\prime \in I
\ee and the weights
\be{ar2a}
\mu^\ast (\{i\}) = \mu_i(U_i).
\ee \end{enumerate} \end{definition}
Having Theorem~(\ref{p.delphic}) we can regard the $h$-top of $\mfu$ as a collection of elements in $\U(h)$ and we will make use of this identification frequently.
\sm

\begin{remark}[Family decompositions]\label{r.famdecom}
Recall the connection between ultrametric measure spaces and $\R$-trees explained in 
Remark (\ref{R.umsptr}). In particular, we can view $\mfu$ as the leaves of a family tree for a population. Then, 
Theorem (\ref{p.delphic}) applied to $\lfloor\mfu\rfloor(h)$ can be stated in the way that a population represented 
by $\mfu$ has a unique family decomposition of depth $h$, that is, a \emph{decomposition into subfamilies} 
$\mfu_i$, $i\in I$, where within each subfamily all individuals share a common ancestor whose death 
dates back at most time $h$. We will speak of the $h$-family decomposition into sub-families of 
individuals whose degree of kinship is at most $h$. In a similar way we can think of $I$ as the set 
of ancestors who lived at time $h$ back in time. The metric of the trunk encodes the genealogy of 
the $h$-ancestors. If the diameter of $\mfu$ is $2t$, then we can view $\mfu$ as the population alive 
at time $t$ and Theorem (\ref{p.delphic}) applied to each $h\in(0,t]$ induces a (unique) collection of 
family decompositions 
\begin{equation}\label{e757}
{\mfu^h_i: i\in I^h}, h\in(0,t].
\end{equation}
The $h$-trunk has the property that
\begin{equation}\label{e758}
\lceil\mfu\rceil(h) \to \mfu,\mbox{as} \; h\searrow 0,
\end{equation}
see Proposition (\ref{P.approxtrunk}). Together with Theorem (\ref{p.delphic}) this allows to say that we 
can 
approximate any $\mfu \in \U$ in a natural way by um-spaces consisting of at most countably many points,
namely the $h$-trunks for $h \downarrow 0$.
\end{remark}
The following result tells us that the operation $ \tau(h) $ called $h$-truncation is continuous.

\begin{proposition}\label{p.tr.cont}
Let $t>0$. The mappings $\U \mapsto \U(h)^\sqcup,\, \mfu \mapsto \lfloor\mfu\rfloor(h) \mbox{and} \;(0,t]\mapsto \U(t)^\sqcup, h \mapsto \lfloor\mfu\rfloor(h)$ are continuous.
\end{proposition}

\begin{remark}\label{r.926}
 It is also possible to establish using the explicit definition of Gromov-Prokhorov distance that the first mapping is a non-expansive map. 
\end{remark}

Note that in ultrametric spaces two open balls are either contained in each other or disjoint. 
Hence it makes sense to speak of the number of open balls of radius $h$ in $\mfu$. By Theorem 
(\ref{p.delphic}) this number is unique. Since it is important we introduce a special notation for 
this number.

\begin{definition}[Number of $h$-balls]\label{D.anz}
 If $\mfu\in\bbU$ and let $I$ be the index set belonging to the decomposition of $\lfloor\mfu\rfloor(h)$ as given in Theorem (\ref{p.delphic}). Then we set $\#_h\mfu:=\# I$.
\end{definition}
The following is analog to Lemma 2.4(a) in\cite{EM14} and proved in Section \eqref{ss.proptt}.

\begin{proposition}[Measurability and additivity]\label{l.2.4EM}
 The number of open $2h$-balls $\anz_h$ is measurable. It is an additive functional on 
$(\U(h)^\sqcup,\sqcup^h)$, 
that is
 \begin{equation}\label{grx75}
  \anz_h(\mfu \sqcup \mfv) = \anz_h(\mfu) + \anz_h(\mfv) \, ,
 \end{equation}
 for all $\mfu,\, \mfv \in \U(h)^\sqcup$, where we interpret $\infty + a = a + \infty = \infty$ for 
$a 
\in \N_0 \cup \{\infty\}$.
\end{proposition}

\begin{remark}\label{r.952}
Note however that the map $\#_h$ is neither upper semi-continuous nor lower 
semi-continuous. This can be seen from the following counterexample. Take $\mfu_n = [\N, 
r(x,y) = \1(x \neq y), \delta_{\{1\}}+ \sum_{i\in \N} 
\frac{1}{n}2^{-i}\delta_{\{i\}}]$ and $\mfv_n = [\{a,b\}, r(a,b)=2h-n^{-1}, \delta_{\{a\}} + 
\delta_{\{b\}} ]$ for $n \in \N$.
\end{remark}

Another example of a measurable functional is what we call the {\em path of tops}, associating to every depth $h$ the $h$-subfamilies in a measure description, see Section (\ref{ss.pathsoftops}).
This functional describes the fragmentation of the tree in smaller sub-trees when we are getting closer to the top of the tree.

\smallskip

For a population and its genealogy a natural concept is the genealogy of a sub-population and its sub-genealogy which induces an order of these objects which is also the natural order association with the group structure. 
Indeed we can equip the space of $h$-forest with a {\em partial order} such that if $\mfu,\mfv$ are
$h$-forests and $\mfu$ is a sub-forest of $\mfv$, then $\mfu$ is smaller than $\mfv$. More formally:

\begin{definition}\label{d.po}
Define a partial order $\leq$ on $(\U(h)^\sqcup,\sqcup^h)$ by setting $\mfu \leq_h \mfv$ if 
there exists $\mfw \in \U(h)^\sqcup$ such that $\mfu \sqcup \mfw = \mfv$.

For $\mfu, \mfv \in \U$ we say $\mfu \leq_h \mfw$ if $\lfloor\mfu\rfloor(h) \leq_h \lfloor\mfv\rfloor(h)$ for the $h$-tops.
\end{definition}
We skip the index $h$ on $\leq_h$ if no ambiguities may occur.
Different interesting partial orders on metric measure spaces are developed in \cite{ggr_GeneralBranching} and \cite{GR16}, the latter less restrictive.

\subsection{Analytical tools for  \texorpdfstring{$h$}{h}-forests: polynomials and their truncation}
\label{ss.anatools}

The key objects to study the $h$-forests are {\em distance matrices}, monomials and {\em polynomials}, objects which we define next.
We begin with the finite sub-trees of size $m$ of an element $\mfu \in \U$.

\begin{definition}[Ultrametric distance matrices]
\label{D.umdm}~
\begin{enumerate} 
 \item Define the set of ultrametric distance matrices of order $m\geq2$ by
\be{rg4}
  \bbD_m:
  =\{(r_{ij})_{i\leq i<j\leq m}:r_{ij}\geq0\,,\,r_{ij}\leq r_{ik}\vee r_{kj}\;\,\forall 1\leq i<k<j\leq m\}\, \text{ and } \bbD_1 = \{0\}.
\ee
  \item For an ultrametric space $\mfu=[U,r,\mu]\in\bbU$ and $m\geq2$ we define the \emph{distance matrix map of order $m$}
\be{rg5}
  R^{m,(U,r)}:U^m\to\bbD_m\,,\quad (u_i)_{i=1,\dotsc,m}\mapsto (r(u_i,u_j))_{1\leq i<j\leq m}
\end{equation}
and the \emph{distance matrix measure of order $m$}
  \begin{align} \label{rg6}
  \nu^{m,\mfu}(\dx \underline{\underline{r}}) &:=\mu^{\otimes m}\circ (R^{m,(U,r)})^{-1}\\
  &=\mu^{\otimes m}(\{(u_1,\dotsc,u_m)\in U^m:(r(u_i,u_j))_{1\leq i<j\leq m}\in\dx \underline{\underline{r}}\})\,. \nonumber
\end{align}
For $m=1$ we set $\nu^{1,\mfu} := \bar{\mfu} := \mu(U)$ the \emph{total mass}.
\end{enumerate}
\end{definition}
Note that we cannot define here a nice distance matrix measure on $\bbD_\infty$ since $\mu$ need not be a probability measure.
For that we have to consider $(\bar u, \wh \nu)$, with $\wh \nu = R^{\infty,(U,r)} (\wh \mu)$
and $\wh \mu = (\bar \mu)^{-1} \mu$ and then the normalized measure allows to define
a distance matrix {\em distribution} on a sampling sequence which provides the full information
on the genealogy.

The finite sub-trees with $m$ leaves can be described by the following test functions.

\begin{definition}[Polynomials]
\label{D.polyn}
For $m\geq1$ and $\phi\in C_b(\bbD_m)$, define the \emph{monomial}
\begin{equation}\label{e:0606131101}
  \Phi=\Phi^{m,\phi}:\bbU\to\bbR\,,\quad \mfu\mapsto \langle\phi,\nu^{m,\mfu}\rangle = \int \nu^{m,\mfu}(\dx \dr)\, \phi(\dr)\,.
\end{equation}
The elements of the algebra generated by $\Pi$ are called \emph{polynomials}, 
the corresponding set $\mcA(\Pi)$.

We denote special classes of monomials for $h>0$ as follows:
\begin{equation}\label{rg8}
  \Pi_h := \{ \Phi^{m,\phi} \in \Pi:\, \supp (\phi) \subseteq [0,2h)^{\binom{m}{2}} \}, \Pi_+ := \{\Phi^{m,\phi} \in \Pi:\, \phi \geq 0\} \text{ and } \Pi_{h,+} = \Pi_h \cap \Pi_+ \, .
\end{equation}
and consider the generated algebras $\CA(\Pi_h),\CA_+(\Pi_h)$, where the latter are the polynomials which are positive on $\U$.
\end{definition}

\begin{remark}\label{r.1027}
Obviously $\Pi$ is closed under multiplication, hence 
$\mcA(\Pi)$ is the span of $\Pi$. 
\end{remark}

To characterize treetops we introduce:
\begin{definition}[Truncation]\label{d:trunc}
~
Let $m\in\bbN$ and $\phi:\bbD_m\to\bbR$.
Define the \emph{upper $h$-truncation} of $\phi$:
\begin{align}\label{rg9}
  \phi_{\hu}(\dr):&=\phi(\dr)\cdot\prod_{1\leq i<j\leq m} \mathbbm{1}_{[0,2h)}(r_{ij}),   
\end{align}
For the monomial $\Phi^{m,\phi} \in \Pi$ define
  \begin{align}\label{rg10}
  \Phi^{m,\phi}_{\hu}(\mfu):=\lan\phi_{\hu},\nu^{m,\mfu}\ran. 
\end{align}
This  extends to polynomials by linearity.
\end{definition}
Note that $\Phi_h$ is $\uparrow$-limit of $\Phi^n \in \CA_+(\Pi_h)$ if $\Phi \in \CA_+(\Pi)$, since we require the test functions in monomials to be continuous.
Truncated polynomials are the key in studying the semigroup 
$(\U(h)^\sqcup,\sqcup^h)$. We list some of the important properties in the next theorem.

\begin{theorem}[Properties of truncation] \label{p.trunc.poly}
 ~
 \begin{enumerate}
  \item\label{i.tr65} For any $\Phi \in \CA(\Pi)$, the mapping $\Phi_{\hu}$: $(\U(h)^\sqcup,\sqcup^h) \to (\R,+)$ is a semigroup homomorphism.
  \item\label{i.tr66} For any $\Phi \in \CA_+(\Pi)$, the mapping $\exp(-\Phi_{\hu}(\cdot))$: $(\U(h)^\sqcup,\sqcup^h) \to ([0,1],\cdot)$ is a semigroup homomorphism.
  \item\label{i.tr121} Elements of $\U(h)^\sqcup$ are identified by truncated monomials:  $\Phi(\mfu) = \Phi(\mfv)$ for all $\Phi \in \Pi_h$ implies that $\lfloor\mfu\rfloor(h) = \lfloor\mfv\rfloor(h)$.
  \item\label{i.tr50}  The topology of $\bbU(h)^\sqcup$ induced by $\bbU$ coincides with the initial topology of $\Pi_h$. That means $\mfu_n \to \mfu$ in $d_{\textup{GPr}}$ iff $\Phi(\mfu_n) \to \Phi(\mfu)$ for all $\Phi \in \CA(\Pi_h)$, i.e. $\CA(\Pi_h)$ is convergence determining.
  \end{enumerate}
\end{theorem}

\begin{proof}[Proof of Theorem~\ref{p.trunc.poly}]\label{pr.1059}
The proofs are given via three statements and their proofs, which can be found in Propositions (\ref{p2907131206}) for (a) and (b), (\ref{p.trpol.sep}) for (c), (\ref{p.trpol.topol}) for (d).
\end{proof}

Item \eqref{i.tr121} states that in order to identify an element in $\bbU(h)^\sqcup$ it suffices to know the distance matrices only at distances strictly less than $2h$, moreover, this even suffices to determine the topology of the space.

Theorem (\ref{p.trunc.poly}) implies  that non-negative, truncated monomials on 
$\bbU(h)^\sqcup$ are monotone with respect to the partial order introduced in Definition (\ref{d.po}). 

\begin{corollary}\label{c.1072}
 Let $\mfu,\mfv\in\bbU(h)^\sqcup$ and $\Phi\in \Pi_{h,+}$. Then $\mfu\leq \mfv$ implies $\Phi(\mfu)\leq 
\Phi(\mfv)$.
\end{corollary}

\subsection{Analytical tools for random  \texorpdfstring{$h$}{h}-forests: Laplace transforms}
\label{ss.anatoolsrandom}

We are now considering {\em random} elements in $\bbU$, which we generically denote by the capital letters $\mfU$, $\mfV$.
As in the classical settings of random measures 
or real-valued random variables, also for random trees the Laplace transform is a 
powerful tool.
In the abstract setting of the semigroup we say that given a (semi-)character $\chi: \bbU(h)^\sqcup \to [0,1]$ (or the complex numbers with modulus not greater than $1$) we define a characteristic function $\mcM_1(\U(h)^\sqcup) \to [0,1], \, \mu \mapsto \int \mu(\dx \mfu) \, \chi(\mfu)$.
In particular, they will play a role when we introduce infinitely divisible random trees.

\begin{definition}[Laplace transform]
\label{D.laplf}
  The \emph{Laplace functional} $L_\mfU : \CA_+(\Pi) \to [0,1]$ of a random um-space $\mfU$ is defined by
  \begin{equation}\label{e0606131100}
    L_\mfU(\Phi):=\bbE\left[\exp(- \Phi^{m,\phi} (\mfU) )\right]\,,\quad \Phi=\Phi^{m,\phi}\in \CA_+(\Pi)\,,
  \end{equation}
  the \emph{truncated Laplace functional} $L_\mfU : \mcA_+(\Pi_h) \to [0,1]$ is defined by restriction of the domain.
\end{definition}

The next result tells us that Laplace transforms on $\bbU(h)^\sqcup$ share an important property with Laplace transforms on $[0,\infty)$:
they well-define a probability measure on that space.

\begin{theorem}[Truncated Laplace transform]\label{p.trLap}
~
\begin{enumerate}
    \item\label{i1006132102a} Let $\mfU,\mfU'\in \U(h)^\sqcup$ be random $h$-forests. Then,
     \begin{equation}\label{rg68}
       \mfU\eqd\mfU'\quad\Longleftrightarrow\quad 
L_\mfU(\Phi)=L_{\mfU'}(\Phi)\;\;\forall \Phi\in \mcA_+(\Pi_h)\,.
     \end{equation}
    \item\label{i1006132102b} Let $\mfU,\mfU_n$, $n\in\bbN$, be random 
$h$-forests. Then,
    \begin{equation}\label{rg69}
     \mfU_n\underset{n\to\infty}{\Longrightarrow}\mfU\quad\Longleftrightarrow\quad 
L_{\mfU_n}(\Phi)\underset{n\to\infty}{\longrightarrow} L_\mfU(\Phi)\;\;\forall \Phi\in \mcA_+(\Pi_h) \,.
    \end{equation}
  \end{enumerate}
\end{theorem}

This result will be established in Section~(\ref{ss.lap}).
Note that we do require \emph{polynomials} in the previous result.
{\em It is open whether we could just use monomials above, at least in the
first claim, more in Remark~(\ref{r.tr1}).}

So, Laplace transforms of the truncated polynomials are a powerful tool for the analysis of the semigroup $(\U(h)^\sqcup,\sqcup)$.
Further properties showing on the importance of the Laplace transforms on semigroups are given in Section~5 of \cite{DMZ08}.

\subsection{Infinite Divisibility}
\label{ss.infdiv}

The next step is to identify the {\em random forests} where we can represent the $h$-tops and path of $h$-tops via {\em Poisson point processes} on $\U(h)^\sqcup$.
Note that we have here not a pure semigroup question, for which one has an abstract theory, compare Section \eqref{ss.compDMZ08}, but we have an $\R^+$ indexed {\em collection} of nested semigroups related via truncation maps for which we want to decompose our law. 
The key concept is therefor the following.

\begin{definition}[Infinite divisibility]\label{d2607131157}

Suppose $t>0$.  A random um-space $\mfU$ taking values in $\U$
which is not identically $0$ is called \emph{infinitely divisible} if for all $h>0$ (or $t$-infinitely divisible if for all
$h\in (0,t]$)
and $n \in \N$ we find i.i.d.~$\mfU_1^{(h,n)},\dotsc, \mfU_n^{(h,n)} \in \U(h)^\sqcup$, s.t.~the $h$-top of $\mfU$ is a concatenation of these random forests:
\be{r10a}
 \mfU(h) \eqd \mfU_1^{(h,n)} \sqcup \dots \sqcup \mfU_n^{(h,n)} \, .
\ee
\end{definition}
Note that by Theorem (\ref{p.trunc.poly}\ref{i.tr66}),~(\ref{p.trLap}) this is equivalent to saying that for all $h>0$ the Laplace functional of the $h$-treetop factorizes for every $n\in \N$:
 \be{r10b}
 \exists \mfU^{(h,n)} \in \mfU(h)^\sqcup \text{ with }  L_\mfU(\Phi) = \left(L_{\mfU^{(h,n)}}(\Phi)\right)^n , \quad \Phi \in \mcA(\Pi_{h,+}) \,. 
  \ee
Given an infinitely divisible $\mfU$ and observing only total population sizes we should get back to the classical concept of infinite divisibility of non-negative $\R$-valued random variables. Indeed the following is proved in Section (\ref{s.lkfproof}).

\begin{proposition}[Infinite divisibility of total mass]\label{p.mass:infdiv}
\hfill
\begin{enumerate}
 \item\label{mass.infdiv1} If $\mfU$ is infinitely divisible (or $t$-infinitely divisible for a 
$t>0$), the total mass $\bar{\mfU}$ is infinitely divisible in the classical sense.
 \item\label{mass.infdiv2} Conversely, let $X$ be an infinitely divisible random variable taking 
values in $\bbR_+$  such that its $\log$-Laplace transform has the form
 \begin{equation}\label{e1175}
 -\log \bbE \left[e^{- t X}\right]= \int_{(0,\infty)} \left(1-e^{-tx}\right)\,\nu(\dx x)\,,\quad t\geq0\,,
 \end{equation}
 where $\nu$ is the L\'{e}vy measure, that is, it is a $\sigma$-finite measure on $(0,\infty)$ 
such that $\int_{(0,\infty)}(1\wedge x)\,\nu(\dx x)<\infty$.
Then, for each $h>0$, there exists a random element $\mfU_h$ taking values in $\bbU(h)^\sqcup$ such 
that $\mfU_h$ is $h$-infinitely divisible and $\bar{\mfU}_h$ has the same distribution as $X$.
\end{enumerate}
\end{proposition}
\begin{remark}\label{r.1158}
The choice in (b) is not unique. A particular example is to take a star-tree with diameter $h$ with a measure $\mu$, namely $[\N,2h,\mu]$.
Here $\mu$ arises by taking the atoms of the Poisson point process on $(0,\infty)$ with intensity $\nu$ labeled bei $\N$ in decreasing size, this size giving the density of $\mu$ w.r.t. to counting measure.
\end{remark}

Note that only $(0,\infty)$-valued infinitely divisible random variables  
can occur as the total mass of an infinitely divisible random tree. 
Recall their Laplace-transform has no other part than the one we use in the r.h.w. of \eqref{e1175}. 
For the representation of the $\U$-valued state we will discuss this important 
aspect in greater detail in Section \eqref{ss.compDMZ08}. 

\begin{remark}
\label{R.idrm} 
The reader might wonder why not defining infinite divisibility using that $\nu^\mfu$ is a measure on distance matrices generating a sampling sequence using the approach as in Kallenbergs theory of random measures.
However we first note that if $ \mfU $ is not a random ultrametric {\em probability} space, we cannot define $ \nu^\mfU $.
This would be necessary to return to the setup of random measures and apply this classical theory.
However another possible definition would be to consider for a random $ \mfU $ the collection of random measures $ \{\nu^{\mfu,m}, m \in \N \}$ and to require infinite divisibility of the random measure $\nu^{\mfU,m}$ defined in \eqref{rg6} for {\em every} $m\in\N$ . 
 By equation \eqref{r10b} it turns out that the latter is implied by our definition but is not very convenient to work with since we can not define easily from their representation the needed PPP on $ \U $.
 Altogether this means the theory of random measures on $ (\R_+)^{\N \choose 2} $ cannot be used.
\end{remark}

\begin{remark}\label{r.1160}
Why does the definition of infinite divisibility include the parameter $t \in (0,\infty]$? Recall Remark (\ref{R.um.transform}) which tells us that we can define an equivalent ultrametric space $\tau^*(\mfU)$ which is then $\tau(t)$-infinitely divisible. However, we will see in Example (\ref{E2.1}) 
 that a restriction is in general necessary, i.e.~to restrict $h$ to some $(0,t]$ in \eqref{r10a}.
\end{remark}

\begin{example}[Compound Poisson forest (CPF)]\label{E2.1}
Fix $t>0$. Let $\theta>0$ and $\lambda\in\mcM_1(\bbU(t)^\sqcup \setminus \{0\})$. 
Let $M$ be Poiss$(\theta)$, i.e.~$M$ is a Poisson random variable with parameter $\theta$.
Let $\mfU_i$, $i\in\bbN$, be an i.i.d. sequence of random $t$-forests with $\mcL[\mfU_1]= \lambda$. 
Assume $(\mfU_i)_{i\in\bbN}\independent M$. Let
\be{rg12}
  \mfP_t:= \sideset{}{^{t}}\bigsqcup_{i=1}^M \mfU_i\,,
\end{equation}
be the canonical $t$-concatenation of $(\mfU_i)_{i\in\{1,\dotsc,M\}}$. We then refer to $\mfP$ as a 
\emph{compound Poisson $t$-forest} 
with parameters $\theta$ and $\lambda$, short, a $\textup{CPF}_t(\theta,\lambda)$.  Note if $\mfP$ is a 
$\textup{CPF}_t(\theta,\lambda)$, then $\mfP\in\bbU(t)^\sqcup$, that is, every $\textup{CPF}_t(\theta,\lambda)$ is a random $t$-forest. 
By construction CPF is infinitely divisible, since we can divide 
$M \eqd M_1 +\dots + M_n$ for $(M_i)_{1\leq i\leq n}$ i.i.d.~and Poiss$(\theta/n)$.
\end{example}

This is however not the only possibility, the general case is a {\em Poisson point process on forests} as our main result shows.
It generalizes as in the classical setting of infinite divisibility on $\R$ to limits of CPF's, i.e.~allowing in (\ref{rg12}) a countable concatenation of independent random elements which are not necessarily identically distributed.

Recall that on a Polish space $ E $, where we have defined  bounded sets together with a point infinitely far away $\mcM^\#(E)$ denotes the set of boundedly finite measures on $E$, which we will consider here for the space $E=\U(h)^\sqcup \setminus\{0\})$ with the point $0$ infinitely far away, see the 
discussion before Proposition~(\ref{p.tight-crit}).
Recall that for $ h=\infty $ we get $ E=\U\setminus \{0\} $.

\begin{theorem}[L\'evy-Khintchine representation of $ \U-$valued random variables]\label{T.LK}
An infinitely divisible random ultrametric measure space
$\mfU$ allows for a L\'evy-Khintchine representation of its Laplace functional; more precisely, there exists a unique $\lambda_\infty \in \mcM^\#(\U \setminus \{0\})$ with 
$\int (\bar{\mfu} \wedge 1) \lambda_\infty (\dx \mfu) < \infty$
such that for any $h \in (0,\infty)$:
\be{ag2}
-\log L_{\mfU}(\Phi_{\hu}) = \int_{\U(h)^\sqcup \setminus \{0\}} \left(1-e^{-\Phi_{\hu}(\mfu)} \right)\, \lambda_h (\dx \mfu) \quad \forall \, \Phi \in \Pi_+ \,, \ee
for
\begin{equation}\label{e.ar1}
  \lambda_{h}(\dx \mfu) = \int_{\U \setminus \{0\}} \lambda_\infty(\dx \mfv) \, \1(\lfloor\mfv\rfloor(h) \in \dx  \mfu)  \in \mcM^\#(\U(h)^\sqcup \setminus \{0\}) \, .
\end{equation}
If $\mfU$ is merely $t$-infinitely divisible, there is a unique $\lambda_t \in \mcM^\#(\U(t)^\sqcup \setminus \{0\})$ such that $\mfu \mapsto (\bar{\mfu}\wedge 1)$ is also integrable, \eqref{ag2} holds for $h\in (0,t]$ and \eqref{e.ar1} holds with $\lambda_t$ instead of $\lambda_\infty$ for $h\in (0,t]$. In either case, 
\be{ag2b}
\lambda_h (\U(h)^\sqcup \setminus \{0\}) = - \log \P(\bar \mfU =0) \in [0,\infty] \text{ for any } h .
\ee
We refer to $\lambda_h$ as the $h$-L\'{e}vy measure and to $\lambda_\infty$ as the L\'{e}vy 
measure. 
\end{theorem}
An interested question is to find the class of random variables for which the \Levy - measure is in fact concentrated on the trees, i.e. on $ \U(h) \setminus \{0\} $ rather than on the forests $ \U(h)^\sqcup $ as in the above statement.
This will be addressed in part II where the concept of Markov random trees is introduced for $ \U-$valued random variables.

Our goal was to decompose our tree in equally distributed independent pieces such that the
collection is decomposed in a consistent way. Indeed the relation (\ref{ag2}) says that the treetop $\lfloor \mfU(t) \rfloor$ can be seen as a Poisson number of $t$-forests. Then, \eqref{e.ar1} says that $\lfloor \mfU \rfloor(h)$ is a concatenation of the \emph{same} Poisson number of objects, now $h$-forests and these $h$-forests are simply $h$-truncations of the $t$-forests.
This is also well understood in the case of the CPF$_t(\theta,\lambda)$ where we have $\lambda_h (\dx \mfv) = \theta \int \lambda (\dx \mfu) \1(\mfu(h) \in \dx \mfv)$, see Proposition (\ref{p1308131621}).

In \cite{ggr_tvF14} we determine the \Levy-measure in the case of the $\U$-valued Feller diffusion explicitly and give various representations for the ingredients of the decomposition described above.

\begin{definition}[Treetop canonical measure]\label{D.canonic}

For $\mfU$  $t$-infinitely divisible, the measure $\lambda_t$ is called the \emph{treetop canonical measure} and  $h < t, \lambda_h$ is the canonical measure at \emph{depth} $h$.
\end{definition}

\begin{remark} The equation \eqref{e.ar1} already exhibits aspects of the path of treetop decompositions, however since we decompose here in forests but not in {\em trees} we will want to refine the analysis and go further beyond the information coded in the semigroup related to a fixed $ h $.
\end{remark}

The last theorem allows us to give a reformulation of infinite divisibility in the sense of a Poisson cluster representation:
\begin{corollary}[Poisson cluster representation]\label{c.PCR}
 Let $\mfU$ be infinitely divisible. Then for every $h>0$ there exists a Poisson point process $N^{\lambda_h}$ on $\U(h)^\sqcup$
with intensity measure $\lambda_h \in \CM^\# (\U(h)^\sqcup \setminus\{0\})$ such that 
\be{ag1c}
\lfloor \mfU \rfloor (\hu) \eqd \bigsqcup\limits_{\mfu \in N^{\lambda_h}} \mfu \, .
\ee
If $\mfU$ is $t$-infinitely divisible, then there exists a Poisson point process on
$\U(t)^\sqcup$ such that the $h$-truncations of the points form a Poisson point process
$N^h$ with L\'evy measure $\lambda_h$ with \eqref{ag1c}.
\end{corollary}

The role of the L\'evy measure is underlined by the following convergence criterion analogous to Theorem 13.14 in \cite{Kall03}, which we prove in Section \eqref{ss.further}.

\begin{theorem}[Convergence and infinite divisibility]\label{t:LimInfDiv}
 Let $h>0$. Assume $\mfU_m$ are $h$-infinitely divisible random trees and $\mfU_m\Longrightarrow \mfU$ where $ \lambda_h^{(m)} $ are the \Levy  measures for $ \mfU_m $. 
 Then $\mfU$ is $h$-infinitely divisible.
Moreover $\mfU_m\Longrightarrow\mfU$ iff $\lambda^{(m)}_h\Longrightarrow \lambda_h$ on 
$\mcM^\#(\bbU(h)^\sqcup)$.
\end{theorem}

There are various classes of infinitely divisible random trees, as this is the case in classical $ \R $-valued or even Banach-space valued random variables, which are characterized by different type of properties and which correspond to special forms of the L\'evy-measure, see here Section \eqref{ss.compDMZ08} for references.
Among these are in our case certain random trees, which arise in {\em branching processes}. 
We will discuss what the appropriate concepts are in the realm of random trees, i.e. random ultrametric measure spaces. We discuss this in the next subsections.

\subsection{Genealogies of branching processes and infinite divisibility}\label{ss.branchproc}

We give now a first idea how we can work with the collection of semigroups of \eqref{e780} to study {\em branching processes} on the level of genealogies and we will show that branching processes always have infinitely divisible marginals.
\begin{definition}[Concatenation and Branching Property]~\label{D.concbp}
 \begin{enumerate} 
 \item Let $h \geq 0$. For $P_1,P_2 \in \mcM_1(\U(h)^{\sqcup})$ let us define their $h$-\emph{convolution} $P_1 \ast^h P_2$ by
    \begin{equation}\label{e.tr222}
     P_1 \ast^h P_2 (A) = \int P_1(\dx \mfu_1) \int P_2(\dx \mfu_2) \1(\mfu_1 \sqcup^{h} \mfu_2 \in A), \quad A \in \mcB(\U).
    \end{equation}
 \item Now, let $(Q_t)_{t\geq 0}$ be a semigroup of probability kernels on $\U \times \mcB(\U)$.
    We say that the semigroup $(Q_t)_{t\geq0}$ has the \emph{branching property} if for all $h\geq 0$:
    \begin{equation}\label{e.tr223}
      Q_t(\mfu \sqcup^h \mfv, A) = Q_t(\mfu, \cdot) \ast^h  Q_t(\mfv,  \cdot) (A), \qquad A \in \mcB(\U(t+h)),\ \mfu, \mfv \in \U(h)^\sqcup,\ t\geq 0.
    \end{equation}
 \end{enumerate}
\end{definition}
The Markov process generated by a semigroup  and an initial value has the branching property if its semigroup has the branching property.

The Markov process describing the genealogy of the Feller continuum state branching diffusion is such a process, see Section~\ref{ss.exam}.

Clearly, the convolution defined above induces a semigroup structure on the 
probability measures on $\bbU(h)^\sqcup$. 
 As a first 
 application of the convolution we may formulate the following consequence of Theorem 
 (\ref{t:LimInfDiv}) which is already stated for groups in Theorem IV.4.1 of \cite{Par75}.

\begin{proposition}[Subsemigroup of infinitely divisible probability measures]\label{p.IV.4.1}
 The set of infinitely divisible probability measures is a closed sub-semigroup of all probability 
measures together with convolution $(\mcM_1(\U(h)^\sqcup), \ast^h)$.
\end{proposition}

It is a classical statement that continuous state branching processes have marginal distributions which are infinitely divisible distribution on $[0,\infty)$. 
We can derive the following result in our context.
\begin{theorem}[$ \U-$valued branching processes have infinitely divisible marginals]\label{THM:INFDIV:BRAN}\mbox{}\\
Let $h>0$ and $\pi\in\mcM_1(\bbU(h)^\sqcup)$ be $h$-infinitely divisible.
Suppose $(Q_t)_{t\geq0}$ is a semigroup which has the branching 
property.
Assume that $(\mfU_t)_{t\geq0}$ is the stochastic process induced by $(\pi Q_t)_{t\geq0}$.
Then $\mfU_t$ is $(t+h)$-infinitely divisible.
\end{theorem}

Another key feature of branching processes is that processes can be realized jointly for different initial values.
This feature is also conserved in the setting of $h$-forests as the next proposition shows.
Recall the partial order $\leq_h$ from Definition~(\ref{d.po}) and denote the corresponding stochastic order by $\preccurlyeq_h$.

\begin{proposition}[Joint realization of branching processes]\label{t:branchorder}
Let $\mfU=(\mfU_t)_{t\geq0}$, $\mfV=(\mfV_t)_{t\geq0}$ be branching processes with the same 
semigroup such that $\bbU(h)^\sqcup\ni \mfU_0=\mfu\leq_h  \mfv=\mfV_0\in\bbU(h)^\sqcup$. Then 
$\mfU_t\preccurlyeq_{h'} \mfV_t$ for all $h' \in (0,t+h)$ and $t\geq0$.
\end{proposition}

\subsection{Examples of $\U$-valued branching processes}\label{ss.exam}
Among $\R^+$-valued processes the \emph{continuous state branching processes} have infinitely divisible one-dimensional marginals starting in a fixed point.
Certainly the most prominent example of a continuous state branching process is the \emph{Feller diffusion} $(X_t)_{t \geq 0}$, the solution of $dX_t=\sqrt{bX_t} \; dW_t$ and $X_0=x \in [0,\infty)$, with $b>0$ and $(W_t)_{t \geq 0}$ standard Brownian motion.
The solution defines the Feller diffusion process where marginal distribution is infinitely divisible whose \Levy-measure is given via the density $x \to (t \frac{b}{2})^{-2} \exp(-\frac{x}{tb/2}) \in (0,\infty)$.
This diffusion is the many individuals-small mass-rapid branching limit of individual based binary critical branching in continuous time.

We obtain the \emph{genealogy} as the limit of the Galton-Watson genealogy both taken as $\U$-valued random variables.
For the Galton-Watson tree growing we can explicitly read of the ultrametric (twice the time back to the most recent common ancestor) and we take the counting measure on the leaves as sampling measure.
Then one passes to rapid-branching-small mass limit, see \cite{Gl12} for the proof of tightness and convergence.
Indeed we can define rigorously an \emph{$\U$-valued diffusion} by a martingale problem, which describes the genealogy of a population as equivalence class of ultrametric measure spaces \cite{ggr_tvF14} and is the limit of the $\U$-valued critical Galton-Watson process see \cite{DG18evolution} for a survey.

We recall the operator for the martingale problem of the $\U$-valued Feller diffusion from \cite{ggr_tvF14}. 
We need the concept of a polynomial to get the domain of the operator. 
Fix $ n \in \N $ and $ \phi \in C^1_b(\R^{(\substack{n\\2})}, \R)$. Then define for an equivalence class of an ultrametric measure space $ [U,r,\mu] $ the function

\begin{align}\label{e1309}
\Phi^{n,\phi}([U,r,\mu])= \int_{U^n} \phi((r(x_i,x_j)),\, 1 \leq i<j\in n) \; \mu(dx_1) \ldots \mu(dx_n)
\end{align}
and the action of the operator denoted $\Omega^\uparrow$ by
\begin{align}\label{e1305}
 \Omega^{\uparrow} \Phi^{n,\phi}(\mfu) = \Omega^{\uparrow,\mathrm{grow}} \Phi^{n,\phi}(\mfu) + 
\Omega^{\uparrow,\mathrm{bran}} \Phi^{n,\phi}(\mfu) 
\end{align}
and $\Omega^{\uparrow} \Phi^{n,\phi}(0) = 0.$ The operators on the r.h.s. are given by
\begin{align}\label{e1307}
 \Omega^{\uparrow,\textrm{grow}}\Phi^{n,\phi}(\mfu)  &=  \Phi^{n,2 \overline{\nabla} \phi} (\mfu) , 
\quad \overline{\nabla} \phi =  \suml_{1\leq i<j \leq n} \frac{\partial \phi}{\partial r_{i,j}},
\end{align}
\begin{align}\label{e1308}
 \Omega^{\uparrow,\textrm{bran}}\Phi^{n,\phi}(\mfu) & = an \Phi^{n,\phi}(\mfu) + \frac{b}{\bar{\mfu}} \sum_{1\leq k < l 
\leq n} \Phi^{n,\phi\circ \theta_{k,l}} (\mfu) ,
\end{align}
 where
\begin{equation}\label{eq:theta}
  \left( \theta_{k,l} (\dr) \right)_{i,j}  := r_{i,j}\1_{\{i\neq l, j\neq l\}} + r_{k,j} 
\1_{\{i=l\}} + r_{i,k} \1_{\{j=l\}} ,\quad 1\leq i < j\,.
\end{equation}
For $a=0$ we have the critical Feller diffusion.

Note that the {\em martingale problem} for $(\Omega^{\uparrow},\Pi(C_b^1))$, where $ \Pi(A) $ denotes the polynomials where $ \phi $ is chosen from $ A $, has a  \emph{unique} solution, see 
\cite{ggr_tvF14} and this solution has the \emph{branching property} and \emph{infinitely divisible} marginal distributions if this holds initially.
Therefore we have with this process the key example where the results of this paper apply.

This $\U$-valued process is studied in great detail in \cite{ggr_tvF14}, where its \emph{\Levy-measures} on $\U^\sqcup (h)$ are identified based on the present work and the \emph{branching property} of the process starting in a fixed element is shown via a new generator criterion in \cite{ggr_GeneralBranching}.
In fact in \cite{ggr_tvF14} we are able to derive an explicit representation of the \Levy-measure on $\U \setminus \{0\}$ using $\U_1$-valued coalescents and the $\R$-valued Feller diffusion.
Nice results can also obtained for the sub- and supercritical case as well as with adding immigration.
The reader finds also further material in the survey article \cite{DG18evolution}.

\subsection{Generalizations: discrete and marked setting}\label{ss.marked}

In this subsection we treat two other situations where genealogies are modeled with special versions or extensions of the space $ \U $.

A suitable concept of infinite divisibility is still important in particular for stochastic population models as genealogies of \emph{individual} based {\em Galton-Watson processes} which is a special case with its own features.
More precisely in order to describe genealogies of stochastically evolving populations it is also important to cover the case of {\em individual based} models, where we get population sizes which are {\em natural numbers} or a multiple of it.

On the other hand \emph{spatial} models as {\em super random walks} or just continuum state \emph{multitype} branching are important models requiring an extension of our approach.
If populations are {\em geographically structured} i.e.\ individuals have a location in a geographic space (some complete separable metric space $ \Omega $) then we have to replace $ \U $ by a more general object, similarly if individuals are of different \emph{types} from some set $K$ or we have both.

However the idea here in the spatial case is similar for populations where individuals carry a type or are at a geographic location.
We want to decompose for $h > 0$ the population in subpopulations which have a common ancestor at most time $h$ in the past. 
The new aspect is that now each of these subpopulations consists of individuals which are located in space or carry a type.
However this quality of individuals we will describe in the many individuals-small mass limit by a measure on the geographic space or the types (or both) which give the population size in a geographic set $A$ or a set of types in a set $B$ for $A,B$ varying in the measurable sets of geographic space or type space.
This means that our decomposition is as before in balls, but these subpopulations now have additional structure which however itself has \emph{no} impact on whether an individual belongs to a genealogically defined subpopulation or not.
Of course this has to be made rigorous.

\paragraph{Topology of the state space for spatial models}
In order to incorporate genealogies in {\em spatial} models it is necessary to generalize the concept of ultrametric measure spaces to {\em marked} ultrametric measure spaces.
How to do this has been developed in \cite{DGP11} and for infinite total population size in \cite{GSW}.
We recall the idea.

Consider a Polish mark space with a metric $ (V,r_V) $ which is \emph{fixed}.
We shall assume that:
\begin{equation}\label{e1356}
(V,r) \mbox{  is a topological group, with neutral element  } 0.
\end{equation}

The reader might think here for example of $V=\Z^d$.
Then let $ (U \times V,r,\nu) $ be the new object where $ (U,r) $ is a Polish space with specified metric $ r $ and $ \nu $ a finite Borel measure on $ \CB(U \times V) $.

Then define $ (U^\prime \times V, r^\prime, \nu^\prime) $ and $ (U \times V,r,\nu) $ as {\em equivalent} if there exists a map $ \wt \varphi $ from $ \supp \mu $, where $ \mu (\cdot)=\nu(\cdot \times V)$, to $ \supp \mu^\prime $ which is an {\em $ (r,r^\prime)-$isometry} such that $ \varphi: U \times V \to U^\prime \times V $ satisfies $ \varphi((u,v))=(\wt \varphi (u),v) $ and $ \varphi_\ast(\nu)=\nu^\prime $.
The {\em equivalence} class of $(U,r,\nu)$ is denoted
\begin{equation}\label{e1304}
[U \times V,r,\nu].
\end{equation}

The set of all such equivalences classes of \emph{$ V-$marked}  ultrametric measure spaces is denoted:
\begin{equation}\label{a1308}
\U^V.
\end{equation}
This set is endowed with the $ V-${\em marked Gromov weak topology} (see \cite{DGP11}), which makes $ \U^V $ a {\em Polish space}.
Namely the space is equipped with the marked Gromov-Prokhorov metric, see Section 1.2. in \cite{GSW} generating the topology.
The reader might think of this as follows.
Write $\nu=\bar \nu \; \wh \nu, \bar \nu=\nu(U \times V)$.
A sequence is converging if first the $\bar \nu_n$ converge and second  the space generated by the samples taken i.i.d. with $\wh \nu$ converges as a finite marked metric space for all sample-sizes, drop the latter condition if $\bar \nu_n \to 0$.
This concept allows also to consider boundedly finite measures on $ \U^V$, which in turn allows to define generalized Poisson point processes on $ \U^V$ see Section 2.4. in \cite{DVJ03}, which we need for the \Levy-Khintchine representation.
In that setup we replace $ \U \setminus \{0\} $ used here so far by $ \U^V \setminus \{0\} $ where now the $ 0 $ space is defined similar as before as $ [\{(0,0^V)\}, \underline{\underline{0}}, \underline{0}] $ with $0^V$ we denote the point $0$ in $V$ a distinguished point in $V$.

\begin{remark}\label{r.1423}
Here we have to address the choice of definition for the isomorphy classes of marked ultrametric measure spaces.
If we have a population which is concentrated on a closed subset $V^\prime$ of the space $V$, the question is whether this is an element of $\U^{V^\prime}$ which we have to distinguish from the element of $\U^V$ where the support of $\nu$ is in $U \times V^\prime$.

Another choice to define the isomorphy $\varphi$ would be to say
\begin{align}
&\varphi: supp (\nu) \to supp (\nu^\prime),& \label{a1380a}\\
&\varphi \left((i,v)\right)= \left(\wt \varphi (i),v \right) \;, \; \forall \; (i,v) \in supp (\nu) & \label{a1380b} \\
&\varphi \left(r(i,i^\prime)\right)=r \left(\wt \varphi(i),\wt \varphi(i^\prime)\right), \; \forall \; i,i^\prime \in supp (\mu) & \label{a1380c} \\
&\varphi_\ast \nu = \nu^\prime.& \label{a1380d}
\end{align}
The former choice is the one usually taken that is requiring, i.e. in \eqref{a1380b} the equation to hold for all $i \in supp(\mu)$ and \emph{all} $v \in V$.
However note that the latter choice leads to taking a quotient w.r.t. to a certain subspace (closed) and we simply can work with the quotient topology.
However if one wants to think about random genealogies it is \emph{not} convenient to work with this concept of isomorphy.
\end{remark}

We define again a (distance matrix, mark)-measure on $ (\R)^{n \choose 2} \times V^n $ and its Borel- $ \sigma $-algebra as push forward corresponding to the map
\begin{equation}\label{e1315}
R_n:(U \times V)^n \longrightarrow \left(\big(r(u_i,u_j)\big)_{1 \leq i < j \leq n}, (v_i)_{1 \leq i \leq n}\right).
\end{equation}
This measure is denoted $ \nu^{\mfu,n} $.

The polynomials take now the form
\begin{equation}\label{e1321}
\Phi(\mfu)= \intl_{(U \times V)^n} \; \nu^{\mfu,n} (d(u_1,v_1) \ldots d((u_n,v_n))  \; \varphi \left( (r(u_i,u_j))_{1 \leq i < j \leq n}) \right) g(v_1, \ldots, v_n),
\end{equation}
where $ \varphi \in C_b(\R^{n \choose 2}, \R), g \in C_b(V^n,\R)$ for some $n \in \N$.
The algebra generated by these monomials is separating, see (\cite{GSW}).
Based on these polynomials on $\U^V$ we define the Laplace transform again via \eqref{e0606131100} and we use the same notation.

In the sequel we will choose as mark space $ V $ the geographic space $ \Omega $ for example $ \Omega=\Z^d $ or $ \R^d $.
In that case of a mark space (different from the usual multitype situation where $ V $ may be a finite set) it is necessary to allow also {\em infinite} measures $ \nu $ in $ [U \times V,r,\nu] $.
However then one has to restrict to \emph{boundedly finite measures}.
Namely if $ \Omega $ can be obtained as $ \Omega_n \uparrow \Omega $ with $ \Omega_n \subseteq \Omega $ and $ \Omega_n $ finite or bounded.
Take for example $ \Omega = \Z^d $ or $ \R^d $.
Then we require that $ \nu \mid_{U \times A} $ is {\em finite} for every $ A $ finite (bounded).
These $ \Omega_n $ are chosen for example as $ [-n,n]^d \cap \Z^d $ in the case of $ \Omega=\Z^d $ or more generally on a Polish space with fixed metric balls around a fixed point.
Then the equivalence classes are defined by requiring that all \emph{restrictions to the sub-populations} $ U \times \Omega_n $ are equivalent in the sense defined above \eqref{e1304}.
Then we obtain still a Polish space $ \U^V $ (for $ V=\Omega $) introducing the \emph{$\Omega$-marked Gromov weak$^\#$ topology}, see \cite{GSW} for the details.
Roughly: we define the topology by defining it again by the convergence in \eqref{ad.3488} just using now the polynomials for the marked case with $ g $ having a \emph{bounded} support.

\paragraph{The semigroup structure of the state spaces}
Next we have to introduce the semigroup structure for the \emph{discrete} and the \emph{spatial} case.

The \emph{discrete semigroups} are a sub-semigroup of $h$-forests $\U(h)^\sqcup$, consisting of those with integer-valued measures (or multiples of those).
The binary operation is the concatenation $\sqcup^h$ from \eqref{e.tr47}.

The \emph{marked $h$-forests} are sub-semigroups of marked $h$-forests consisting of those marked um-spaces with genealogical distance bounded by $2h$.
For the marked setting we define the binary operation, the concatenation as follows.
Denote $\wt \mu,\wt \nu$ as the extensions from $U$ resp. $W$ to $U \uplus W$, then set
\begin{equation}\label{e762}
 [U,r_U,\mu] \sqcup_V^h [W,r_W,\nu] = [U \uplus W, r_U \sqcup^h r_W, \wt \mu + \wt \nu], \mbox{  with  } \ [U,r_U,\mu], [W,r_W,\nu] \in \U^V(h)^\sqcup .
\end{equation}
We see in particular that we just lift the operation of concatenation on $(U,r)$ to $(U \times V,r \otimes r_V)$ and the addition of measures from $\CB(U)$ to $\CB(U \times V)$.
Note that measures on a space form a topological semigroup, the genealogical part does as well as we saw.
Therefore the operation on $\U^V$ inherit much of the structure and this is easily seen.

\begin{definition}[Marked ultrametric spaces, discrete spaces]\label{D.markspace}\mbox{}\\
Let $h\geq 0$.
 \begin{enumerate}
  \item Let $a>0$. Define the set
  \begin{align} \label{e760}
&& \U(h,a)^\sqcup = \left\{ \left[ U = \{1,\dotsc,n\},r', a \sum_{i\in U} \delta_i \right] \in \U(h)^\sqcup \mid n \in \N_0 \text{ and } \right.\\
&& \left. r' \text{ a pseudo-ultrametric on } U \right\} \nonumber
\end{align}   

  and call $\U(h,1)^\sqcup$  the set of \emph{discrete $h$-forests}.
  \item Let $V$ be a Polish space.
  The set of \emph{marked $h$-forests} is defined as
  \begin{equation}\label{e761}
   \U^V(h)^\sqcup = \left\{ \mfu \in \U^V \mid \nu^{2,\mfu}\left(  (2h,\infty) \times V^2\right) = 0 \right\} .
  \end{equation}
 \end{enumerate}
 Combinations of the two definitions in $\U^V(h,a)^\sqcup$ are defined analogously.
\end{definition}

Again we can define for $ h \in (0,h^\prime) $ {\em truncation maps} $ \tau_V(h): \U^V(h^\prime)^\sqcup \to \U^V(h) $ by just acting with $ \tau(h^\prime) $ on $ (U,r)$.
Furthermore both sets form semigroups.

We then obtain with our setup:
\begin{proposition}[Semigroup properties]\label{p.tr2}
\hfill\\
\begin{itemize}
\item[(a)] The adapted Theorem (\ref{p.delphic}) holds for the two  semigroups above. 
In particular it is factorial for $h>0$ (i.e. we have \eqref{e.tr49}).

Furthermore we have for $h\geq 0$:
\item[(b)] Let $a>0$. Then $\left(  \U(h,a)^\sqcup, \sqcup^h\right)$ is a closed sub-semigroup of $(\U(h)^\sqcup,\sqcup^h)$.
  
\item[(c)] The set $( \U^V(h)^\sqcup, \sqcup_V^h)$ forms a topological semigroup, i.e. $\sqcup^h$ is continuous as function of two variables.
 \end{itemize}
\end{proposition}
Note that (c) follows from Theorem \eqref{p.delphic} as well the second part of (b) while closedness is straightforward.
Therefore we will later in Section~\ref{ss.proext} have to verify essentially (a).

\begin{remark}\label{r.1527}
We note two facts. The $h$-subfamily decomposition means now (as we can deduce having \emph{proved} uniqueness of the decomposition) that the $h$-balls in which we decompose now lead to $[U_i \times V, r \mid_{U_i},\nu_i]$ with $\nu_i= \mu \mid_{U_i} \otimes \kappa$ where $\kappa$ is the transition kernel from $U$ to $V$ which can be obtained writing $\nu=\mu \otimes \kappa$ where $\mu(\cdot)=\nu(\cdot \times V)$.
The $[U_i \times V, r \mid_{U_i},\nu_i],i=1,\cdots,$ is now the $h$-family decomposition.
If we project the $\nu_i$ on $V$ we obtain $\lambda_i$, measures on $(V,\mcB(V))$  which form the decomposition of $\lambda=\suml_{i \in I} \lambda_i$ which is the decomposition of the random measure on $V$ induced by $\mfU$, which is what is treated in the Kallenberg \cite{Kall83} on random measures and his concept of infinitely divisible random measures, gives the induced decomposition once we project our decomposition on $V$.
\end{remark}

\begin{remark}\label{r.1414}
 For $h=0$, the semigroup $\U^V(h)^\sqcup$ equals the set of non-negative measures on $V$, which indeed is a semigroup.
\end{remark}

By Proposition (\ref{p.tr2}), for the discrete semigroup $\left(  \U(h,a)^\sqcup, \sqcup^h\right)$ being a closed sub-semigroup of $h$-forests \emph{all} of the results obtained on $ \U(h)^\sqcup $ hold.

Finally the generalizations involve the set of truncated (marked) polynomials, see Definition 3.7 in \cite{DGP13} and \cite{GSW} for more detail, which we denote:
\begin{equation}\label{e763}
 \Pi_h^V = \{ \Phi^{m,\varphi g} \in \Pi^V \mid \Phi^{m,\varphi} \in \Pi_h \} .
\end{equation}
For the \Levy - measures we have now measures on $ \U^V \setminus \{0\} $.
Then we can generalize our results.
\emph{All} results above hold once we make the indicated changes in the statements:
\begin{theorem}[Results in the marked setting]\label{T.resmark}
 For the marked $h$-forests $( \U^V(h)^\sqcup, \sqcup_V^h)$ and the truncation $ \tau_V(h^\prime), h^\prime \in [0,h) $ the following results reworded as indicated above hold: Theorem (\ref{p.trunc.poly}), Theorem (\ref{p.trLap}), Theorem (\ref{T.LK}), Theorem (\ref{t:LimInfDiv}) and Theorem (\ref{THM:INFDIV:BRAN}).
\end{theorem}
We can apply this to the states of the genealogies of the {\em  super random walk} the spatial version of the Feller diffusion which is a well known measure valued process (see \cite{D93}).
More precisely we talk about the genealogy (modeled as $\U^V$-valued process) of the Markov process $X(t)=(x_i(t))_{i \in \Omega}$ for $\Omega$ countable abelian group given by the SSDE
\begin{equation}\label{e1558}
dx_i(t)=\sum a_{i,j} (x_j(t)-x_i(t))dt + \sqrt{bx_i(t)} \; dw_i(t), \; i \in V,
\end{equation}
with $(w_i(t))_{t \geq 0}$ i.i.d. standard brownian motions, $a$ is a homogeneous summable transition matrix on $\Omega \times \Omega$ and $b>0$.
The corresponding $\U^\Omega$-valued process of genealogies is treated in \cite{ggr_tvF14} and \cite{ggr_GeneralBranching} in great detail.
However once we can construct the latter then we can apply the theorem to the above object.

\subsection{Discussion: Relation to negative definite functions  \texorpdfstring{\cite{DMZ08}}{DMZ08}}
\label{ss.compDMZ08}

We have here a collection of semigroups in $ h $ which are all consistent with respect to the \emph{additional operation} of truncation maps (besides concatenation and scalar multiplication) which gives the interesting features given here as well as those worked out in \cite{ggr_GeneralBranching}.
Nevertheless we can focus on a particular $ h $ and see what one can get from abstract theory for that object alone. That means we now relate our set up to the general theory of
 characters and semigroups, i.e. to {\em harmonic analysis}. 

The semigroup $(\bbK,+) := (\U(h)^\sqcup, \sqcup^h)$ can be seen as an example of a convex cone as defined in Section~2 of \cite{DMZ08} (here their relation 2.5. does not hold).
The multiplication by positive scalars there works in the way: $a [U,r,\mu] = [U,r,a\mu]$ for $a \in (0,\infty)$ and $[U,r,\mu] \in \bbK$.
The origin coincides with the neutral element $0$ and $\bbK$ is a normed cone w.r.t.~the Gromov-Prokhorov metric $\dGPr$.
We have shown in Theorem (\ref{p.trunc.poly}) that the semigroup $(\bbK,+)$ possesses a strictly separating  class of homomorphisms.

The Laplace-transform is an important tool for the analysis of random elements taking values in $\bbK$ as presented in Section 5 of \cite{DMZ08}.
In particular, as already known by classical results on positive definite functions (\cite{BCR84}), infinitely divisible random elements allow a representation of the Laplace functional via a \Levy-Khintchine formula.
Our formula \eqref{ag2}, however, has some special features in comparison to general \Levy-Khintchine formulas.
We list these features below in a list after having explained briefly the relation to positive definite functions.

For an infinitely divisible random element in $\U(h)^\sqcup$ one can check that the map
\begin{equation}\label{e.tr84}
 \begin{cases}
 \E\Pi_h:\{\exp(-\Phi(\cdot)): \Phi \in \Pi_h\} & \to [0,\infty) \\
  \exp(-\Phi(\cdot)) & \mapsto -\log \E[ \exp(-\Phi(\mfU))] ,
 \end{cases}
\end{equation}
on the semigroup $\E\Pi_h$ is {\em negative definite}, see Section 5.2 of \cite{DMZ08}.
Note that $\E\Pi_h$ is a subset of the semigroup homomorphisms from $\bbK$ to $[0,1]$, denoted by $\tilde{K}$ in that reference.
Then Theorem 4.3.19 in \cite{BCR84} establishes the existence of a type of \Levy-Khintchine formula for the map defined in \eqref{e.tr84} compare (6.5) in Section 6.1. of \cite{DMZ08}.

What is the relation to our setup and our \Levy-Khintchine formula \eqref{ag2} to the one obtained for fixed $ h $?
\begin{itemize}
\item In general the \Levy-measure will be a measure on the {\em bidual} space of $\bbK$, see Section 7.2 of \cite{DMZ08}.
  In our Theorem (\ref{T.LK}), we see that it is actually supported on $\bbK$ (more precisely $\iota(\bbK)$ if $\iota$ denotes the injection of a space into its bidual).
To us it is not clear how to get this from abstract grounds in our case.
 \item There is no quadratic form part.
      This comes from the fact that $\E\Pi_h$ has no involution, except the identity, see Theorem 4.3.20 in \cite{BCR84} or (6.6) in \cite{DMZ08}.
 \item There is no linear term. We have no {\em easy} explanation for that, recall we deal with ultrametric \emph{measure} spaces here. The result follows from the fact that in \eqref{e.LK:1} we can show that $\pi_h=0$.
 Another argument would be to use an analogous result to Lemma 5.8 of \cite{EM14} which states that all non-decreasing continuous functions are constant and their application in their Section 9.
 \end{itemize}

\subsection{Outline proof section}
\label{ss.outps}
The proofs are presented in three sections and an appendix. We prepare in Sections (\ref{s.tt}),(\ref{s.proofsprobm}) the ground by establishing first the results on properties of the state space and the semigroups structure and then the results on the properties of probability laws on these structures. The Section (\ref{s.lkfproof}) contains the proofs of the main results on the infinite divisibility. In the appendix more technical points are collected.

\section{Proofs of statespace description and semigroup results}\label{s.tt}

We collect here in five subsections the main technical ingredients for the proofs of our theorems, topological basics 
concerning our
state spaces and key objects of tree description, trunks, evaluations of
polynomials reading off tree tops only (Section (\ref{ss.proptt})-(\ref{ss.pot})) and  the algebraic structure of our subfamily decomposition (Section (\ref{ss.concsemi}), (\ref{ss.uniqfact})).
We work here with polynomials rather than the metric structure.


\subsection{Concatenation semigroup: Proof of Theorem \ref{p.delphic}}\label{ss.concsemi}

In this subsection we will 
follow \cite{EM14} and leave out some of the proofs since this reference provides elaborated proofs 
on rather similar statements with minor modifications.

Before we start this section we give a quick remark on \cite{EM14}.
\begin{remark}\label{r.2107}
In Definition (\ref{D.concat}) we defined a one-parameter family of semigroups
\begin{equation}\label{e1699}
(\U(h)^\sqcup, \sqcup^h)_{h >0}.
\end{equation}
Another binary operation $ \boxplus $, leading to a very different tree, is defined in \cite{EM14}.
 For $g>0$ let $\bbM_g = \{ \mfx \in \bbM:\, \bar{\mfx} = g\}$.
 For $\mfx = [X,r_X,\mu_X],\, \mfy =[Y,r_Y,\mu_Y] \in \bbM_g$ define
 \begin{equation}\label{e2112}
  \mfx \boxplus_g \mfy = [X \times Y, r_X \oplus r_Y, \frac{1}{g} \mu_X \otimes \mu_Y ] \in \bbM_g \, .
 \end{equation}
 The operation $\oplus$ on the metric was explained in \cite{EM14}.
 The operation $\boxplus_g$ coincides with the definition $\boxplus$ when $g=1$ and it is easy to see that this defines a semigroup isomorphic to $(\bbM_1,\boxplus_1)$.
 One may also augment the space and consider $\bbM_{\leq g} =  \{ \mfx \in \bbM:\, \bar{\mfx} \leq g\}$. Even though we get very different trees from this operation, algebraic properties are very similar.
\end{remark}
\begin{proposition}
[Topological properties of $h$-forests]\label{l.frsts.clsd}
 The subset $\U(h)^\sqcup \subset \U$ is closed in the Gromov-weak topology. Its subset $\U([0,2h))$ 
is a $G_\delta$ subset of $\U$ which is dense in $\U(h)^\sqcup$.
\end{proposition}
\begin{proof}\label{p.2124}
$\U(h)^\sqcup$: Suppose $(\mfu_n)_{n\in \N} \subset \U(h)^\sqcup$ and $\mfu_n \to \mfu$ in the 
Gromov-weak topology. This includes saying that $\nu^{2,\mfu_n} \Rightarrow \nu^{2,\mfu}$ weakly as 
measures on $[0,\infty)$ as $n\to \infty$; here $\nu^{2,\cdot}$ are the distance matrix measures 
introduced in \eqref{rg6}. But  $\nu^{2,\mfu_n}((2h,\infty)) = 0$ for all $n \in \N$, $(2h,\infty)$ 
is an open subset of $[0,\infty)$ and thus the Portmanteau theorem says that then also 
$\nu^{2,\mfu}((2h,\infty)) = 0$.

$\U(h)$: this is proven much like Proposition 5.1 in \cite{EM14}. The mapping $\nu^{2,\cdot}: \, \U 
\to \mcM_f([0,\infty)), \, \mfu \mapsto \nu^{2,\mfu}$ is continuous. Let $A \subset [0,\infty)$ be 
closed (in $[0,\infty)$). By the Portmanteau theorem the mapping $\mcM_1([0,\infty)) \to 
[0,\infty),\, \nu \mapsto \nu(A)$ is upper semi-continuous. Thus the mapping $\U \mapsto 
[0,\infty),\, \mfu \mapsto \nu^{2,\mfu}([2h,\infty))$ is upper semi-continuous and the set $\{ \mfu 
:\, \nu^{2,\mfu}([2h,\infty)) < c \}$ is open for any $c >0$. We write $\U(h) = \bigcap_{c >0} \{ 
\mfu :\, \nu^{2,\mfu}([2h,\infty)) < c \} $ and note that this is a countable intersection of open 
sets, i.e.~$G_\delta$.
 
 For any $\mfu \in \U(h)^\sqcup$ it is true that $\mfu(h-n^{-1}) \rightarrow \mfu$ in 
$d_{\textrm{GPr}}$, since $\Phi(\mfu(h-n^{-1})) \to \Phi(\mfu)$ as $n\to \infty$ for any $\Phi \in 
\Pi$. Moreover $\mfu(h-n^{-1}) \in \U([0,2h-2n^{-1}]) \subset \U(h)$ for any $n \in \N$..
\end{proof}

\begin{lemma}[Lemma 2.1 in \cite{EM14}]\label{l.2.1EM} ~
\begin{enumerate}
 \item\label{i.tr11} The operation $\sqcup^h: \U(h)^\sqcup \times \U(h)^\sqcup \to \U(h)^\sqcup$ is 
continuous. Moreover,
 \begin{equation}\label{grx84}
  d_{\text{GPr}}(\mfu_1 \sqcup \mfu_2, \mfu_1' \sqcup \mfu_2') \leq d_{\text{GPr}}(\mfu_1  ,\mfu_1') 
+ d_{\text{GPr}}(\mfu_2,\mfu_2') \, ,
 \end{equation}
 when $\mfu_1,\mfu_1',\mfu_2,\mfu_2' \in \U(h)^\sqcup$.
 \item\label{i.tr12} The metric $d_{\text{GPr}}$ is translation invariant w.r.t.~$\sqcup$.
 \item\label{i.tr13} For $\mfu,\mfv,\mfw_1,\mfw_2 \in \U(h)^\sqcup$ we have
 \begin{equation}\label{grx85}
  d_{\text{GPr}} (\mfu,\mfv) \leq d_{\text{GPr}}(\mfu \sqcup \mfw_1, \mfv \sqcup \mfw_2) + 
d_{\text{GPr}}(\mfw_1, \mfw_2) \, .
 \end{equation}
\end{enumerate}
\end{lemma}
\begin{proof}\label{p.2163}
 This result is established as Lemmas 2.1 and 2.2 in \cite{EM14}.
\end{proof}

\begin{lemma}\label{l.2168}
 $(\U(h)^\sqcup, \sqcup^h)$ is a topological semigroup which is commutative and $0$ is the neutral 
element.
\end{lemma}
\begin{proof}\label{p.2172}
  It is elementary to show that the binary continuous operation $\sqcup^{h}$ on $\U(h)^\sqcup$ defines a 
semigroup with the neutral element $0$. Likewise commutativity is obvious.
\end{proof}

Recall the partial order from Definition (\ref{d.po}) and the modulus of mass distribution $v_\delta(\cdot,h)$ from \cite{GPW09}:
\begin{equation}\label{e2113}
 v_\delta(\mfu,h) = \int \mu^\mfu (\dx x) \1\left( \mu^\mfu (B_{2h}(x)) < \delta \right) \, .
\end{equation}
Then we can show the following lemma.
\begin{lemma}[Lemma 2.7 in \cite{EM14}]\label{l.2.7EM}~
\begin{enumerate}
 \item\label{i.2.7EM.a} $\nu^{2,\mfu \sqcup \mfv} = \nu^{2,\mfu} + \nu^{2,\mfv} + 2 
\bar{\mfu}\bar{\mfv} \delta_{2h}$, $u,v\in{\U(h)}^\sqcup$.
 \item\label{i.2.7EM.b} For $h'\leq h$: $v_\delta(\mfu \sqcup \mfv, h') = v_\delta(\mfu,h') + 
v_\delta(\mfv,h')$,  $u,v\in{\U(h)}^\sqcup$.
 \item\label{i.2.7EM.c} For $h'<h$, $\mfu \leq \mfv \in \U(h)^\sqcup$: $v_\delta(\mfu, h') \leq 
v_\delta(\mfv,h')$,  $u,v\in{\U(h)}^\sqcup$.
 \item\label{i.2.7EM.d} For any compact set $A \subset \U(h)^\sqcup$, the set $\bigcup_{\mfu \in A} \{ 
\mfv : \, \mfv \leq \mfu \}$ is compact. Additionally, $\{(\mfu,\mfv) \in (\U(h)^\sqcup)^2:\, \mfu \leq 
\mfv \in A\}$ is compact.
 \item\label{i.2.7EM.e} The mapping $K$ from $\U(h)^\sqcup$ to the compact subsets of $U([0,2h])$ 
defined by $K(\mfu) = \{\mfv: \, \mfv \leq \mfu\}$ is upper semi-continuous, i.e.~if $F\subset 
\U(h)^\sqcup$ closed, then $\{\mfu:\, F\cap K(\mfu) \neq \emptyset \}$ is closed. Equivalently, if 
$\mfu_n \to \mfu$ and $\mfv_n \in K(\mfu_n)$ converges to $\mfv$, then $\mfv \in K(\mfu)$.
\end{enumerate}
\end{lemma}
\begin{proof}\label{p.2199}
 The first claim is obvious. For \eqref{i.2.7EM.b} set $\mfu = [U,r^\mfu,\mu^\mfu]$, $\mfv = 
[V,r^\mfv,\mu^\mfv]$ and $\mfu \sqcup \mfv = [U\uplus V, r, \mu^{\mfu}+\mu^{\mfv}]$. For $x \in U$ 
it is clear that $\{ y \in U \uplus V:\, r(x,y) < 2h'\} \subset U$ and therefore 
$(\mu^\mfu+\mu^\mfv)(B_{2h'}^{(r)}(x)) = \mu^\mfu (B_{2h'}^{(r^\mfu)}(x))$. Likewise for $z \in V$ 
and therefore:
 \begin{align}\label{grx85b}
  v_\delta(\mfu \sqcup \mfv, h') & = \int (\mu^\mfu + \mu^\mfv)(\dx x) \1\left( (\mu^\mfu+\mu^\mfv) 
(B_{2h'}(x)) < \delta \right)  \\
  & = \int_U  \mu^\mfu (\dx x) \1\left( (\mu^\mfu+\mu^\mfv) (B_{2h'}(x)) < \delta \right)   \nonumber \\ 
  & + \int_V  
\mu^\mfv(\dx z) \1\left( (\mu^\mfu+\mu^\mfv) (B_{2h'}(z)) < \delta \right)  \nonumber \\
  & = v_\delta(\mfu) + v_\delta(\mfv) \, .\nonumber
 \end{align}
 \eqref{i.2.7EM.c} is a trivial consequence of \eqref{i.2.7EM.b}.
 
To show \eqref{i.2.7EM.d} see the following: Since $A$ is compact, we know by Proposition 
(\ref{p.pre-comp}) that for all $h>0$, $\eps >0$ there is a $\delta(h,\eps)>0$ s.t.
  \begin{equation}\label{grx86}
   \sup_{\mfv \in A} v_{\delta}(\mfv,h) < \eps \, .
  \end{equation}
 But, if $\mfu \leq \mfv$, then $v_\delta(\mfu,h) \leq v_\delta(\mfv,h)$ for $h>0$ by the Lemma's part \eqref{i.2.7EM.c}. Thus,
  \begin{equation}\label{grx87}
   \sup_{\mfu \in \{\mfu:\,\exists\, \mfv \in A:\,\mfu\leq \mfv \}} v_\delta(\mfu,h) \leq \sup_{\mfv 
\in A} v_{\delta}(\mfv,h) < \eps \, \text{  for all  } h > 0.
  \end{equation}
However it suffices by Remark (\ref{R.pre-comp}) to show the above for all $h > 0$ to continue.
Namely we can also establish 
the closedness. Suppose $(\mfu_n)_{n\in \N} \subset \bigcup_{\mfv \in A} \{\mfu: \, \mfu \leq 
\mfv\}$ and $\mfu_n \to \mfu_\infty$. For any $n \in \N$ there are $\mfv_n \in A$ and $\mfw_n \in 
\bigcup_{\mfv \in A} \{\mfu: \, \mfu \leq \mfv\}$ with $\mfu_n \sqcup \mfw_n = \mfv_n$. This allows 
to deduce sub-sequential limits $\mfv_n \to \mfv_\infty \in A$ (by compactness) and $\mfw_n \to 
\mfw_\infty \in \U(h)^\sqcup$ (by pre-compactness). By continuity of $\sqcup$ (in Lemma (\ref{l.2.1EM}) \eqref{i.tr11}) we deduce, $\mfu_\infty 
\sqcup \mfw_\infty = \mfv_\infty$. Thus $\mfu_\infty \in \bigcup_{\mfv \in A}\{\mfu:\, \mfu \leq 
\mfv\}$. Similar arguments lead to the second statement.
 \eqref{i.2.7EM.e} is a consequence of \eqref{i.2.7EM.d}'s second statement.
\end{proof}

An element $\mfu \in \U(h)^\sqcup$ is called \emph{irreducible} if $\mfu \neq 0$ and $\mfv \leq \mfu$ 
for $\mfv \in \U(h)^\sqcup$ implies that $\mfv$ is either $0$ or $\mfu$. We characterize the set of 
irreducible elements; that it is a measurable subset of $\U(h)^\sqcup$ was provided in Lemma 
(\ref{l.frsts.clsd}). This is analogous to Proposition 5.1 in \cite{EM14}.
\begin{lemma}[Trees are the irreducible elements]\label{p.5.1EM}
 The set $\U(h)$ is the set of irreducible elements in $\U(h)^\sqcup$.
\end{lemma}
\begin{proof}\label{p.2283}
 By the definition of $\anz_h$ we have $\U(h) = \{\mfu \in \U(h)^\sqcup:\, \anz_h(\mfu) =1 \} $.

 We show first that $\U(h) \subset$ irreducible elements. Let $\mfu \in \U(h)$, i.e.~$\anz_h(\mfu) = 
1$ and $\mfv \leq \mfu$, i.e.~there is a $\mfw \in \U(h)^\sqcup$ with $\mfv \sqcup \mfw = \mfu$. By 
Lemma (\ref{l.2.4EM}) we know that $\anz_h(\mfv) + \anz_h(\mfw) = 1$. This lets us with the two 
possibilities $\anz_h(\mfv) = 0$ or $\anz_h(\mfw)=0$, the first implying, that $\mfv = 0$ and the 
latter that $\mfv = \mfu$.
 
 On the contrary suppose $\mfu = [U,r,\mu]$ is irreducible. Suppose $\anz_h(\mfu) \geq 2$. For a 
point $x \in U$ consider $B(x,2h) = \{y \in U: \ r(x,y) < 2h \}$. Then $U' = U \setminus B(x,2h)$ 
has positive mass w.r.t.~$\mu$. Define $\mfv = [B(x,2h), r, \mu]$ and $\mfw = [U',r,\mu]$ where we 
always restrict $r$ and $\mu$ to the corresponding sets. It is easy to show that $\mfu = \mfv \sqcup 
\mfw$, which is a contradiction to irreducibility of $\mfu$.
\end{proof}

\begin{proposition}[Decomposition of forests in trees] \label{l.delphic}
 $(\U(h)^\sqcup,\sqcup^h)$ is a Delphic semigroup with $0$, the neutral element, as the only 
infinitely divisible element. Any $\mfu \in \U(h)^\sqcup$ can be represented as
 \begin{equation}\label{e.tr55}
  \mfu = \bigsqcup_{i \in I} \mfu_i 
 \end{equation}
for a countable index set $I$ and $\mfu_i \in \U(h)$.
\end{proposition}
Before we give a proof we say that the semigroup is also \emph{sequentially Delphic} in the sense of 
\cite{Davidson68}, since $\U(h)^\sqcup$ is first countable (metric space!).
\begin{proof}\label{p.2309}
 We establish criterions (A), (B) and (C) in \cite{Kendall68}.

 (A): The total mass mapping $\Delta: \U(h)^\sqcup \to [0,\infty),\, \mfu \mapsto \bar{\mfu}$ is a 
semigroup homomorphism since $\Delta(\mfu \sqcup \mfv) = \overline{\mfu \sqcup \mfv} = (\mu + \nu) 
(U \uplus V) = \mu(U) + \nu(V) = \bar{\mfu} + \bar{\mfv} = \Delta(\mfu) + \Delta(\mfv)$ for 
$\mfu=[U,r_U,\mu], \mfv = [V,r_V,\nu] \in \U(h)^\sqcup$. The mapping is continuous by definition of 
the Gromov-weak topology.
 
 (B): For any $\mfu \in \U(h)^\sqcup$, the set $\{ \mfv \in \U(h)^\sqcup: \, \exists \mfw \ \mfu= \mfv 
\sqcup \mfw \}$ of divisors of $\mfu$ is compact. This is true by Lemma (\ref{l.2.7EM}) 
\eqref{i.2.7EM.d}.
 
 (C): Suppose $\{ \mfu(i,j) \in \U(h)^\sqcup :\, 1\leq j \leq i \in \N\}$ is a null-triangular array, 
i.e.~for any $i \in \N$ there is a $c(i) \geq 0$ such that $\Delta(\mfu(i,j)) \leq c(i)$ for all $j 
\leq i$ and $\lim_{i\to \infty} c(i) = 0$. Suppose
 \begin{equation}\label{e2114}
  \mfv_i := \sideset{}{^{h}}\bigsqcup_{1\leq j \leq i} \mfu(i,j)
 \end{equation}
 converges to a limit $\mfv$ ($\in \U(h)^\sqcup$ by closedness, see Lemma (\ref{l.frsts.clsd})). We want 
to establish that $\mfv = 0$, which is stronger than what Kendall requires in (C), but also states 
that \emph{the only infinitely divisible element} in the semigroup is $0$ by his Theorem II. In order 
for $\mfv(i)$ to converge the sequence needs to be tight in the Gromov-weak topology. However, for 
any $\delta >0$ we find an $i(\delta)$ with $c(i') < \delta$ for $i' \geq i(\delta)$ and thus
 \begin{equation}\label{e2115}
  v_\delta (\mfv_{i'}, h/2) = \sum_{j=1}^{i'} v_\delta(\mfu(i',j), h/2) = \sum_{j=1}^{i'} 
\bar{\mfu}(i',j) = \bar{\mfv}_{i'} \, 
 \end{equation}
by Lemma (\ref{l.2.7EM}) \eqref{i.2.7EM.b}. In order to satisfy the tightness criterion in Proposition 
(\ref{p.tight-crit}) we need to have that $\bar{\mfv}_{i'} < \eps$ for all large $i'$ and thus to have 
$\bar{\mfv} \to 0$. This means that for the limit $\bar{\mfv} = 0$ and hence $\mfv = 0$. Thus we 
have established that $\U(h)^\sqcup$ is a Delphic semigroup.

By Theorem III in Kendall's article and with Lemma 6 we know that for any $\mfu \in \U(h)^\sqcup$ a representation as 
in \eqref{e.tr55} exists; he calls irreducible elements ``indecomposable''.
\end{proof}

\begin{proof}[Proof of Theorem~\ref{p.trLap}]\label{pr.1903}
From the Proposition~\ref{l.delphic} above we get part (a) of the Theorem. The part (b) follows since the truncation provides an $h^\prime$-decomposition, so that the uniqueness shown in (a) gives the claim.
\end{proof}

\subsection{Properties of truncated monomials}
\label{ss.proptt}

We begin studying truncation and monomials.

\begin{proposition}[Truncated monomials of concatenation]\label{p2907131206}
  Let $h > 0$ and $\mfu_i \in \U(h)^\sqcup$, $i \in I$ for a finite or countable set $I$. Let 
$\mfu=\bigsqcup_{i \in I}\mfu_i$ be the $h$-concatenation of $(\mfu_i)_{i\in I}$ as in Definition 
(\ref{D.concat}). Then, for every $\Phi \in \Pi$,
  \begin{equation}\label{e2907131046}
    \Phi_{\hu}(\mfu)=\Phi_{\hu}(\lfloor\mfu\rfloor(h))=\sum_{i\in I}\Phi(\mfu_i).
  \end{equation}
\end{proposition}
This establishes Theorem (\ref{p.trunc.poly}) \eqref{i.tr65} and \eqref{i.tr66}.
\begin{proof}\label{p.1573}
 Suppose that $\Phi = \Phi^{m,\phi} \in \Pi$ is a monomial. Let $\mfu_i = (U_i, r_i, 
\mu_i)$ for $i \in  I$. Recall that for $\mfu = (U,r,\mu)$: $\mu = \sum_{i\in  I} \mu_i$ and thus
 \begin{align}\label{grx74}
  \Phi_{\hu} (\mfu) = & \int_U \mu^{m} (\dx \underline{x}) \, \phi_h(\dr(\underline{x}))  = \int_U 
(\sum_{i\in  I} \mu_i)^m (\dx \underline{x}) \, \phi(\dr(\underline{x})) \prod_{1\leq i < j\leq m} 
\1(r(x_i,x_j) < 2h) \\
  & = \sum_{i \in  I} \int_{U_i} \mu_i (\dx \underline{x}) \, \phi(\dr(\underline{x}))  = \sum_{i 
\in  I} \Phi(\mfu_i), \nonumber
 \end{align}
where we used that $r(x,y) = 2h$ whenever $x$ and $y$ are not contained in the same $U_i$, $i \in  
I$. The other equality follows by $\Phi_{\hu}(\lfloor \mfu \rfloor(\hu)) = \Phi_{\hu}(\mfu)$.
\end{proof}
In the previous result we did not require $\Phi \in \Pi_h$; it sufficed to have $\Phi \in \Pi$.

Recall that $\#_h(\mfu)$ is defined as the number of prime elements of $\lfloor\mfu\rfloor(h)$.
Let $\#_h(\mfu)$ be the (unique) number of open balls of radius $h$ in $\mfu$, (it is easy to see that one can recover $\#_h(\mfu)$ from the distance matrix distributions): for $\mfu\in\bbU$,  
\begin{equation}\label{e.D.anz}
    \anz_h(\mfu) = \sup \{ m \in \N: \ \nu^{m,\mfu} ([2h,\infty)^{\binom{m}{2}}) >0, \,  
 \nu^{m+1,\mfu} 
  ([2h,\infty)^{\binom{m+1}{2}})=0 \} \in \N_0 \cup \{\infty \} \, .
\end{equation}
Now we can prove Proposition (\ref{l.2.4EM}) which states the additivity and measurability of the 
map $\#_h:\bbU\to\bbN_0$.

\begin{proof}[Proof of Proposition \ref{l.2.4EM}]\label{p.pl24em}
 Measurability is clear, since $\1(\nu^{m,\mfu}([0,2h)^{\binom{m}{2}} > 0)$ is measurable for all $m \in \N$.
 It is true that for $m \in \N$:
 \begin{equation}\label{grx76}
  \nu^{m, \mfu \sqcup \mfv} ( [2h,\infty)^{\binom{m}{2}} ) = \sum_{k_1+k_2 = m} \nu^{k_1,\mfu} 
([2h,\infty)^{\binom{k_1}{2}}) \nu^{k_2,\mfv} ([2h,\infty)^{\binom{k_2}{2}}) \, .
 \end{equation}
 This establishes the claim, since taking $m > \anz_h(\mfu) + \anz_h(\mfv)$ does not allow one to find a 
pair $(k_1,k_2)$ such that the right hand side is positive, whereas $m\leq \anz_h(\mfu) + 
\anz_h(\mfv)$ allows at least one positive summand on the right hand side choosing $k_1 \leq 
\anz_h(\mfu)$ and $k_2 \leq \anz_h(\mfv)$.
\end{proof}
The previous lemma directly implies the following result.
\begin{corollary}\label{c.tr1}
 For $\mfu \in \U(h)^\sqcup$ with $\mfu = \bigsqcup_{i=1}^k \mfu_i$ for $\mfu_i \in \U(h), i\leq k \in \N \cup \{\infty\}$ we have that $\anz_h(\mfu)= k$.
\end{corollary}

Above we saw truncated monomials are homomorphisms on our topological semigroup.
The next two results state that the class of truncated monomials allows to identify elements and that their initial topology on $\U(h)^\sqcup$ coincides with the induced topology from $(\U,d_{\textup{GPr}})$.

\begin{proposition}[$h$-truncated monomials characterize $h$-top]\label{p.trpol.sep}
 Let $h  > 0$ and $\mfu,\mfu'\in\bbU(h)^{\sqcup}$. Then $\mfu=\mfu'$ iff $\Phi(\mfu) = \Phi(\mfu')$
  for all $\Phi \in \CA_+(\Pi_h)$.
 Furthermore, the element $\lfloor\mfu\rfloor(h)\in\bbU(h)^{\sqcup}$ is uniquely determined by
  \begin{equation}\label{e.tr226}
    \Phi(\mfu)=\Phi(\lfloor\mfu\rfloor(h))\,, \quad \Phi \in \CA_+(\Pi_h).
  \end{equation}
\end{proposition}
\begin{proof}\label{p.1627}
The second statement is an easy consequence of the first one. For the first statement necessity is 
obvious. 

For sufficiency, note that for fixed $m \in \N$ we have $\lan \phi, \nu^{\mfu,m} \ran = \lan \phi , \nu^{\mfu',m} \ran$ for all $\phi \in C_b([0,\infty)^{\binom{m}{2}}, [0,\infty))$ which are equal to zero outside of $[0,2h)^{\binom{m}{2}}$.
These functions $\phi$ \emph{generate an algebra} of bounded continuous functions \emph{separating points} in $[0,2h)^{\binom{m}{2}}$.
Therefore Theorem 3.4.5 in \cite{EK86} implies that
 \begin{equation}\label{e.tr227}
  \nu^{m,\mfu}|_{[0,2h)^{\binom{m}{2}}} = \nu^{m,\mfu'}|_{[0,2h)^{\binom{m}{2}}} \, .
 \end{equation}
We need to show $\nu^{m,\mfu} = \nu^{m,\mfu'}$ for all sets in 
$\mcB(\R^{\binom{m}{2}}), m\geq 1.$ Then $\mfu = \mfu'$ by Gromov’s
reconstruction theorem for mm-spaces. 

Fix $m \geq 1.$ We will
deduce an expression for $\nu^{\mfu,m}(A)$ for $A \in \mcB(\R^{\binom{m}{2}})$ only relying on values 
of $\nu^{\mfu,m}|_{[0,2h)^{\binom{m}{2}}}$. This derivation can then also be done for $\nu^{\mfu',m}$ 
giving the result. It suffices to check the equality for
 \begin{equation}\label{e.tr228}
  A \in \mcB(\R^{\binom{m}{2}}), \text{ where } A = \bigtimes_{1\leq i < j \leq m} A_{ij},\, A_{ij} =[a_{ij},b_{ij})
 \end{equation}
 for some $a_{ij},b_{ij} \in \R$.
 However, since $\nu^{\mfu,m}$ and $\nu^{\mfu',m}$ only have 
positive mass on $\bbD_m \cap ([0,2h)\cup \{2h\})^{\binom{m}{2}}$
 , we can further restrict to $A$ of the form in \eqref{e.tr228} with the \emph{additional} property that:
 \begin{equation}\label{e49}
   A \in \mcB(\R^{\binom{m}{2}}), \text{ where } A_{ij}\subset [0,2h) \text{ or } A_{ij} = \{2h\}, 
1\leq i< j \leq m.
 \end{equation}
Define a permutation $\pi$ on $\{1,\dotsc,m\}$ depending on $A$ in the following inductive way: let 
$\pi(1)=1$ and let $(i_2,\dots, i_{l_1})$ be the ordered collection of indices $i \in 
\{2,\dotsc,m\}$ with $A_{1i} \subset [0,2h)$.
Define $\pi(i_j) = j$.
If there are any indices left (that is if $l_1 <m$), then take the smallest one, that is $m_1 = \inf ( \{1,\dotsc,m\} \setminus 
\{1,i_2,\dotsc, i_{l_1}\})$ and let $\{i_{l_1+1},\dots, i_{l_2}\}$ those indices with $A_{m_1,*} 
\subset [0,2h)$. Define $\pi(i_j) = j.$ Continue until no indices are left.
The considerations with the permutation $\pi$ allows us to give a ``subtree decomposition'' $\pi^\ast A$ of $A$ of the form in \eqref{e49}.

We define now a symbol $V$ as follows.
By symmetry of $\nu$, we know that $\nu^{\mfu,m}(A) = \nu^{\mfu,m}(\pi_*A)$, where $\pi_*A = \{\dr: 
r_{\pi^{-1}(i)\pi^{-1}(j)} \in A\}.$ 
Thus we may work with the rearranged $A$ now. Then by \eqref{e49},
\begin{equation}\label{e1768}
r_{ij} = 2h$, if $1\leq i\leq l_p< j\leq m \text{  for some  } p \in \{1,\dots, k\}.
\end{equation}
Thus we may work with $\pi_\ast A$ instead of $A$.
We can w.l.o.g.~restrict to $A$ of the form.
 \begin{align}\label{e.tr229}
  A &= A_{l_1}^{(1)} \vee \dots \vee A_{l_k}^{(k)} \qquad \text{ for } 0=l_0<l_1< \cdots l_k =m, \\
  \intertext{and requiring for $p \in \{1,\dots,k\}$ that}
  r_{ij} &\in A \Leftrightarrow \begin{cases} r_{ij} & \in A_{l_p}^{(p)} \quad \text{ if 
}l_{p-1}<i<j \leq l_p \\
                                r_{ij} & = 2h  \quad \text{ otherwise}.
                               \end{cases}\label{e.tr229b}
 \end{align}

Then we have, restricted to $\bbD_m$:
\begin{align}\label{e.tr230}
 A &= \{ r_{ij} \in A_{l_p}^{(p)} \quad \text{ if }l_{p-1}<i<j \leq l_p,\ p\leq k,\ \text{  otherwise \eqref{e1768} holds  }\\
 & = \{ r_{ij} \in A_{l_p}^{(p)} \quad \text{ if }l_{p-1}<i<j \leq l_p\} \label{e764b}\\
 & \qquad \setminus \{ r_{ij} \in A_{l_p}^{(p)} \quad \text{ if }l_{p-1}<i<j \leq l_p, \label{e764c}\\
 & =: \hat{A} \setminus (\hat{A} \cap) \text{  (the negation of  \eqref{e1768} holds)}.\label{e764d}
\end{align}
Now we can use the inclusion-exclusion formula to obtain:
\begin{align}\label{e.tr231}
 \nu^{\mfu,m} (A) = \nu^{\mfu,m}(\hat{A}) - \sum_{l=1}^{\binom{k}{2}} (-1)^{l+1} \sum_{I\subset \{ 
(p,q):\, 1\leq p < q \leq k\}, |I|=l} \nu^{\mfu,m}( \hat{A} \cap \bigcap_{(p,q)\in I} \{r_{l_p l_q} 
\in [0,2h) \} ).
\end{align}
But this is a formulation where on the right hand side in between two chosen points either their 
distance is less than $2h$ or there is no restriction at all. This allows to calculate 
$\nu^{\mfu,m}(A)$ using $\nu^{\mfu,n}|_{[0,2h)^{\binom{n}{2}}}$, $1\leq n \leq m$ only. That is what 
we had to show.
\end{proof}

In the next proposition we show that the class of truncated continuous monomials is also convergence-determining.
The metric $d_{\textup{GPr}}$ is defined in (\ref{e.dGPr}).
\begin{proposition}[Convergence criterion for $h$-forests] \label{p.trpol.topol}
 For $\mfu_n ,\mfu \in \U(h)^\sqcup$:  $\mfu_n \to \mfu$ in $d_{\textup{GPr}}$ as $n\to \infty$ iff $\Phi(\mfu_n) \to \Phi(\mfu)$ for all $\Phi \in \Pi_{h,+}$.
\end{proposition}
\begin{proof}\label{p.1714}
 For necessity suppose that $\mfu_n \to \mfu$ as $n\to \infty$ and $\Phi \in \Pi_{h,+}$.
 Since $\Phi \in \Pi$ we know by Theorem 5 in \cite{GPW09} that $\Phi (\mfu_n) \to \Phi(\mfu)$.
 
 For sufficiency note that Theorem 5 in \cite{GPW09} also tells us that it suffices to show that $\nu^{\mfu_n,m} \Rightarrow \nu^{\mfu,m}$ as $n \to \infty$ for any $m \in \N$; convergence here means weak convergence of measures on $[0,\infty)^{\binom{m}{2}}$.
 Fix $m \in \N$.
 By assumption we know that $\lan \phi, \nu^{\mfu_n,m} \ran \to \lan \phi , \nu^{\mfu,m} \ran$ for all $\phi \in C_b([0,\infty)^{\binom{m}{2}}, [0,\infty))$ which are equal to zero outside of $[0,2h)^{\binom{m}{2}}$.
 Therefore, we know that for $n \to \infty$:
 \begin{equation}\label{e765}
  \nu^{m,\mfu_n}|_{[0,2h)^{\binom{m}{2}}} \Rightarrow \nu^{m,\mfu}|_{[0,2h)^{\binom{m}{2}}} \, .
 \end{equation}
We need to extend that convergence to $[0,\infty)^{\binom{m}{2}}$.
Theorem 2.4 in \cite{Bill2nd} states that it suffices to show convergence of the measures only for sets of the form \eqref{e.tr228} which have a boundary with $\nu^{m,\mfu}$-measure zero.
 For such sets we can derive \eqref{e.tr231} for $\mfu$ and also for $\mfu_n$,$n \in \N$.
 The argument in Billingsley's proof applied for the measures $\nu^{1,\mfu}, \dots, \nu^{m,\mfu}$ allows to deduce weak convergence.
\end{proof}

The following quantitative estimate plays a role.
\begin{lemma}\label{l.tr1}
 For $s<t$, $\mfu =[U,r,\mu] \in \U$ and a monomial $\Phi = \Phi^{m,\phi} \in \Pi$:
 \begin{equation}\label{grx79}
  |\Phi_t(\mfu) - \Phi_s(\mfu)| \leq {\binom{m}{2}} \bar{\mfu}^{m-2} \|\phi\|_\infty  
\nu^{2,\mfu}([2s,2t)) \, .
 \end{equation}
\end{lemma}
\begin{proof}\label{p.1856}
 The result is obtained via direct calculation:
 \begin{align}\label{grx80}
  |\Phi_t(\mfu) - & \Phi_s(\mfu)|  = | \int \mu^{\otimes m}(\dx \underline{x}) \phi(\dr 
(\underline{x})) \\
  & \phantom{aaaaaa} ( \1(r(x_i,x_j) < 2t \ \forall 1\leq i < j \leq m ) - \1(r(x_i,x_j) < 2s \ 
\forall 1\leq i < j \leq m ) ) | \nonumber \\
  & \leq \|\phi \|_\infty \int \mu^{\otimes m}(\dx \underline{x}) \label{grx80b}\\
  & \phantom{aaaaaa} | \1(r(x_i,x_j) < 2t \ \forall 1\leq i < j \leq m ) - \1(r(x_i,x_j) < 2s \ 
\forall 1\leq i < j \leq m )  | \nonumber \\
  & \leq \|\phi \|_\infty \int \mu^{\otimes m}(\dx \underline{x}) \1(2s \leq r(x_i,x_j) < 2t \ 
\exists 1\leq i < j \leq m  )  \, . \nonumber
 \end{align}
 Next note that the set in the indicator goes to the empty set as $s \uparrow t$, hence continuity follows.
\end{proof}

These results allows us to give a proof of Proposition (\ref{p.tr.cont}).
\begin{proof}[Proof of Proposition \ref{p.tr.cont}]\label{p.pptrcont}
Let $\mfu_n \to \mfu$ in $d_{\textup{GPr}}$. Then by Theorem 5 in \cite{GPW09}:
 \begin{align} \label{e766a}
  \mfu_n \to \mfu &\Leftrightarrow \Phi(\mfu_n) \to \Phi(\mfu) \ \text{ for all }  \Phi \in \Pi \\
&\Rightarrow \Phi_h(\mfu_n) \to \Phi_h(\mfu) \ \text{ for all }  \Phi \in \Pi_h \label{e766b}\\
&\Leftrightarrow \Phi_h(\lfloor\mfu\rfloor_n(h)) \to \Phi_h(\lfloor\mfu\rfloor(h)) \ \text{ for all }  \Phi \in \Pi_h \label{e766c}\\
&\Leftrightarrow \lfloor\mfu\rfloor_n(h) \to \lfloor\mfu\rfloor(h)  \, .\label{e766d}
 \end{align}
In the last step we applied Proposition (\ref{p.trpol.topol}). The continuity in $ h $ follows from Lemma~\ref{l.tr1}.
\end{proof}

Now we give two technical results about truncated monomials.
Given a monomial $\Phi = \Phi^{m,\phi}$ with
$\phi: \bbR^{\binom{m}{2}} \to \R$ for $m \in \N$ and $\phi = \phi_h$, i.e.~$\phi(\dr) = 
\phi(\dr) \prod_{1\leq i < j \leq m} \1(r_{ij}<2h)$ for $n \in \N$ define
\begin{equation}\label{e.tr74}
 \phi^{(n)}_h: \begin{cases}
                &\bbR^{n\binom{m}{2}} \to \bbR \\
                & \dr \mapsto \prod_{1\leq p \leq n} \phi \left( (r_{ij})_{(p-1)m + 1 \leq i < j 
\leq pm } \right) \1(r_{pm, pm+1} < 2h ) \, ,
               \end{cases}
\end{equation}
where we set $r_{nm,nm+1} = 0$. We sometimes omit the subscript $h$ and simply write $\phi^{(n)}$ 
for $\phi^{(n)}_h$.

We define another function for $n_1, \dots, n_l \in \N$:
\begin{align}\label{e.tr75}
\mbox{}
\end{align}
$ \phi^{(n_1, \dots, n_l)} : 
\begin{cases}
\bbR^{\sum_{q=1}^l n_l \binom{m}{2}} \to \R &\\
\dr \mapsto \prod_{q=1}^l \phi_h^{(n_q)} \left( (\dr 
_{ij})_{\sum_{p=1}^{q-1} n_p m + 1\leq i < j \leq \sum_{p=1}^q n_p m} \right) &\\
\qquad \qquad \qquad \qquad \qquad \qquad \qquad \qquad \qquad \qquad\1( r_{\sum_{p=1}^q 
n_p m , \sum_{p=1}^q n_p m +1 } =2h ) \, ,&
\end{cases}                           $
\\
\\
where we set $r_{x, x +1 } = 2h$ for $x = \sum_{p=1}^{q} n_p m$.

The next result is used for truncated monomials.
\begin{lemma}\label{l.tr20}
Let $\mfu = \sqcup_{i=1}^k \mfu_i , \, \mfu_i \in \bbU(h)$ and $\phi = \phi_h \in \mcB(\bbR^{\binom{m}{2}}),\, m \in \N$.
 \begin{enumerate}
 \item\label{i.tr46} For $n \in \N$ 
    \begin{equation}\label{e767}
   \Phi^{\phi^{(n)},mn} (\mfu) = \sum_{i=1}^k \left( \Phi^{m,\phi}(\mfu_i) \right)^n \, .
  \end{equation}
 \item\label{i.tr47} For $n_1, \dots, n_l \in \N, \, l \in \N$ with
   $S_l$ is the set of permutations of $\{1,\dots, l\}$
  \begin{equation}\label{e768}
   \Phi^{\phi^{(n_1,\dots, n_l)},m\sum_i n_i} (\mfu) = \sum_{1\leq k_1 < \cdots < k_l \leq k} 
\sum_{s \in S_l} \prod_{i=1}^l \left( \Phi^{m,\phi}(\mfu_{k_i}) \right)^{n_{s(i)}} \, .
  \end{equation}
 \end{enumerate}
\end{lemma}
\begin{proof}\label{p.1793}
\eqref{i.tr46}: First note that here we consider a truncated monomial:
 \begin{equation}\label{e.tr76}
  \Phi^{nm, \phi_h^{(n)}} = \left( \Phi^{nm, \phi_h^{(n)}} \right)_h \, .
 \end{equation}
 This is due to the fact that for $\dr \in \supp(\phi_h^{(n)})$ we have short distances $(< 2h)$ 
within each block of length $m$ and moreover $r_{pm,pm+1}<2h$ for $p=1,\dots, n$. Using this and
ultrametricity we have for $k,l \in \{1,\dots, nm\}$, $k < l$ 
 \begin{equation}\label{e769}
  r_{kl} \leq r_{k,m\lceil k/m \rceil} \wedge r_{m\lceil k/m \rceil, l} \leq \dots \leq r_{k, m 
\lceil k/m \rceil} \wedge r_{m \lceil k/m \rceil, m (\lceil k/m \rceil  +1)} \wedge \dots \wedge 
r_{m\lfloor l/m \rfloor , l} < 2h \, .
 \end{equation}
 Therefore we can apply Proposition (\ref{p2907131206}) for $\Phi^{nm,\phi^{(n)}}$ to obtain:
 \begin{equation}\label{e.tr77}
  \Phi^{nm,\phi^{(n)}} (\mfu) = \sum_{i=1}^k \Phi^{nm,\phi^{(n)}} (\mfu_i) 
 \end{equation}
 and we will continue the proof with $\mfu_1 = [U_1, r, \mu_1] \in \bbU(h)$. Let us use that in 
$U_1$ all distances are $< 2h$ $\mu_1$-a.s.~from 1st to 2nd line:
 \begin{align}\label{grx77}
   \Phi^{nm,\phi^{(n)}}(\mfu_1) & = \int \mu_1^{\otimes nm} (\dx \underline{x} ) \, \prod_{1\leq p 
\leq n} \phi \left(  (r_{ij})_{(p-1)m < i < j \leq pm} (\underline{x} ) \right) \1( 
r_{pm,pm+1}(\underline{x}) < 2h ) \\
  & \int \mu_1^{\otimes nm} (\dx \underline{x} ) \, \prod_{1\leq p \leq n} \phi \left(  
(r_{ij})_{(p-1)m < i < j \leq pm} (\underline{x} ) \right) \nonumber \\
   & = \prod_{1\leq p \leq n} \int \mu_1^{\otimes m} (\dx x_{(p-1)m}, \dots , \dx x_{pm}) \phi 
(\dr(\underline{x})) \nonumber\\
   & = \left( \int \mu_1^{\otimes m} (\dx x_1, \dots, \dx x_m) \, \phi(\dr (\underline{x})) 
\right)^n  = \left( \Phi^{nm, \phi^{(n)}} (\mfu_1) \right)^n \, . \nonumber
 \end{align}
 Combine this and \eqref{e.tr77} to obtain the result.

\eqref{i.tr47}:
 By definition of a monomial obtain for $\mfu_i = [U_i, r, \mu_i]$:
 \begin{align}\label{grx77b}
  \Phi^{\phi^{(n_1,\dots, n_l)},m \sum_i n_i} & (\mfu)  = &  \\
& \int  \left( \sum_{p=1}^k \mu_i \right)^{ 
\otimes \sum_{q=1}^l n_q m} (\dx \underline{x}) \, \prod_{q=1}^l \phi_h^{(n_q)}  \left( (\dr 
_{ij})_{\sum_{p=1}^{q-1} n_p m + 1\leq i < j \leq \sum_{p=1}^q n_p m} \right)&  \nonumber \\
& \qquad \qquad \qquad \qquad \qquad \qquad \qquad \qquad \1( r_{\sum_{p=1}^q n_p m , \sum_{p=1}^q n_p m +1 } =2h ). & \nonumber
 \end{align}
 This is equal to zero if $q < \anz_h(\mfu) = k$ by the very definition of $\anz_h$.
 Otherwise choose $l$ out of the $k$ trees without replacement and choose the $mn_1, \dots, mn_l$ points $x_i$ 
from these sub-trees respectively.
The choice is made without resemblance to the order so we also 
need to take into account permutations leading us to:
 \begin{align}\label{grx78}
  \Phi^{\phi^{(n_1,\dots, n_l)},m\sum_i n_i} (\mfu) &= \sum_{1\leq k_1 < \cdots < k_l \leq k} 
\sum_{s \in S_l} \prod_{i=1}^l \int \mu_{k_i}^{\otimes m n_{s(i)}} (\dx x_1, \dots , \dx x_{m 
n_{s(i)}}) \phi_h^{(n_{s(i)})} (\dr (\underline{x})) \\
  & = \sum_{1\leq k_1 < \cdots < k_l \leq k} \sum_{s \in S_l} \prod_{i=1}^l \left( 
\Phi^{m,\phi}(\mfu_{k_i}) \right)^{n_{s(i)}} \, ,\nonumber
 \end{align}
 by part \eqref{i.tr46}.
\end{proof}

\subsection{Uniqueness of the factorization: Proof of Proposition \ref{p.tr1}}\label{ss.uniqfact}

The question to address later is the uniqueness of the factorization in \eqref{e.tr55}.
\begin{lemma}\label{l.tr21}
 Let $w,z \in [0,\infty)^\N$ with $w_1\geq w_2 \geq \dots \geq 0$, $z_1 \geq z_2 \geq \dots \geq 0$ and $\suml_{i < 1}^\infty w_i^n + z_i^n < \infty \; \forall n \in \N$. 
If 
 \begin{equation}\label{e.tr78}
  \sum_{k \in \N} w_k^n = \sum_{k\in \N} z_k^n 
 \end{equation}
 for any $n \in \N$ then $w = z$.
\end{lemma}

\begin{proof}\label{p.2438}
 Taking the $n$-th root in \eqref{e.tr78} and letting $n \to \infty$ we obtain $w_1 = z_1$.
 Subtract $w_1^n$ from the polynomial expressions and repeat this procedure iteratively.
\end{proof}

%
%

\begin{lemma}\label{l.tr22}
 Let $X$ be a set and $P$ a family of functions on $X$ which separates points. Suppose $P =  
\bigcup_{i=1}^n A_i$ for $n \in \N$. If none of $A_2, \dots, A_n$ does separate two given points in $X$, then 
$A_1$ separates these two points in $X$.
\end{lemma}

\begin{proof} 
It suffices to show the statement for $n=2$, since it can be extended by considering $A_1$ versus $A_2 \cup \ldots \cup A_n$.

Suppose that $x_i,y_i \in X$, $i \in \{1,2\}$ with $x_2 \neq y_2$ and
\begin{align}\label{grx88}
 &\phi_i(x_i) = \phi_i(y_i) \ \ \forall \phi_i \in A_i \, , \, i = 1,2 \, .
\end{align}
That means $A_2$ does not separate these two points in $X$. We claim that $A_1 = P \setminus A_2$ needs to 
separate these points, i.e.~$x_1=y_1$.

Define $A_1' = \{\phi \in P: \ \phi(x_2) \neq \phi(y_2) \} \subset P \setminus A_2 = A_1$.
Note that $A_1'$ is not empty since otherwise $\phi(x_2) = \phi(y_2)$ for any $\phi \in P$, which 
implies $x_2 = y_2$ since $P$ separates points; this would be a contradiction.
Fix $\phi_1 \in A_1'$. Then it is true that for \emph{any} $\psi \in A_2$:
\begin{equation}\label{grx89}
 (\phi_1 + \psi)(x_2) - (\phi_1 + \psi)(y_2) = \phi_1(x_2) - \phi_1(y_2) \neq 0 \, .
\end{equation}
Therefore $\phi_1 + \psi \in A_1'$, which is a subset of $A_1$. Thus,
\begin{align}\label{grx90}
 0&= (\phi_1 + \psi)(x_1) - (\phi_1 + \psi)(y_1) \\
 &= \psi(x_1) -\psi(y_1) \quad (\text{since } \phi_1 \in A_1) \, .\nonumber
\end{align}
Therefore $\psi(x_1) = \psi(y_1)$ for \emph{any} $\psi \in A_2$.
Altogether $\phi(x_1) = \phi(y_1)$ for any $\phi \in A_1 \cup A_2 = P$. Thus $x_1=y_1$.
\end{proof}

\begin{lemma}\label{l.tr23}
Define a subset of monomials $P = \{ \Phi = \Phi^{m,\phi} \in \Pi:\, [\inf \phi, \sup \phi] \subset [1,2], \, m \in \N
\}$.
 \begin{enumerate}
  \item\label{i.tr48} $P$ separates points in $\U$.
  \item\label{i.tr49} If $\Phi^{m,\phi} \in P$ and $\mfu,\mfv \in \U$ with $\Phi(\mfu) = \Phi(\mfv)$ then $\bar{\mfv} \in [\bar{\mfu} /2 , 2 \bar{\mfu} ]$.
 \end{enumerate}
\end{lemma}
\begin{proof}\label{p.2513}
 \eqref{i.tr48}: This is a standard argument and we omit it.
 \eqref{i.tr49}: Since $\Phi \in P$ we know that $\Phi(\mfu) \in [\bar{\mfu}_1^m, 2 \bar{\mfu}_2^m]$. This allows to give 
the bounds.
\end{proof}

\begin{proposition}[Unique factorization, prime elements]
\label{p.tr1}
~
 \begin{enumerate}
  \item\label{i.tr44} Let $k, k' \in \N \cup \{\infty\}$. Suppose $\mfu = \bigsqcup_{i=1}^k \mfu_i = 
\bigsqcup_{j=1}^{k'} \mfv_j$, where $\mfu_i, \mfv_j \in \U(h)$, $1\leq i \leq k, \, 1\leq j \leq k'$. Then $k=k'$ 
and there is a $ i \in \{1,\dots, k\}$ such that $ \mfu_1 = \mfv_i$.
  \item\label{i.tr45} Any irreducible element is prime, i.e.~the set $\U(h)\setminus\{0\}$ is the set of prime 
elements.
 \end{enumerate}
\end{proposition}

\begin{proof}[Proof of Proposition \ref{p.tr1}]\label{p.ptr1}
\mbox{}

 \eqref{i.tr44}: First it is clear that $k=k'$ since $\#_h$ is a well-defined 
function on the semigroup and thus $ \#_h (\mfu) = k = k' \in \N \cup \{\infty\}$ by Corollary (\ref{c.tr1}).
 
Fix $\Phi = \Phi^{m,\phi} \in \Pi$. Define the real numbers
 \begin{equation}\label{grx91}
  w_i := \Phi(\mfu_i) \, , \, z_i := \Phi(\mfv_i)\, , \ 1\leq i \leq k                                 
\end{equation}
and note that $w_i = \Phi_h(\mfu_i)$, $z_i=\Phi_h(\mfv_i)$ since $\mfu_i, \mfv_i \in \U(h)$.
Recall that for any monomial $\Psi$ that $\Psi_h(\bigsqcup_i \mfu_i) = \Psi_h (\bigsqcup_i \mfv_i)$ implies by 
Lemma (\ref{l.tr20}):
 \begin{equation}\label{grx92}
  \sum_{i\leq k} w_{i}^{n} = \sum_{i \leq k} z_{i}^{n} \,, \ n \in \N .
 \end{equation}
 But Lemma (\ref{l.tr21}) tells us that there is an $i \in \{1,\dots, k\}$ s.t.~$w_1 = z_i$, meaning 
that 
 \begin{equation}\label{grx93}
   \Phi(\mfu_1) = \Phi(\mfv_i)
 \end{equation}
 for the particular $\Phi = \Phi^{m,\phi}$.
 
 We can do the above for any $\Phi^{m,\phi} \in \Pi_h$.
 Note it is possible that we always get a different index $i \in \{1\,\dotsc, 
k\}$ such that $\Phi(\mfu_1) = \Phi(\mfv_i)$, and therefore $i = i(\Phi) \in \N \cap [1,k]$.
Altogether we have,
 \begin{equation}\label{e.tr73}
  \forall \Phi \in \Pi_h \ \exists i \text{ s.t. } \Phi(\mfu_1) = \Phi(\mfv_i) \, .
 \end{equation}

Since $\infty > \sum_{i =1}^k \bar{\mfv}_i$ we know that there is a $k'' < \infty$ such that 
$\bar{\mfv}_i < \bar{\mfu}/2$ for all $i \geq k'' +1$. Now use Lemma (\ref{l.tr23}) \eqref{i.tr49}. 
Thus, we know that \eqref{e.tr73} with $[\min \phi, \max \phi] \subset [1,2]$ can only hold for 
$i(\Phi) \in \{1,\dots, k''\}$.
 
 Therefore define the sets of monomials
 \begin{equation}\label{grx94}
  A_i := \{ \Phi^{m,\phi} \in \Pi_h:\, \Phi^{m,\phi}(\mfu_1) = \Phi^{m,\phi}(\mfv_i),\, [\min \phi, \max 
\phi] \subset [1,2]  \} \ , \, 1\leq i \leq k'' \, .
 \end{equation}
 By \eqref{e.tr73} we know that $P = A_1 \cup \dots \cup A_{k''}$. Lemma (\ref{l.tr23}) \eqref{i.tr48} 
says that $P$ is separating and by Lemma (\ref{l.tr22}) we know that at least one of the sets $A_i$, 
say $A_{i*}$, must be separating two given points. Thus, $\mfu_1 = \mfv_{i*}$.
 
 \eqref{i.tr45}: Suppose $\mfu_1$ is irreducible and divides $w+z$. By Lemma 
(\ref{l.delphic}) we know that $w=\bigsqcup_{i \in I_1} w_i$ and $z=\bigsqcup_{j \in I_2} z_j$ for some 
irreducible $w_i,z_j$ and at most countable sets $I_1$ and $I_2$, which we assume to be disjoint. 
Thus, $w+z = \bigsqcup_{i \in I} \mfv_i$ for $I=I_1 \cup I_2$ and $\mfv_i = w_i \1(i \in I_1) + z_i \1(i \in 
I_2)$. By \eqref{i.tr44} we know that there is $i* \in I$ such that $\mfu_1 = \mfv_{i*}$. If $i* \in I_1$, 
then $\mfu_1 | w$ and if $i* \in I_2$, then $\mfu_1 | z$. This is all we needed to show.
\end{proof}

\begin{proof}[Proof of Theorem \ref{p.delphic}]\label{p.2360}
 The combination of Proposition (\ref{l.delphic}) and Proposition (\ref{p.tr1}) \eqref{i.tr44} allows to 
deduce the result.
\end{proof}

The final remarks in this subsections are about the measurability of the prime factorization.
Recall the notation $\mcN^\#(E)$ as the set of locally bounded point measures on $E$, see Section (\ref{s.tightness}).
It is natural to write 
\begin{equation}\label{e2116}
P_\mfu = \sum_k m_k \delta_{\mfu_k} \mbox{ for any } \mfu = \bigsqcup_{k \in I} 
(\bigsqcup_{i=1}^{m_k} \mfu_k)
\end{equation}
as a measure concentrated on the set of irreducible elements 
$\U(h)$, $m_k \in \N$, $\mfu_k \in \U(h)$, $k \in I \subset \N$.
\begin{lemma}\label{l.6.1.EM14}
 The mapping $\pfd = \pfd_h: \, \U(h)^\sqcup \to \mcN^\#(\U(h)\setminus \{0\}), \, \mfu \mapsto 
P_\mfu$ is a bijection and bi-measurable.
\end{lemma}
This is exactly the result in \cite{EM14}'s Proposition 6.1 using only algebraic properties hence could be copied exchanging the binary operations, we omit the proof.
The previous result allows to define functionals via polynomials on the set $\mcN^\#(\U(h)\setminus\{0\})$. For $n  \in 
\mcN^\#(\U(h)\setminus\{0\})$  and $\Phi \in \Pi$ set:
\begin{equation}\label{grx96}
 n(\Phi) := \int_{\U(h)\setminus \{0\}} n(\dx \mfy) \, \Phi(\mfy) \, .
\end{equation}

We obtain the following easy lemma.
\begin{lemma}\label{l.tr2}
 Let $\mfu \in \U$ and $h>0$. Then for $\Phi \in \Pi$:
 \begin{equation}\label{grx97}
   \left( \pfd_h(\lfloor\mfu\rfloor(h)) \right) (\Phi) = \Phi_h(\mfu) = \sum_{i \leq \#_h(\mfu)} \Phi_h(\mfu_i)\, 
,
 \end{equation}
 the last equation under the assumption that $\mfu = \bigsqcup_{i\leq \#_h(\mfu)} \mfu_i$.
\end{lemma}

\begin{proof}\label{pr.2480}
Recall \eqref{e2116} and use the fact that $\Phi_h$ is an $\sqcup^h$-homomorphism.
\end{proof}

\begin{proposition}[Countable support of $ \nu^{2,\mfu}$]\label{pr.1871}
 For every ultrametric measure space $\mfu \in \bbU$ and $\eps >0$ the measure 
$\nu^{2,\mfu}|_{[\eps,\infty)}$ has a countable support, i.e.~there exist $x_n \in [\eps, \infty)$ 
and $m_n >0$, $n\in I$ for a countable index set $I$ such that
 \begin{equation}\label{grx81}
  \nu^{2,\mfu}|_{[\eps,\infty)} = \sum_{n\in I} m_n \delta_{x_n} \, .
 \end{equation}
\end{proposition}

\begin{proof}\label{p.1880}
Recall the remark~\ref{r.778}. Consider the $\tau(\eps)(\mfu)$. By Theorem~\ref{p.delphic} there is a 
unique prime factorization of $\tau(\eps)(\mfu)$ into countably many elements:
 \begin{equation}\label{grx82}
 \tau(\eps)(\mfu) = \bigsqcup_{i\in \N} \mfu(\eps,i) \, .
 \end{equation}
 For $\mfu(\eps, i)= [U_i,r,\mu_i]$ denote by $x_i$ an element in $U_i$. If $y_i$ is another element 
in $U_i$, then for $i\neq j$ by ultrametricity: $r(y_i,x_j) \leq r(y_i,x_i) \vee r(x_i,x_j) = 
r(x_i,x_j)$; moreover $r(y_i,x_j) \geq r(x_i,x_j)$, since $r(x_i,y_i) < 2 \ve \leq r(x_i,y_i)$. Thus 
$r(x_i,x_j) = r(y_i,x_j)$ for any element $y_i \in U_i$ and it suffices to only take one $x_i$ from 
each prime element $U_i$. Then we can calculate for $A \subset [\eps, \infty)$ with a countable sum:
 \begin{align}\label{grx83}
  \nu^{2,\mfu} (A) = \sum_{ i,j \in \N} \bar{\mfu}_i \bar{\mfu}_j \1( r(x_i,x_j) \in A) \, .
 \end{align}
\end{proof}

\subsection{Paths of family decompositions}\label{ss.pathsoftops}

For $\mfu\in\bbU(h)^\sqcup$, we can observe the {\em path} of family 
decompositions $\mfu(s)$ for $s\in[0,h]$. We denote the space of c\`{a}dl\`{a}g paths from 
$I\subseteq \bbR_+\to E$ by $D(I,E)$ equipped with the $(J_1)$-Skorohod topology. It is convenient to work with measure-valued representations. 

The $h$-family decomposition whose existence and uniqueness is guaranteed by 
Theorem~\ref{p.delphic} naturally induces a {\em point measure on subfamilies}. This measure represents 
the set of $h$-subfamilies in an equivalent way. The measure is in general not finite, but 
it is boundedly finite, that is, it is finite on bounded subsets. We denote the set of 
boundedly finite measures on a metric space $(E,d)$, by $\mcN^\#(E)$, and equip it with the 
weak$^\#$-topology, see \cite{DVJ03}.

In our case $E= \U(h)\setminus\{0\}$, the $h$-trees is equipped with the metric
\label{e1912}\[ d(\mfx,\mfy) = \dGPr (\mfx, \mfy) + \abs{ \bar{\mfx}^{-1} - \bar{\mfy}^{-1} } , \ \mfx, \mfy \in \U(h)\setminus \{0\} . \]

Define
\begin{equation}\label{e770}
 \pfd_h:\begin{cases}
         \bbU(h)^\sqcup&\to \mcN^\#(\bbU(h)\setminus\{0\})\,,\\
         \mfu=\bigsqcup_{i=1}^{\#_h(\mfu)}\mfu_i 
         &\mapsto\sum_{i=1}^{\#_h(\mfu)}\delta_{\mfu_i}\,.        
\end{cases}
\end{equation}
One can view $\pfd_h$ also as a map on $\bbU$ via the composition with $h$-truncation: $\pfd_h := \pfd_h \circ \cdot(h)$.
This map is a bi-measurable bijection (see Lemma~\ref{l.6.1.EM14}).  We have the following result:

\begin{proposition}[Path of tops measurable] \label{p:pathmeas}
Let $t>0$. The mapping
 \begin{equation}\label{grx99b}
  \begin{cases}
   \U & \to D ([0,t),\mcN^\#(\U(t)\setminus\{0\}))\,, \\
   \mfu & \mapsto (\Theta_{t-s}(\lfloor \mfu \rfloor (t-s)))_{s \in [0,t)}\,,
  \end{cases}
 \end{equation}
 is measurable. 
\end{proposition}

This is important when we want to consider the evolution of the family 
decomposition of a random forest of diameter $2t$ as $h$ varies, that is stochastic
process taking values in $\bbU(h)^\sqcup$.

\begin{proof}\label{pr.2041}
Like any \cadlag function, the function $(\Theta_{t-s}(\mfu))_{s\in [0,t)}$ can be approximated in the Skorohod topology via a step function
 \begin{equation}\label{e.EK86.ex.3.11.12}
   \Theta_{t-\cdot}(\mfu) = \lim_{n\to \infty} \sum_{k=1}^\infty \1\left( n(t-\cdot) \in [k-1,k) \right) \Theta_{k/n}(\mfu) \, .
 \end{equation}
 Thus, it suffices to check measurability of the mapping $\U \to \mcN^\#(\U(t)\setminus \{0\}), \, \mfu \mapsto \Theta_r(\mfu)$ for any $r \in(0,t)$.
 But this is obvious by the following two points. First, the mapping $\bbU \to \bbU(r)^\sqcup, \, \mfu \mapsto \mfu(r)$ is continuous by Lemma~\ref{p.tr.cont} and second, 
 the mapping $\Theta_r: \U(r)^\sqcup \to \mcN^\#(\U(r)\setminus \{0\})$ is measurable by Lemma~\ref{l.6.1.EM14}. We also used that $\mfU(r)$ is a measurable subset of $\lfloor \mfU \rfloor (t)$.
\end{proof}

\begin{lemma}\label{l.Phi_s.cadlag}
Suppose $t>0$.
Let $\mfu \in \U$.
 \begin{enumerate} 
 \item\label{i.tr16} For $t_n \nearrow t$:  $\pfd_{t_n}(\mfu) \Rightarrow 
\pfd_t(\mfu)$ converges boundedly weak in $\mcN^\#(\U(t)\setminus\{0\})$.
  \item\label{i.tr17} For $s_n \searrow s$ :  $\pfd_{s_n}(\mfu)$ converges boundedly 
weak in $\mcN^\#(\U(s_1)\setminus\{0\})$.
 \item\label{i.tr14} Define the path $[0,t) \mapsto \mcN^\#(\U(t)\setminus\{0\})$ by
 \begin{equation}\label{grx98}
  [0,t) \ni s \mapsto \pfd_{t-s}(\mfu(t-s)) \, .
 \end{equation}
 This path lies in $D ([0,t), \mcN^\#(\U(t)\setminus\{0\}))$.
\end{enumerate}
\end{lemma}
\begin{remark}\label{r.Phi_s.noncont}
 However, the mapping is not continuous. Consider for example $\mfu_n = [\{a,b,c\}, r(a,b)=2-n^{-1}, r(a,c) = 2+n^{-1}, \delta_a+ \delta_b + \delta_c]$ for $n \in \N \cup \{\infty\}$. Then $\mfu_n \to \mfu_\infty$ in $d_{\textup{GPr}}$. However, it is not true that the paths of the tops converge in the Skorohod $J_1$ topology. In \cite{MAX_Griesshammer} a finer metric on $\U$ is defined under which the mapping is continuous.
\end{remark}

\begin{proof}[Proof of Lemma~\ref{l.Phi_s.cadlag}]\label{p.1971}
\mbox{}

\eqref{i.tr16}:
 First, $\pfd_{t_n}(\mfu) \in \mcN^\#(\U(t_n) \setminus \{0\}) \subset \mcN^\#(\U(t)\setminus 
\{0\})$, so the statement makes sense.
The topology on $\mcN^\#(\U(t)\setminus\{0\})$ is that of boundedly finite measures. 
Tightness of the sequence $\{\pfd_{t_n}:\, n \in \N\}$ is obvious from Proposition (\ref{p.tight-crit}).
So we only need to show that the only limit point of that sequence is $\pfd_t$.
Suppose $\pfd'$ was another limit point for the sequence $t_n \to t$.
Let $\Phi = \Phi^{m,\phi} \in \Pi_+$:
\begin{align}\label{grx100}
 \abs{ \pfd'(1-e^{-\Phi(\cdot)}) - \left(\pfd_t(\mfu)\right) (1-e^{-\Phi(\cdot)}) } & \leq \abs{ \pfd'(1-e^{-\Phi(\cdot)}) - 
\left(\pfd_{t_n}(\mfu)\right) (1-e^{-\Phi(\cdot)}) } \\ &\quad  + \abs{ \left(\pfd_{t_n}(\mfu)\right)(\Phi) - 
\left(\pfd_t(\mfu)\right)(\Phi) } .\label{grx100b}
 \end{align}
 The first expression vanishes by definition of the limit and so we consider the second expression further using Lemmas~\ref{l.tr2} and~\ref{l.tr1}:
 \begin{align}\label{e771}
  \abs{ \left(\pfd_{t_n}(\mfu)\right) (\Phi) - \left(\pfd_{t}(\mfu)\right) (\Phi) }  & = \abs{ 
\Phi_{t_n}(\mfu) - \Phi_t(\mfu) } \\
 & \leq {\binom{m}{2}} \norm{\phi}_{\infty} \nu^{\mfu,2}\left([t_n,t)\right) \bar{\mfu}^{m-2} \, .\nonumber
\end{align}
But $[t_n,t) \searrow \emptyset$ and thus the right hand side approaches zero.
So the expression on the left-hand side of \eqref{grx100} is arbitrarily small.
By Proposition (\ref{c.tr2}) we obtain that $\pfd_t$ is the unique limit point of $\{\pfd_{t_n}:\, n 
\in \N\}$

\eqref{i.tr17}:
First, note that $(\Theta_{s_n}(\mfu))_{n \in \N}$ is tight in $\mcM^\#(\U(t)^\sqcup)$ by Proposition (\ref{p.tight-crit}).

A similar argumentation as before allows to derive that there is only a unique limit point with the help of Proposition (\ref{c.tr2}).
Let $\Phi = \Phi^{m,\phi} \in \Pi_+$.
For the increasing sequence $s_n$ we get for $1\leq k \leq n$:
\begin{align}\label{grx101}
 & \abs{ \left(\Theta_{s_n}(\mfu)\right)(1-e^{-\Phi(\cdot)}) - \left(\Theta_{s_k}(\mfu)\right)(1-e^{-\Phi(\cdot)}) }  \leq \abs{ \left(\Theta_{s_n}(\mfu)\right)(\Phi) - \left(\Theta_{s_k}(\mfu)\right)(\Phi) } \\
 &\qquad = \abs{ \Phi_{s_n}(\mfu) - \Phi_{s_k}(\mfu) } \ (\text{Lemma~\ref{l.tr2}}) \label{grx101b}\\
 &\qquad  \leq {\binom{m}{2}} \norm{ \phi }_{\infty} \nu^{\mfu,2}\left([s_n,s_k)\right) \bar{\mfu}^{m-2} \  \ (\text{Lemma~\ref{l.tr1})}.\nonumber
\end{align}
Since $\nu^{2,\mfu}$ is a finite measure we can bound the right hand side arbitrarily if only we 
take $k \in \N$ sufficiently large.

\eqref{i.tr14}: this is a  consequence of \eqref{i.tr16} and \eqref{i.tr17}.
\end{proof}

We obtain a result about the mass-fragmentation of an ultrametric measure space.
Consider the Polish space of decreasing numerical sequences which was defined in \cite{Bertoin}:
\begin{equation}\label{grx60}
 \mcS^{\downarrow} = \{\mathbf{s}=(s_1,s_2,\dots) \in [0,\infty)^\N :\, s_1 \geq s_2 \geq \cdots \geq 0,\, \sum_i s_i < \infty \} \subset \ell^1\, ,
\end{equation}
For every $h>0$ define the  map:
\begin{equation}\label{grx61}
 \mfS_h: \U \to \mcS^{\downarrow},\, \mfu  \mapsto (\bar{\mfu}_1, \bar{\mfu}_2, \dots ) \, ,
\end{equation}
if we assume that $\lfloor \mfu \rfloor (h) = \bigsqcup_{i\in \N} \mfu_i$ and that the trees $\mfu_i \in \U(h)$ are size-ordered 
 w.r.t.~their mass.
 The topology on $\mcS^{\downarrow}$ is given by the $\ell^1$ distance.

\begin{corollary}[Mass fragmentation of the top]\label{c.massfrag}
  The mapping
 \begin{equation}\label{grx99c}
  \begin{cases}
   \U & \to D ([0,t),\mcS^\downarrow) \\
   \mfu & \mapsto (\mfS_{t-s}(\lfloor \mfu \rfloor(t-s)))_{s \in [0,t)}
  \end{cases},
 \end{equation}
 is measurable.
\end{corollary}
\begin{proof}\label{p.2045}
 We already know by Proposition~\ref{p:pathmeas} that $\mfu \mapsto (\Theta_{t-s}(\mfu))_{s\in [0,t)}$ is measurable.
 By \cite[Exercise 3.11.13]{EK86} it suffices to show that $\mcN^\#(\U(t)\setminus\{0\}) \to \mcS^\downarrow, \sum_i \delta_{\mfu_i} \to (\bar{\mfu}_1, \bar{\mfu}_2, \dots)$ is continuous.
 The topology on $\mcN^\#(\U(t)\setminus\{0\})$ is that of boundedly finite convergence; that means that a sequence converges if all restrictions to bounded sets (i.e.~sets of the form $\{\mfu: \,\bar{\mfu}\geq \eps\}$) do converge.
 The topology on $S^\downarrow$ is that of $\ell^1$ and thus the continuity is obvious.
\end{proof}
The non-continuity issue in Remark (\ref{r.Phi_s.noncont}) is not resolved for the path of the mass fragmentation in the previous lemma.
With the same counterexample as there we see that the mapping is not continuous.
The interesting question whether the previous mapping is invertible has a negative answer, e.g.~$\mfu_i = [\{a,b,c,d\}, r_i, \delta_a + \delta_b + \delta_c + 2 \delta_d],\, i=1,2$ with $r_1(a,b)=1,r_1(a,c)=2,r_1(a,d)=3$ and $r_2(a,b)=1,r_2(a,c)=3,r_2(c,d)=2$ and the necessary extensions for ultrametric spaces do lead to the same mass-fragmentation process.

%

\subsection{Properties of trunks}\label{ss.pot}

\begin{proposition}[Approximation by trunk]
\label{P.approxtrunk}
  Letting $\mfu\in\bbU$ in the Gromov-weak topology:
  \begin{equation}\label{rg106}
    \lim_{h\downarrow 0} \mfu(\hd)=\mfu.
  \end{equation}
  
\end{proposition}

\begin{proof}\label{p.2079}
 Let $\Phi=\Phi^{m,\phi}\in \Pi$. Since $\phi$ is bounded and $\nu^{m,\mfu}$ is a finite measure 
dominated convergence implies
 \begin{align}\label{rg107}
   \lim_{h\downarrow0}\Phi(\mfu(\hd)))
   &=\lim_{h\downarrow0}\int\phi((r_{ij}-2h)_+)_{1\leq i<j\leq m})
            \,\nu^{m,\mfu}(\dx \dunderline{r}) \\
   &=\int\lim_{h\downarrow0}\phi((r_{ij}-2h)_+)_{1\leq i<j\leq m})
            \,\nu^{m,\mfu}(\dx \dunderline{r}) \label{rg107b}\\
   &=\int\phi((r_{ij})_{1\leq i<j\leq m})\,\nu^{m,\mfu}(\dx \dunderline{r}) \label{rg107c}\\
   &=\Phi(\mfu)\,.\label{rg107d}
 \end{align}
 Since this works for every $\Phi\in \Pi$, $\lim_{h\downarrow0}\mfu(\hd)=\mfu$ in 
Gromov-weak topology by Theorem 5 in \cite{GPW09}
\end{proof}

\section{Proofs for probability measures on the space $\U$}
\label{s.proofsprobm}
In this section we study {\em random} $ h-$forests and hence use heavily the algebraic and order structure of the semigroup of $ h-$forests from the last paragraph.
The central objects are Laplace transform and stochastic order.

\subsection{Laplace transform}\label{ss.lap}
In this section we will actually need polynomials instead of monomials.
Therefore we introduce some notation.
For $k \in \N$, $m_1, \dots, m_k \in \N$ and $\phi_i \in C_b(\bbD_{m_i})$ for $i=1,\dots, k$ we use the notation
\begin{equation}\label{e2117}
 \underline{m} = (m_1,\dots, m_k) \text{ and } \underline{\phi} = (\phi_1, \dots, \phi_k) \, .
\end{equation}
Then for every polynomial $\Phi =  \Phi^{(\underline{m},\underline{\phi})} \in \mcA(\Pi)$ 
we have the representation
\begin{align}\label{e2118}
 \Phi^{(\underline{m},\underline{\phi})} (\mfu) = \sum_{i=1}^k \int \nu^{\mfu, m_i}(\dx \dr) \, \phi_i(\dr) 
\end{align}
Any polynomial can be written in that way.
Monomials are those polynomials where $k=1$.

For the first result on general Laplace transforms we may work just with
monomials.
\begin{proposition}[Laplace functionals convergence determining]\label{p1006132107}
~
  \begin{enumerate}[(a)]
    \item\label{i0606131126} Let $\mfU,\mfU'$ be random um-spaces, i.e. with values in $ \U $. Then,
     \begin{equation}\label{e.tr262}
       \mfU\eqd\mfU'\quad\Longleftrightarrow\quad L_\mfU(\Phi)=L_{\mfU'}(\Phi)\;\;\forall \Phi\in 
\Pi_+\,.
     \end{equation}
    \item\label{i0606131126a} Let $\mfU,\mfU_n$, $n\in\bbN$, be random um-spaces. Then,
    \begin{equation}\label{e.tr262b}
     \mfU_n\underset{n\to\infty}{\Longrightarrow}\mfU\quad\Longleftrightarrow\quad 
L_{\mfU_n}(\Phi)\to L_\mfU(\Phi)\;\;\forall \Phi\in \Pi_+ \,.
    \end{equation}\label{e.tr263}
    Here, as usual, $\Longrightarrow$ denotes convergence in distribution.
  \end{enumerate}
\end{proposition}

\begin{remark}\label{R.prop}
The proposition says that the class of functions $\{\exp(-\Phi(\cdot)):\Phi\in \Pi_+\}$ is 
separating 
and convergence determining for $\mcM_1(\bbU)$ (of course, the second statement implies the first 
one).
\end{remark}

\begin{proof}[Proof of Proposition~\ref{p1006132107}]
\mbox{}\\
\noindent \eqref{i1006132102a}: Assume
\begin{equation}\label{gloede3}
 \bbE[ \exp(-\Phi^{m,\phi}(\mfU)) ] = \bbE[ \exp(-\Phi^{m,\phi}(\mfV)) ]  \quad  \forall m\,, \phi
\end{equation}
and hence
\begin{equation}\label{gloede4}
 \E[ \exp(-a\Phi^{m,\phi}(\mfU)) ] = \E[ \exp(-a\Phi^{m,\phi}(\mfV)) ]   \quad \forall a>0\,, m\,, 
\phi\,.
\end{equation}
By Theorem 3.1 (ii)' of \cite{Kall83}, this implies 
\begin{equation}\label{gloede5}
 \nu^{m,\mfU}\eqd\nu^{m,\mfV}\quad\forall m\,.
\end{equation}
It remains to show that
\begin{equation}\label{gloede6}
 \nu^{m,\mfU}\eqd\nu^{m,\mfV}\quad\forall m\quad\Longrightarrow\quad 
 \mfU\eqd\mfV\,.
\end{equation}
We can define a polar decomposition
\begin{equation}\label{gloede7}
 \pi:\bbU\to \bbR_+\times\bbU_1
\end{equation}
as $ (\bar \mfu, \wh \mfu) $ where $ \wh \mfu $ is $ (U,r,\bar \mfu^{-1} \mu)$ if $ \bar \mfu \neq 0 $ and otherwise define the 
following polar decomposition of $0$:
\begin{equation}\label{gloede8}
\pi(0):=(0,\delta)=(0,[\{1\},r,\delta_1])\,.
\end{equation}
Consider the maps
\begin{equation}\label{gloede9}
 \pi:\quad \pi_1:\nu^{m,\mfu}\mapsto \bar{\mfu}\,,\quad  
 \pi_2:\nu^{m,\mfu}\mapsto \nu^{m,\hat{\mfu}}\,.
\end{equation}
Restricted to um-spaces with positive mass $\pi$ is continuous. Take
$A_1\in\mcB(\bbR_+)$ and $A_2\in\mcB(\bbU_1)$. Then
\begin{equation}\label{gloede13}
 \pi^{-1}(A_1\times A_2)=\pi^{-1}((A_1\cap (0,\infty))\times 
A_2)\cup\pi^{-1}((A_1\cap\{0\})\times A_2)
\end{equation}
Note that $\pi^{-1}((A_1\cap (0,\infty))\times A_2)$ is measurable by continuity of $\pi$ 
on um-spaces with positive mass. Moreover $(A_1\cap\{0\})=\{0\}$. But 
$\pi^{-1}((A_1\cap\{0\})\times A_2)=\{0\}$. Thus one gets measurability.

Then $\pi_1,\pi_2$ are measurable  and hence
\begin{equation}\label{gloede10}
 \nu^{m,\mfU}\eqd\nu^{m,\mfV}\quad\Longrightarrow\quad
 (\bar{\mfU},\nu^{m,\hat{\mfU}})=\pi(\nu^{m,\mfU})\eqd\pi(\nu^{m,\mfU})=
 (\bar{\mfV},\nu^{m,\hat{\mfV}})\,.
\end{equation}
 Hence,
 \begin{equation}\label{e2212131153}
 \bbE[\bar{\Phi}(\bar{\mfU})\hat{\Phi}^{m,\phi}(\hat{\mfU})]
 =\bbE[\bar{\Phi}(\bar{\mfV})\hat{\Phi}^{m,\phi}(\hat{\mfV})]\,,
\end{equation}
 for any $\bar{\Phi} \in \mcB_b(\R_+)$, $\hat{\Phi}^{m,\phi}\in \hat{\Pi}$. Recall that $\hat{\Pi}$ 
is separating for 
$\mcM_1(\bbU_1)$ (see Proposition 2.6. in \cite{GPW09} for this reconstruction theorem) and $C_b(\bbR_+)$ is separating for $\mcM_1(\bbR_+)$. Hence, by \cite{EK86}[Prop. 
3.4.6], $\mcS:=\{\mfu\mapsto \bar{\Phi}(\bar{\mfu})\hat{\Phi}(\hat{\mfu}):\bar{\Phi}\in 
C_b(\bbR_+)\,,\hat{\Phi}\in\hat{\Pi}\}$ is separating for $\bbR_+\times \bbU_1$. Therefore, 
\eqref{e2212131153} implies:
\begin{equation}\label{gloede11}
 (\bar{\mfU},\hat{\mfU})\eqd(\bar{\mfV},\hat{\mfV})
\end{equation}
and this means
\begin{equation}\label{grx112}
 \mfU\eqd\mfV\,.
\end{equation}

\noindent \eqref{i1006132102b}: the proof can be obtained with a similar technique since in the statement we have given the limit object already.
\end{proof}

\begin{proof}[Proof of Theorem~\ref{p.trLap}]\label{p.ptrLap}
\mbox{}

 \eqref{i1006132102a}: It is clear that $B := \{\mfu \mapsto \exp(- \Phi(\mfu)) :\, \Phi \in \mcA(\Pi_{h,+}) \}$ is an 
algebra of bounded continuous functions on $\U(h)^\sqcup$. By Theorem (\ref{p.trunc.poly}) this 
algebra 
separates points in $\U(h)^\sqcup$ and therefore Theorem 3.4.5 of \cite{EK86} tells that $B$ is 
separating on $\U(h)^\sqcup$.\\
 
 \eqref{i1006132102b}: This follows from Theorem (\ref{p.trunc.poly}) \eqref{i.tr50} and Lemma 4.1 in 
\cite{HoffJ76}.
\end{proof}

\begin{remark}\label{r.tr1}
 One should note the difference to Proposition~(\ref{p1006132107}): there we only required to know 
about {\em monomials} in the Laplace transform, whereas in the truncated setting we actually need 
{\em polynomials}. Due to truncation we lack information about the joint distribution of the different 
sub-trees. Let $(X_1,X_2), (Y_1,Y_2)$ be random variables taking values in the cone $E= \{(x_1,x_2) \in \R^2:\, x_1 \geq x_2 \geq 0\}$.
Suppose $\mfU = (\{a,b\},r(a,b) = 3, X_1\delta_a + X_2\delta_b)$  and $\mfV = (\{a,b\},r(a,b) = 3, Y_1\delta_a + Y_2\delta_b)$ .
Then for $h =1$, we observe the contribution of $\nu^{m,\mfU}$ to $0$ namely $(X_1^m + X_2^m)$ and of
$\nu^{m,\mfV} \mbox{ being } (Y_1^m + Y_2^m)$.
So requiring $\nu^{m,\mfU}|_{[0,2h)^{\binom{m}{2}}} \eqd \nu^{m,\mfV}|_{[0,2h)^{\binom{m}{2}}}$ for all $m \in \N$ means that
\begin{equation}\label{e.tr333}
 X_1^m +X_2^m \eqd Y_1^m + Y_2^m \ \text{ for all } m \in \N \, .
\end{equation}
We do not know whether this is sufficient to state: $(X_1,X_2) \eqd (Y_1,Y_2)$.
However, it seems possible that there are similar examples where the restriction of Laplace transforms to monomials does not suffice to determine the laws.
This question seems to be related to inverse problems of Radon type in the cone $E$.
Even though there are results which state injectivity of restrictions of Radon transforms for compactly supported measures, see \cite{Krishnan}, we do not think that injectivity in \eqref{e.tr333} holds in general.
\end{remark}

\subsection{Example: compound Poisson forest}\label{ss.excomp}

We continue by calculating the Laplace transform in some examples.
That will allow us to deduce the Laplace transform for the compound Poisson forest from Example 
(\ref{E2.1}) and we see the \Levy-Khintchine formula for this example.

\begin{proposition}[Laplace transforms of i.i.d.~concatenations]\label{p2907131212}
Let $M \in \N_0$ be a random variable with probability generating function $G_M(s)=\sum_{k=0}^\infty p_k s^k$, $s\in [0,1]$. Moreover, let 
$(\mfU_i)_{i\in\bbN}$ be an i.i.d.~sequence of random $h$-trees such that 
$M\independent(\mfU_i)_{i\in\bbN}$.
Let $\mfU:= \bigsqcup_{i=1}^M\mfU_i$. Then, for 
all $\Phi\in\Pi_+$,
  \begin{equation}\label{e.tr278}
    L_\mfU(\Phi_{\hu})=G_M(L_{\mfU_1}(\Phi))\,.
  \end{equation}
\end{proposition}

\begin{proof}\label{p.2774}
Using Proposition~(\ref{p2907131206}) and the independence between $\mfU_0$ and $(\mfU_i)_{i\in\bbN}$ 
we obtain
\begin{align}\label{e.tr279}
  \bbE\left[\exp(-\Phi_{\hu}(\mfU))\right]
  &=\bbE\left[\exp\Bigl(-\sum_{i=1}^M\Phi(\mfU_i)\Bigr)\right]
  =\bbE\left[\bbE\bigl[\exp(-\Phi(\mfU_1))\bigr]^M\right]
  =G_M(L_{\mfU_1}(\Phi))\,.
  \end{align}
\end{proof}

\begin{proposition}[Laplace transform of CPF]\label{p1308131621}
  Let $\mfP$ be a $\textup{CPF}_h(\theta,\lambda)$. Then, for all $\Phi\in\Pi_+$,
  \begin{equation}\label{e.tr280}
    -\log L_\mfP(\Phi_{\hu})= \theta\int \big(1- e^{-\Phi(\mfu)}\big)\,\lambda(\dx\mfu).
  \end{equation}
\end{proposition}

\begin{proof}\label{p.2792}
By Proposition~(\ref{p2907131212}) and inserting the generating function of Poiss 
$(\theta)$
gives the claim. 
\end{proof}

\subsection{Stochastic order and applications}\label{ss.stochord}

\medskip

Recall that in $h$-forests the notion of sub-trees induces a partial order, see Definition 
(\ref{d.po}). This partial order induces a stochastic partial order for random $h$-forests. For 
a general treatment of stochastic orders we refer to \cite{kamae1977} and \cite{stoyan1983comparison}. Let 
$f\in\mathrm{bm}\mcB(\bbU)$ the latter denotes the measurable bounded and monotone functions on $\bbU$.

\begin{definition}\label{d:StochOrder}
 Let $h>0$. Suppose $\mfU$ and $\mfV$ are random variables taking values in $\U(h)^\sqcup$. Then we 
say that $\mfU \preccurlyeq_h \mfV$ if $\bbE[f(\mfU)]\leq \bbE[f(\mfV)]$ for all 
$f\in\mathrm{bm}\mcB(\bbU(h)^\sqcup)$.
\end{definition}

Our first remark is implied by classical results in \cite{kamae1977}.

\begin{remark}\label{r.2815}
Suppose $\mfU_n$ is an 
$h$-subtree in $\mfV_n$ for each $n$ and $\mfU_n\Longrightarrow \mfU$, $\mfV_n\Longrightarrow 
\mfV$
then, there are $\mfU'\eqd\mfU$, $\mfV'\eqd\mfV$ on a common probability space such that $\mfU'$ 
can be embedded as a subtree into $\mathfrak{V}^\prime$. \end{remark}

\begin{proposition}[Tightness via domination]\label{p:OrderTight}
 If $\{ \mfU_n: \, n \in \N\} \subseteq \U$ is tight and $\mfV_n \preccurlyeq_h \mfU_n$ for all $n \in \N$ for some $h>0$,
then $\{\mfV_n: \, n \in \N\}$ is tight.
\end{proposition}

\begin{proof}[Proof of Proposition \ref{p:OrderTight}]\label{p.p:OrderTight}
By Theorem 1 in \cite{kamae1977} we may assume that, for each $n$, $\mfU_n$ and $\mfV_n$ are 
defined on the same probability space and $\mfV_n \leq \mfU_n$ almost surely. 
We denote by 
\begin{equation}\label{e2934}
w_\mfX \langle \cdot \rangle =\mu^{\otimes 2} \left(\{r(x,x^\prime) \in \cdot ; x,x^\prime \in X\}\right) \text{  for  } \mfX=\left[X,r,\mu \right].
\end{equation}
By Proposition \eqref{p.pre-comp} (compare also  \cite{GPW09}[Proposition 8.1]) we have to show that the total masses are tight and 
that for each $\eps>0$ there exists $\delta>0$ and $C>0$ such that 
\begin{equation}\label{e:TightVW}
 \sup_n \bbP[v_\delta(\mfV_n)+w_{\mfV_n}([C,\infty))]<\eps\,.
\end{equation}
But obviously $w_{\mfV_n}([C,\infty))\leq w_{\mfU_n}([C,\infty))$ almost surely and also by 
Lemma~(\ref{l.2.7EM}) 
$v_\delta(\mfV_n)\leq v_\delta(\mfU_n)$. By \cite{GPW09}[Proposition 8.1], \eqref{e:TightVW}
holds for $\{\mfU_n\}$ and hence also for $\{\mfV_n\}$. As for the total mass, obviously 
$\bar{\mfV}_n\leq \bar{\mfU}_n$ almost surely and hence tightness follows from the tightness 
of $\{\bar{\mfU}_n\}$.
\end{proof}

\begin{example}\label{ex.2845}
  Note that the partial order on $\bbU(h)^\sqcup$ induces a 
partial order on the path space $D([0,\infty),\bbU(h)^\sqcup)$ in the following way:
\begin{equation}\label{e2119}
(\mfu_t)_{t\geq0}\leq (\mfv_t)_{t\geq0}\quad \Longleftrightarrow\quad \mfu_t\leq\mfv_t\;\forall 
t\geq0\,.
\end{equation}
Hence we also obtain a partial stochastic order on $\mcM_1(D([0,\infty),\bbU(h)^\sqcup))$. Hence by 
the same reasoning as before if $\mfU=(\mfU_t)_{t\geq0}$, 
$\mfV=(\mfV_t)_{t\geq0}$, $\mfU^n=(\mfU^n_t)_{t\geq0}$, $\mfV^n=(\mfV^n_t)_{t\geq0}$, 
$n\in\bbN$ are stochastic processes taking values in $\bbU(h)^\sqcup$ and $\mfU^n\preccurlyeq 
\mfV^n$ for all $n$. 
By Strassen's result Lemma 13 in \cite{Str65}, $\mfU\preccurlyeq \mfV$ and we can find a probability space $ (\Omega,\CA, \bbP) $ 
such that $\bbP\{\mfU_t\leq \mfV_t\;\forall t\geq0\}=1$. That is we have an almost 
surely path-wise embedding of the dominated process into the dominating process.
A typical example is a branching process starting in the zero tree but with
different masses, as we shall see in the next section. 
\end{example}

\subsection{Proof of Theorem~\ref{t:LimInfDiv}}\label{ss.prth239}
We now prove Theorem (\ref{t:LimInfDiv}) which says that weak limits of infdiv random trees are again 
infdiv and moreover the L\'{e}vy measures converge.

\begin{proof}[Proof of Theorem \ref{t:LimInfDiv}]\label{p.t:LimInfDiv}
By assumption for each $m$ and $n$ there is $\mfU^{(n)}_m$ taking values in $\bbU(h)^\sqcup$ such 
for all $\Phi\in\Pi_+$ 
\begin{equation}\label{e:Lapl1}
 \bbE\left[\exp(-\Phi_h(\mfU_m))\right]
 =\left(\bbE\left[\exp(-\Phi_h(\mfU^{(n)}_m))\right]\right)^n\,.
\end{equation}
Also,
\begin{equation}\label{e:Lapl2}
 \lim_{m\to\infty}\bbE\left[\exp(-\Phi_h(\mfU_m))\right]\to \bbE\left[\exp(-\Phi_h(\mfU))\right]\,.
\end{equation}
Note that, by Skorohod embedding, we can choose $\mfU^{(n)}_m$ on the same probability space as 
$\mfU_m$ and $\mfU^{(n)}_m\leq \mfU_m$ almost surely. Hence, by Proposition 
(\ref{p:OrderTight}) tightness of $\{\mfU_m:m\in\bbN\}$ implies tightness for 
$\{\mfU^{(n)}_m:m\in\bbN\}$. So let $(m_k)$ be a sub-sequence and assume that 
$(m_{k_l})$ is a further sub-sequence which converges to some $\mfU^{(n)}$ so that 
\begin{equation}\label{grx144b}
 \lim_{l\to\infty}\bbE\left[\exp(-\Phi_h(\mfU^{(n)}_{{m_{k_l}}}))\right]\to 
\bbE\left[\exp(-\Phi_h(\mfU^{(n)}))\right]\,.
\end{equation}
Taking into account \eqref{e:Lapl1} and \eqref{e:Lapl2}, 
\begin{equation}\label{grx145}
 \lim_{l\to\infty}\bbE\left[\exp(-\Phi_h(\mfU^{(n)}_{{m_{k_l}}}))\right]\to 
\left(\bbE\left[\exp(-\Phi_h(\mfU))\right]\right)^{1/n}
\end{equation}
and hence
\begin{equation}\label{grx146}
\bbE\left[\exp(-\Phi_h(\mfU^{(n)}))\right] 
= \left(\bbE\left[\exp(-\Phi_h(\mfU))\right]\right)^{1/n}\,.
\end{equation}
This shows that all convergent sub-sequences of sub-sequences have the same limit which implies 
convergence of the sequence itself 
, that is, for each 
$n$,
\begin{equation}\label{grx147}
 \operatorname*{w-lim}_{m\to\infty}\mfU^{(n)}_m=\mfU^{(n)}
\end{equation}
It also shows that $\mfU$ is infinitely divisible and $\lambda^{(m)}_h\Longrightarrow
\lambda_h$ as $ m \to \infty $ as boundedly finite measures, see Corollary \eqref{e:t-crit:tmass}.
\end{proof}

\section{Proof of infinite divisibility and related results}
\label{s.lkfproof}

This section proves the L\'evy-Khintchine formula in the first subsection and the further
claims from Subsection (\ref{ss.infdiv}) related to it.

\subsection{Proof of total mass results}\label{ss.prtotmas}

We first prove Proposition (\ref{p.mass:infdiv}).
We use for $ (b) \Rightarrow (a) $ facts which are proved further below.

\begin{proof}[Proof of Proposition \ref{p.mass:infdiv}]\label{p.p.mass:infdiv}
\mbox{}

\eqref{mass.infdiv1}
 This is obvious if we consider the polynomials $\Phi(\mfu) = \lambda \bar{\mfu}$ for $\lambda >0$.
 Then \eqref{r10b} tells us that we may find that the variable $\bar{\mfU}^{(h,n)}$ is such that the 
total mass $\bar{\mfU}$ can be written as the sum of $n$ i.i.d.~copies of $\bar{\mfU}^{(h,n)}$.

\eqref{mass.infdiv2} Recall for a measure $\nu$ we denote by $\bar \nu$ the total mass and by $\wh \nu$ the normalized measure.
By \cite{Klenke}[Satz 16.5] there are $\nu^{(n)}\in \mcM_f(\bbR_+)$ such that, 
if we denote 
$\theta^{(n)}:=\nu^{(n)}(\bbR_+)$,
\begin{equation}\label{grx142}
\CL [X] = \operatorname{w-lim}_{n\to\infty}  (\hat{\nu}^{(n)})^{\ast 
\operatorname{Poiss}(\theta^{(n)})}\,.
\end{equation}
Recall the definition of compound Poisson forest from Example (\ref{E2.1}) and let with $a \otimes \mfe$ being the $\delta$-measure on the element of $\U$ arising by multiplying the mass of $\mfe \in \U_1$ with $a > 0$:
\begin{equation}\label{grx143}
 \mfU^{(n)}:=\left(CPF(\nu^{(n)}(\bbR_+),\bar \nu^{(n)}\otimes\mfe)\right)_h
\end{equation}
Then $\bar{\mfU}^{(n)} \Rightarrow X$.
Moreover, $\nu^{(n)} \Rightarrow \nu$ as boundedly finite measures, with $\nu$ the L\'evy-measure of $X$:
\begin{equation}\label{e2903}
-\log \E[ \exp(-\lambda X)] = \int_0^\infty \nu(\dx x)\, (1-e^{-\lambda x}) , \ \lambda \geq 0.
\end{equation}
By assumption $\int \nu(\dx x) (1\wedge x) < \infty$.

We prove below that $(\mfU^{(n)})_{n\in \N}$ is tight.
Let $\mfV$ be a weak limit point of $\mfU^{(n)}$.
Since $\mfU^{(n)}$ is $t$-infinitely divisible for any $n \in \N$, we know that $\mfU$ is $t$-infinitely divisible by Theorem (\ref{t:LimInfDiv}) and the L\'evy measures of the $\mfU^{(n)}$  converge.
Thus,
\begin{align}\label{e2120}
 -\log \E[\exp (-\Phi(\mfU))] & =\int \nu(\dx x) (1-\exp(-\Phi(x \mfe))) \, .
\end{align}
If we use the polynomial of order 1, $\Phi=\Phi^{1,\lambda}$ we get
\begin{align}\label{e2121}
 -\log \E[\exp (- \lambda \bar \mfU)] & =\int \nu(\dx x) (1-\exp(-\lambda x )) \, ,
\end{align}
hence $\bar{\mfU} \eqd X$, that as we wanted to show.

Finally we prove that $(\mfU^{(n)})_{n\in \N}$  is tight.
By Proposition (\ref{p.tight-crit}) we need to verify equations  \eqref{e:t-crit:1-e} and \eqref{e:t-crit:tmass}.
Let $\eps >0$.
For $C=2h+1$ we have $\nu^{2,\mfU^{(n)}} ([2h+1,\infty)) = 0$.
The existence of an $M$ in \eqref{e:t-crit:tmass} follows using  that $\bar{\mfU}^{(n)}$ is tight (see \eqref{grx143}).
So we are left with the modulus of continuity i.e. \eqref{e:t-crit:1-e} we calculate:
\begin{align}\label{e2122}
 \nn \int P(\mfU^{(n)}  & \in \dx \mfu) \, \left( 1- \exp(-\mu^{\mfu} (x: \mu^{\mfu}(B_{2h}(x)) <\delta) ) \right) \\
 \nn & = 1- \exp \left( - \theta^{(n)} \int_0^\infty \hat{\nu}^{(n)}(\dx x) \, \1(x < \delta) (1-e^{-x}) \right) \\
 \nn &\leq \theta^{(n)} \int_0^\infty \hat{\nu}^{(n)}(\dx x) \, \1(x < \delta) (1 \wedge x) \\
 & \to \int_0^\infty \pi(\dx x) \, \1(x< \delta) (1\wedge x) \ \text{(as $n\to\infty$)}.
\end{align}
And since $(1\wedge x) \pi(\dx x)$ is a finite measure on $(0,\infty)$, we may choose $\delta$ so small that the last expression is less than the given $\eps$.
So we have verified \eqref{e:t-crit:1-e} and we know that $(\mfU^{(n)})_{n\in \N}$  is tight.
\end{proof}

\subsection{The L\'evy-Khintchine formula (Proof of Theorem~\ref{T.LK})}\label{ss.levyform}

We now have to prove the L\'evy-Khintchine formula. In this section we will denote the law of the random tree $\mfU$ by $P \in \mcM_1(\U)$ and that of its $n$-th root at depth $h>0$ i.e. $ \mfU_h^{(1,n)} $ by $P_n^h \in \mcM_1(\U(h)^\sqcup)$ and hence:
\begin{equation}\label{e.4.1}
 \int P(\dx \mfu) \exp(- \Phi_h (\mfu)) = \left( \int P_n^h(\dx \mfu) \exp(- \Phi_h (\mfu)) \right)^n \, , \ \forall \Phi \in \CA  (\Pi_+) \, .
\end{equation}


We consider $ nP_n^h $, a sequence of boundedly finite measures and prove it converges to a limit, the excursion law.
With this fact we conclude later the proof in \eqref{grx139d}-\eqref{grx140b} easily.
Hence the key to the proof is to show {\em tightness} of $ \{nP^h_n, n \in \N\} $ and the {\em uniqueness} of the limit points.
\\

It will be necessary to use tightness criteria for sequences of measures on $\U\;\mbox{or}\; \U\setminus\lbrace0\rbrace$ and to define what we mean by weak convergence of boundedly finite measures; we refer the reader to Section (\ref{s.tightness}). The strategy is to show first the {\em tightness} of measures on $\U(h)\setminus\lbrace0\rbrace$ (in the Gromov weak $ ^\#-$topology, see Appendix \eqref{s.tightness}) and then the {\em uniqueness} of a {\em limit point} in two steps.
\\

{\bf Step 1} (Tightness) 
\\
We want to establish the existence of an excursion of the process arising following the path of $ h-$truncated states.
A step is the tightness of the marginal distributions.

\begin{lemma}[Tightness]\label{l.Q_n:tight}
 Assume $P \in \mcM_1(\U)$ and \eqref{e.4.1}. Then the sequence $(\1(\cdot \neq 0) nP_n^h)_{n\in \N} \subset \mcM^\#(\U\setminus \{0\})$ is tight (in the weak $ ^\#-$topology (see Proposition \eqref{p.tight-crit})) and $\limsup_{n\to \infty} \int nP_n^h(\dx \mfu) \, (\bar \mfu \wedge 1) < \infty $. Moreover, for any $\eps \in (0,1)$ there exists an $M(\eps)>0$ such that
 \begin{equation}\label{e.tr27} \limsup_{n\to \infty} \int nP_n^h(\dx \mfu) \, \1(\bar{\mfu} > \eps) < M(\eps) \, . \end{equation}
\end{lemma}
\begin{proof}[Proof of Lemma~\ref{l.Q_n:tight}]\label{p.l.Q_n:tight}
We use the tightness criterion of Proposition (\ref{p.tight-crit}).
\\
\\
First, we verify the claim that $\limsup_{n\to \infty} \int nP_n^h(\dx \mfu) \, (\bar \mfu \wedge 1) < \infty $: We know that $\bar{\mfU}$ is a non-negative infinitely divisible random variable by Proposition~(\ref{p.mass:infdiv}). Therefore there exist $c_1 \geq 0$ and a measure $\nu$ on $(0,\infty)$ with $\int_0^\infty (1\wedge h) \, \nu(dh) <\infty$, s.t. (see \cite{D93}, Theorem 3.3.1).
\begin{equation}\label{grx114}
 \bar{\mfU} \eqd c_1 + \int N^\nu (\dx h) h \, ,
\end{equation}
where $N^\nu$ is a PPP on $(0,\infty)$ with intensity measure $\nu$. Then
\begin{align}\label{grx115}
 \int nP_n^h(\dx \mfu) \, (\bar{\mfu} \wedge 1) & = n \E\left[ \bar{\mfU}_1^{(n)} \wedge 1 \right] = 
 n \E\left[ \left(\frac{c_1}{n} + \int N^{\nu/n} (\dx h) h \right) \wedge 1 \right] \\
 & \leq n \E\left[ \frac{c_1}{n} + \int N^{\nu/n} (\dx h) \, (h\wedge 1) \right]
 = c_1 + n \E\left[\int N^{\nu/n} (\dx h) \, (h\wedge 1) \right] \nonumber \\
 & = c_1 + n \int \frac{1}{n} \nu(\dx h) (h \wedge 1) = c + \int \nu(\dx h) (h\wedge 1) < \infty \, .\nonumber
\end{align}
\\
 We can also deduce that
\begin{equation}\label{grx116}
 \limsup_{n\to \infty} \int nP_n^h(\dx \mfu) \, \1(\bar{\mfu} > \eps) \leq \eps^{-1} \limsup_{n\to \infty} \int nP_n^h(\dx \mfu) \, (1\wedge \bar{\mfu}) < \infty,
\end{equation}
which establishes \eqref{e:t-crit:tmass} of Proposition~(\ref{p.tight-crit}) and \eqref{e.tr27}.

 We need to verify the other condition of that proposition.
We use $ \mu^\mfu $ (instead of just $ \mu $) if we want to stress the dependence on the equivalence class $ \mfu $.
  Choosing $C= 2h+1$ yields $\int nP_n^h(\dx \mfu) \, (1- \exp(-\nu^{2,\mfu}([C,\infty)) ) = 0$. Moreover, for $h>0$ and $\delta >0$:
 \begin{align}\label{grx117}
 \limsup_{n\to\infty} \int & nP_n^h(\dx \mfu)  \, \1(\mfu \neq 0) \left( 1- \exp(- \mu(x: \mu(B_{h}(x)) < \delta)) \right)  \\
 & =  \limsup_{n\to \infty} n \int P_n^h( \dx \mfu)  \left( 1- \exp(- \mu^\mfu(x:\, \mu^\mfu(B_{h}(x)) < \delta) ) \right) \nonumber\\
    &= \limsup_{n\to \infty} n \left(  1- \left( \int P( \dx \mfu) \exp(- \mu^\mfu(x:\, \mu^\mfu(B_{h}(x)) < \delta)) \right)^{1/n} \right) \nonumber \\
   &= -\log \int P( \dx \mfu) \exp(- \mu^\mfu(x:\, \mu^\mfu(B_{h}(x)) < \delta)) \, . \nonumber
 \end{align}
Since the law $P(\dx \mfu) \in \mcM_1(\U)$ is tight, for any $\eps$ we can find a $\delta=\delta(\eps)$ such that by Proposition~(\ref{p.tight-crit}):
\begin{equation}\label{grx118}
 \int P( \dx \mfu) \exp(- \mu^\mfu(x:\, \mu^\mfu(B_{h}(x)) < \delta)) > 1- (1-\exp(-\eps)) \, .
\end{equation}
Then,
\begin{equation}\label{e.tr17}
  \limsup_{n \to \infty}  \int nP_n^h(\dx \mfu) \, \1(\mfu\neq 0) \left( 1- \exp(- \mu^\mfu(x:\, \mu^\mfu(B_{\eps}(x)) < \delta) \right) \leq
  - \log \exp(-\eps) = \eps \, .
\end{equation}
A similar argument based on the representation in \eqref{e.tr27} allows to show the tightness of the masses $ \bar \mfu \text{  in  } \R^+ $ by showing that $ nP^h_n(\bar u > M) < \infty $.
Altogether, $\mathbbm{1}_{\{\cdot\neq0\}}nP_n^h$ is a tight sequence on $\mcM^\#(\U(h)^\sqcup\setminus \{0\})$.
 \end{proof}
{\bf Step 2} (Uniqueness) 
\\
Now we need to show that there is only {\em one} limit point of the sequence $(nP_n^h \1(\cdot \neq 0))_{n \in \N}$. 
The final proof is then done with the next Lemma (\ref{l.lim.exc}) below. Since the excursion measure is a measure on truncated trees we need first some preparation to get its uniqueness, recall here  Theorem~(\ref{p.trLap}).

For the uniqueness problem, we need to show that for any polynomial $\Phi \in \CA(\Pi_+)$ the Laplace transforms coincide. For $\underline{m} = (m_1,\dotsc, m_l) \in \N^l$ (repetitions allowed) let 
\begin{equation}\label{grx120}
 [0,2h)^{\binom{\underline{m}}{2}} = [0,2h)^{\binom{m_1}{2}}\times \dotsc \times [0,2h)^{\binom{m_l}{2}} \, 
\end{equation}
and $\bbD_{\underline{m}} = \bbD_{m_1+\dotsc + m_l} \cap [0,2h)^{\binom{\underline{m}}{2}}$. 
To evaluate a function $\underline{\phi}: \bbD_{\underline{m}}(h)  \to \R$ we need a vector of measures. We write a vector $\underline{\nu}$ of measures in the following form:
\begin{equation}\label{grx122}
 \underline{\nu} = (\nu_1,\dotsc, \nu_l): \begin{cases} \mcB(X_1) \times \cdots \times \mcB(X_l) \to \R_{\geq 0}^l \\
                               (A_1, \dotsc , A_l) \mapsto (\nu_1(A_1), \dotsc, \nu_l(A_l))
                              \end{cases}
\end{equation}
where $l \in \N$ is fixed and $X_1, \dotsc, X_l$ are Polish and $\nu_i \in \mcM_f(X_i), \, 1\leq i \leq l$. We call the set of such objects $\mcM(X_1, \dotsc, X_l) = \mcM(X_1) \times \cdots \mcM(X_l)$. If we exclude the case that the measure attains the value zero in any coordinate we write
\begin{equation}\label{grx123}\begin{array}{l}
 \mcM(X_1,\dotsc, X_l)^o = \{ \nu \in \mcM(X_1, \dotsc, X_l) : \nu_1 \neq 0, \dotsc, \nu_l \neq 0\} =  \\
 \hspace{6cm}(\mcM(X_1)\setminus \{0\}) \times \cdots \times \mcM(X_l)\setminus \{0\}) \, .
\end{array}
\end{equation}

 Define $ \nu^{\underline{m},\mfu} = (\nu^{m_1,\mfu},\dotsc, \nu^{m_l,\mfu}) , \mfu \in \U \, $ and use the following notation for a polynomial
\begin{equation}\label{e.tr8}
 \Phi^{\underline{m},\underline{\phi}}(\mfu) = \sum_{i=1}^l \Phi^{m_i,\phi_i}(\mfu) = \sum_{i=1}^l \lan \phi_i , \nu^{m_i,\mfu} \ran = \lan \underline{\phi}, \nu^{\underline{m},\mfu} \ran \, .
\end{equation}
Even though the above looks close to the desired result we have to realize that this does not mean we have this on the level of elements in $\U$ yet.
\begin{lemma}[Representation of sample Laplace-functional]\label{L.mfU1}
 Assume $P \in \mcM_1(\U)$ and \eqref{e.4.1}. Then for any $\underline{m} \in \N^l$, $\underline{\phi} \in C(\bbD_{\underline{m}}(h),\R_+)$, there exists $\lambda_{\underline{m},h} \in \mcM(\mcM(\bbD_{m_1},\dotsc, \bbD_{m_l})^o)$ s.t.
\begin{equation}\label{g3}\begin{split}
  -\log \E[\exp(-\lan \underline{\phi}, \nu^{\underline{m},\mfu} \ran) ] 
&=  \int \lambda_{\underline{m},h}(\dx \underline{\nu}) \,(1- e^{-\lan \underline{\phi},\underline{\nu} \ran}) \\
&=\lim_{n \to \infty} \int (1- e^{-\lan\underline{\phi}, \nu^{\underline{m},\mfu} \ran }) \, n P_n^h(\dx \mfu) \, .
\end{split}\end{equation}
\end{lemma}

\begin{proof}\label{p.3045}
Let $\mfU$ be the realization of a random variable with law $P(\cdot)$ and $\mfU_h^{(i,n)}, \, 1\leq i \leq n$, i.i.d.~copies of random elements in $\U(h)^\sqcup$ with law $P_n^h(\cdot)$. Then by \eqref{e.4.1}:
\begin{equation} \mfU (h) \eqd \mfU_h^{(1,n)} \sqcup \dots \sqcup \mfU_h^{(n,n)} . \end{equation}
 That means for any $\underline{m} \in \N$:
\begin{equation}\label{grx124}
 \nu^{\underline{m},\mfU}|_{[0,2h)^{\binom{\underline{m}}{2}}} \eqd \nu^{\underline{m},\mfU_h^{(1,n)}}|_{[0,2h)^{\binom{\underline{m}}{2}}} + \dots + \nu^{\underline{m},\mfU_h^{(n,n)}}|_{[0,2h)^{\binom{\underline{m}}{2}}}.
\end{equation}
Thus the measure vector $\nu^{\underline{m},\mfU_h}$ restricted to $[0,2h)^{\binom{\underline{m}}{2}}$ is infinitely divisible and by 
the extension to {\em vector measures} of Proposition 6.1 in \cite{Kall83}, see Section 3.1. in \cite{GR91}, 
there exists $\pi_h^{\underline{m}} \in \mcM([0,2h)^{\binom{m_1}{2}}, \dotsc , [0,2h)^{\binom{m_l}{2}})$ and $\rho_h^{\underline{m}} \in \mcM(\mcM([0,2h)^{\binom{m_1}{2}}, \dotsc , [0,2h)^{\binom{m_l}{2}})^o)$ with
 \begin{equation}\label{e.LK:1}
   - \log P \left( \exp(-\lan \underline{\phi}_h, \nu^{\underline{m}, \cdot} \ran) \right) = \lan \underline{\phi},  \pi_h^{\underline{m}} \ran +\int \rho_h^{\underline{m}}(\dx \nu) \, (1-e^{-\lan \underline{\phi}, \nu \ran})  .
 \end{equation}
 We have to show that the linear term vanishes.

In Proposition 6.1 of \cite{Kall83} it is also stated how to calculate $\pi_h^{\underline{m}}$ using the law $P_n^h$. Since we want to show that $\pi_h^{\underline{m}} = 0$ we directly use $\phi_k(\dr) = \1(r_{ij}<h), \, 1\leq k \leq l$. If $\mfu \in \U$ we write $\lfloor\mfu\rfloor(h) = \sqcup_{i\in I} \mfu_i$ for $\mfu_i \in \U(h)$ and a countable index set $I$. First note that the following inequality holds for $\delta \in (0,1)$ and $m \in \N$:
 \begin{equation}\label{grx125}
  \sum_{i\in I} \bar{\mfu}_i \prod_{j\in I} \1(\bar{\mfu}_j < \delta^{1/m}) \leq \mu(x: \, \mu(B_{2h}(x)) < \delta^{1/m}) \, .
 \end{equation}
 Suppose $m_1 = \min m_k$ and $m_2 = \max m_k$. Then we can calculate starting from Kallenberg' s formula (given next) as follows:
\begin{align}\label{e.tr18} 
 \lan \phi, \pi_h^{\underline{m}} \ran & = \lim_{\delta \to 0} \limsup_{n \to \infty} 
  \int  n P_n^h(\dx \mfu) \, \sum_{k=1}^l \left(1 \wedge \sum_{i\in I} \bar{\mfu}_i^{m_k} \right) \1( 0 < \sum_{i\in I} \bar{\mfu}_i^{m_k} < \delta) & \\
 \nn \leq 2 l \lim_{\delta \to 0} & \limsup_{n\to \infty} 
  \int n P_n^h(\dx \mfu) \, \left(1- \exp(- (1\wedge  \sum_{i\in I} \bar{\mfu}_i^{m_1} )) \right) \1( 0 < \sum_{i\in I} \bar{\mfu}_i^{m_2} < \delta) & \\
 \nn \leq 2 l \lim_{\delta \to 0} & \limsup_{n\to \infty}  
  \int n P_n^h(\dx \mfu) \, \left(1- \exp(-  (1\wedge \sum_{i\in I} \bar{\mfu}_i^{m_1} )) \right) \1(\mfu \neq 0) \prod_{i\in I} \1(  \bar{\mfu}_i^{m_2} < \delta) & \\
 \nn \leq 2 l \lim_{\delta \to 0} & \limsup_{n\to \infty} 
  \int n P_n^h(\dx \mfu) \, \left(1- \exp(- (1\wedge (\sum_{i\in I} \bar{\mfu}_i )^{m_1})) \right) \1(\mfu \neq 0) \prod_{i\in I} \1(  \bar{\mfu}_i^{m_2}< \delta) & \\
\nn \leq 2 l \lim_{\delta \to 0} & \limsup_{n\to \infty}  
 \int n P_n^h(\dx \mfu) \, \1(\mfu \neq 0) & \\
& \left[1- \exp\left(- \left(1\wedge (\mu(x:\, \mu(B_{2h}(x))< \delta^{1/m_2}))\right)^{m_1} \right) \right] \, . & \nonumber
 \end{align}
 As we have seen in the proof of Lemma~(\ref{l.Q_n:tight}) in \eqref{e.tr17} for any $\eps >0$ we can choose $\delta$ so small that \eqref{e.tr18} is less than $\eps$ uniformly in $n \in \N$. Therefore $\pi_h^{\underline{m}} = 0$.

Then use \eqref{e.LK:1} to get:
\begin{align}\label{grx127}
 \int \rho_h^{\underline{m}}(\dx \underline{\nu}) (1-e^{-\lan \underline{\nu},\underline{\phi} \ran}) & = - \log \int P(\dx \mfu) \, \exp( - \lan \underline{\phi}_h, \nu^{\underline{m},\mfu} \ran ) \\
  & = \lim_{n\to \infty} n \left(1- \left(\int P(\dx \mfu) \, \exp( - \lan \underline{\phi}_h, \nu^{\underline{m},\mfu} \ran )\right)^{1/n} \right) \label{grx127b}\\
 &\stackrel{\eqref{e.4.1}}{=} \lim_{n \to \infty} \int (1- e^{-\lan \underline{\phi}, \nu^{\underline{m},\mfu} \ran }) \, n P_n^h(\dx \mfu) \, .\label{grx127c}
\end{align}
This shows the two parts of the statement.
\end{proof}
Next comes the existence of the excursion measure.
\begin{lemma}[Limit of excursion measure]\label{l.lim.exc}
 Let $P \in \mcM_1(\U)$ and assume \eqref{e.4.1}. The sequence of  measures $(\1(\cdot \neq 0)nP_n^h)_{n \in \N}$ on $\U(h)^\sqcup \setminus \{0\}$ converges in the weak$^\#$ topology to a boundedly finite measure $\lambda_h \in \mcM^\# (\U(h)^\sqcup \setminus \{0\})$ with $\int \lambda_h(\dx \mfu) \, (1\wedge \bar{\mfu}) < \infty$ and for $\Phi \in \Pi_+$:
 \begin{equation}\label{e.tr11}
  -\log \E[\exp(-\Phi_h(\mfU)) ] = \int \lambda_h(\dx \mfu) \,(1- e^{-\Phi_h(\mfu)}) . 
 \end{equation}
\end{lemma}
\begin{proof}\label{p.3097}
  By Lemma~(\ref{l.Q_n:tight}) $(\1(\cdot \neq 0)nP_n^h)_{n\in \N}$ is tight. Assume that there are two limit laws $\lambda_h, \, \lambda_h' \in \mcM^\#(\U(h)^\sqcup)$. By Fatou's lemma they obey $\int \lambda_h(\dx \mfu) \, (1\wedge \bar{\mfu}) < \infty$ and $\int \lambda_h'(\dx \mfu) \, (1\wedge \bar{\mfu}) < \infty$. Then by Lemma~(\ref{L.mfU1})
 \begin{equation}\label{grx128}
  \int \lambda_h(\dx \mfu) \, \left(1- \exp ( - \Phi_h (\mfu) ) \right)  = \int \lambda_h'(\dx \mfu) \, \left( 1- \exp ( - \Phi_h (\mfu) ) \right),
 \end{equation}
 for all polynomials $\Phi \in \CA(\Pi_+)$. This class of functionals is separating on measures in $\mcM^\#(\U(h)^\sqcup)$ satisfying the integrability criterion by Proposition~(\ref{c.tr2}). Therefore the limit $\lambda_h$ is unique and \eqref{e.tr11} follows from \eqref{g3}.
\end{proof}

If we want to identify the cases in which the measure $\lambda_{h}$ is finite, the following observation is helpful:

\begin{lemma}[Total weight \Levy -measure]\label{L.mfU}
Let $P \in \mcM_1(\U)$ and assume \eqref{e.4.1}. Then 
\be{gl3}
\liml_{n \to \infty} n P_n^h( \mfu  \neq 0) = -\log P(\mfu =0 ) = \int \lambda_h(\dx \mfu) \, 1 \in [0,\infty] \, ,
\ee
which is finite iff $P( \mfu = 0)>0$.
\end{lemma}
\begin{proof}[Proof of Lemma~\ref{L.mfU}]\label{p.L.mfU}
Let us set $c_n := n P_n^h( \mfu \neq 0 )$. With \eqref{e.4.1} we calculate:
 \begin{equation}\label{e.P:U=0}\begin{split}
  P(\mfu = 0 ) & = P(\lfloor\mfu\rfloor(h) = 0) = \left( P_n^h(\mfu = 0) \right)^n = \left( 1- \tfrac{1}{n} c_n ) \right)^n \, .
 \end{split}\end{equation}
This implies
\begin{equation}\label{g1}
 c_n=n(1-(P(\mfu = 0 ))^{1/n})\overset{n\to\infty}{\longrightarrow} -\log P(\mfu = 0)\in [0,\infty ] \, .
\end{equation}
To get the second equality set now $ m = 1 $ and $ \phi \equiv a $ to get :
 \begin{align}\label{e.tr283}
-\log \; \E[(\exp(-a \; \bar \mfU)]
= \int \lambda_h (\dx \mfu) \, (1-e^{-a\bar{\mu}})  \, .
 \end{align}

Letting $a \to \infty$ we get the claim.
\end{proof}

\begin{lemma}[Consistency of $\lambda_h$]\label{l.lam.cons}
 Consider $\lambda_h \in \mcM^\#(\U(h)^\sqcup \setminus\{0\})$ of Lemma~(\ref{l.lim.exc}) for $h>0$. Then for $0< h' < h$:
 \begin{equation}\label{grx130}
  \lambda_{h'} (\dx \mfu) = \int \lambda_h(\dx \mfv) \1(\lfloor \mfv \rfloor (h') \in \dx \mfu) \, .
 \end{equation}
\end{lemma}

\begin{proof}[Proof of Lemma~\ref{l.lam.cons}]\label{p.l.lam.cons}

 For $\Phi \in \Pi_+$ it is true that $(\Phi_{h'})_{h} = \Phi_{h'}$ since $h'<h$. Therefore,
 \begin{equation}\label{grx131}
   \E[\exp(-(\Phi_{h'})_h(\mfU)) ] = \E[\exp(-\Phi_{h'}(\mfU)) ] \, ,
 \end{equation}
 which by \eqref{e.tr11} implies that
 \begin{equation}\label{grx132}
  \int \lambda_h(\dx \mfv) \,(1- e^{-\Phi_{h'}(\mfv)})  = \int \lambda_{h'}(\dx \mfu) \,(1- e^{-\Phi_{h'}(\mfu)})  \, .
 \end{equation}
 But we can rewrite the left hand side in this equation to get 
 \begin{equation}\label{grx133}
  \int \lambda_h(\dx \mfv) \, (1-e^{-\Phi_{h'}(\lfloor\mfv\rfloor(h'))}) = \int \lambda_{h'}(\dx \mfu) \,(1- e^{-\Phi_{h'}(\mfu)})  \, ,
 \end{equation}
 since $\Phi_{h'}(\mfv) = \Phi_{h'}(\lfloor\mfv\rfloor(h'))$ for $\mfv \in \U(h)^\sqcup$. Use Corollary~(\ref{c.tr2}) to deduce the claim.
\end{proof}

For the special case $t=\infty$ we need to establish the existence of $\lambda_\infty$ and we will do that in the next lemma.
\begin{lemma}[Existence of $\lambda_\infty$]\label{l.ar1}
 In the case that $t =\infty$, the sequence $(\lambda_h)_{h\in (0, \infty)}$ of Lemma~(\ref{l.lim.exc}) has for $ h \to \infty $  a unique limit $\lambda_\infty \in \mcM^\#(\U \setminus \{0\})$.
\end{lemma}
\begin{proof}\label{p.3164}
 We need to verify that the sequence is tight and that the set of limit points contains only a single object. Let us first do the uniqueness and assume there are $\lambda_\infty,\, \lambda_\infty' \in \mcM^\#(\U\setminus \{0\})$ with two sequences $h_n \nearrow \infty$ and $h_n'\nearrow \infty$ as $n\to \infty$:
 \begin{equation}\label{grx134}
  \lambda_\infty = \lim_{n\to \infty} \lambda_{h_n} \, , \quad \lambda_\infty' = \lim_{n\to \infty} \lambda_{h_n'} \, .
 \end{equation}
Then for $H>0$ and $n$ sufficiently large with Lemma (\ref{l.lam.cons}),
\begin{align}\label{grx135}
 \int \lambda_\infty(\dx \mfu) \, \1(\lfloor\mfu\rfloor(h) \in \cdot ) & = \lim_{n\to \infty} \int \lambda_{h_n} \1( \lfloor\mfu\rfloor(h) \in \cdot) \\
 & = \lambda_H(\cdot) \label{grx135b}\\
 &=  \lim_{n\to \infty} \int \lambda_{h_n'} \1( \lfloor\mfu\rfloor(h) \in \cdot) \nonumber \\
 & = \int \lambda_\infty'(\dx \mfu) \, \1(\lfloor\mfu\rfloor(h) \in \cdot ) \, .\nonumber
\end{align}
But this means that both measures coincide on $\U(H)^\sqcup$ for any $H>0$ and that suffices since then expectations of all polynomials $\Phi^{m,\phi}$ (which by definition has compactly supported $\phi$) coincide.

It remains to show the tightness of the sequence $(\lambda_h)_{h>1}$ using Proposition (\ref{p.tight-crit}); $1$ was chosen arbitrarily.
First, $v_\delta(\mfu,h') = v_{\delta}(\mfu(1),h')$ gives with Lemma \eqref{l.lam.cons} that for $h'<2$ and $h>1$:
\begin{align}\label{grx136}
 \sup_{h>1}\int \lambda_h (\dx \mfu) \, (1-e^{-v_\delta(\mfu,h')}) & = \sup_{h>1} \int \lambda_h(\dx \mfu) \, (1-e^{-v_\delta(\mfu(1),h')}) \\
 & = \int \lambda_{1} (\dx \mfv) \, (1-e^{-v_\delta(\mfv,h')})  \, .\label{grx136b}
\end{align}
The measure $\lambda_{1}$ as a single measure is tight and therefore it allows for any $\eps>0$ to choose $\delta$ such that the last quantity is bounded by $\eps$. Since $\bar{\mfu} = \bar{\mfu}(1)$ we can show \eqref{e:t-crit:tmass} and the only thing left to show is the first part of \eqref{e:t-crit:1-e}. Note that by \eqref{ag1c}, Lemma (\ref{l.2.7EM}) \eqref{i.2.7EM.a} for $h>0$ :
\begin{equation}\label{grx137}
  \nu^{2,\mfU}([2h,\infty)) \geq  \sum_{i=1}^{N(\lambda_h)} \nu^{2,\mfu_i}(\{2h\}) \, \text{  in stochastic order}. 
\end{equation}
Hence the exponential transforms satisfy:
\begin{align}\label{grx138}
 \E \left[ \exp (- \nu^{2,\mfU}([2h,\infty)) ) \right] & \leq \E \left[ \exp (-\sum_{i=1}^{N(\lambda_h)} \nu^{2,\mfU_i}(\{2h\}) ) \right] \\
 & = \exp \left( - \int \lambda_h(\dx \mfu) \, (1-e^{-\nu^{2,\mfu}(\{2h\}) } ) \right) \, .\label{grx138b}
\end{align}
Using this for $h=N$ allows to derive the following inequalities:
\begin{align}\label{grx139}
 \sup_{h>1} \int \lambda_h(\dx \mfu) \, \left( 1- e^{-\nu^{2,\mfu}([2N,\infty))} \right) & = \sup_{h\geq 2N} \int \lambda_h(\dx \mfu) \, \left( 1- e^{-\nu^{2,\mfu}([2N,\infty))} \right) \\
 & = \int \lambda_{2N} (\dx \mfu) \, \left( 1- e^{-\nu^{2,\mfu}(\{2N\})} \right) \label{grx139b}\\
 & \leq -\log \E \left[ \exp \left(- \nu^{2,\mfU}([2N,\infty))\right) \right] \label{grx139c}\\
 & \leq - \log (1-\eps) \leq \eps \, .\label{grx139d}
\end{align}
 The next to last inequality holds, since $\mcL[\mfU]$ is tight, allowing to choose $N$ sufficiently large.
\end{proof}

Finally, we can give a proof of the {\em \Levy-Khintchine representation}.

\begin{proof}[Proof of Theorem~\ref{T.LK}]\label{p.T.LK}
We have shown the main things already, we just put them together again. Let $\Phi = \Phi^{\underline{m},\underline{\phi}}$, see \eqref{e.tr8} be a polynomial. Then,
\begin{align}\label{grx140}
 -\log \E[\exp(-\Phi_h(\mfU))] &= \lim_{n\to \infty} n \left(1- \left(\int P(\dx \mfu) \, \exp( - \lan \underline{\phi}_h, \underline{\nu}^{\mfu} \ran )\right)^{1/n} \right) \\
 &\stackrel{\eqref{e.4.1}}{=} \lim_{n \to \infty} \int (1- e^{-\lan \phi, \nu^{m,\mfu} \ran }) \, n P_n^h(\dx \mfu) \label{grx140b}\\
 & = \int \lambda_h(\dx \mfu) \, (1-e^{-\lan \phi, \nu^{m,\mfu} \ran }) \, .\label{grx140c}
 \end{align}
where the last equation holds by Lemma~(\ref{l.lim.exc}) which also states that $\int \lambda_h(\dx \mfu) \, (1\wedge \bar{\mfu}) < \infty$. Furthermore, by Lemma~(\ref{L.mfU}): $\int \lambda_h(\dx \mfu) \, 1 = - \log P(\mfu = 0 )$. In the case $t= \infty$, Lemma (\ref{l.ar1}) provides the existence of a unique $\lambda_\infty$.
\end{proof}

\subsection{Proofs of related results}\label{ss.further}

\begin{proof}[Proof of Corollary \ref{c.PCR}]\label{p.c.PCR}
 Denote by $\mfV$ the right hand side of \eqref{ag1c}. We calculate the truncated Laplace transform of $\lfloor \mfV \rfloor (\hu)$. Use first Proposition~(\ref{p2907131206}) and then that $N^{\lambda_h}$ is a Poisson process:
 \begin{align}\label{grx141}
  \E[\exp(-\Phi_h(\lfloor\mfV\rfloor(\hu)))] & = \E[ \exp (- \Phi_h(\bigsqcup_{\mfu \in N^{\lambda_h}} \mfu ) ) ] \\
  & = \E[ \exp (- \sum_{\mfu \in N^{\lambda_h}} \Phi_h(\mfu) ) ] \label{grx141b}\\
  & = \exp (- \int \lambda_h(\dx \mfu) \, (1-e^{-\Phi_h(\mfu)})) \, .\label{grx141c}
 \end{align}
But this equals the Laplace transform of $\lfloor \mfU \rfloor (\hu)$. Since Theorem~(\ref{p.trLap}) tells us that Laplace transforms uniquely determine the law restricted to $\U(h)^\sqcup$, we can conclude that $\mfV(\hu) \eqd \mfU(\hu)$.
\end{proof}

If $\mfu\in\bbU$ is fixed and $\mu\in\mcM_1(\bbR_+)$, then if $X\sim\mu$ we  write $\mu\otimes \mfu$ 
for the law of the random measure arising as $X \cdot$ sampling measure of $\mfu$, which means the measure on $\U$ induced by taking the element $\mfu$ and multiplying its mass with the random variable $X$.



\begin{lemma}[Tightness of compound Poisson forests]\label{l:TightPCF}
 Let $h>0$. Assume that $\theta^{(n)}\in\bbR_+$ and $\nu^{(n)}\in\mcM_1(\bbR_+)$ such that
$\limsup_{n\to\infty}\theta^{(n)}<\infty$ and $\{\nu^{(n)}\}$ is tight.
Let $\mfe$ be an element of $\U$ with total mass $1$.
Then the family 
$\{(CPF(\theta^{(n)},\nu^{(n)}\otimes\mfe))_h,n \in \N\}$ is tight. 
\end{lemma}

\begin{proof}\label{p.3261}
 Assume that $CPF((\theta^{(n)},\nu^{(n)}\otimes\mfe))=[U_n,r_n,\mu_n]$. It is obvious that the total masses of the collection of CPF's is tight. Hence by the tightness criterion in Theorem 3 in \cite{GPW09},  we only have to show that 
 \begin{equation}\label{e:TightPCF}
 \lim_{\delta\downarrow0} 
\sup_n\bbE\left[\inf\{\eps>0:\mu\{x\in U_n:\mu_n(B_\eps(x))\leq\delta\}\leq\eps\right]=0\,.
 \end{equation}
Let $X^{(n)}_i\sim$ i.i.d. $\nu^{(n)}$ and $N^{(n)}\sim\Poiss(\theta^{(n)})$ and 
$\independent_i X_i^{(n)}\independent N$ for each $n$. Denote the density of 
$\Poiss(\theta^{(n)})$ by $(p^{(n)}_k)_{k\in\bbN_0}$. Then, we have
\begin{align}\label{grx144}
 &\bbE\left[\inf\{\eps>0:\mu\{x\in U_n:\mu_n(B_\eps(x))\leq\delta\}\leq\eps\right]\\
 &\leq \bbE\left[\inf\{\eps>0:\mu\{x\in U_n:\mu_n(B_h(x))\leq\delta\}\leq\eps\right] \label{grx144bb}\\
 &\leq \bbE\left[\mu\{x\in U_n:\mu_n(\{x\})\leq\delta\}\right] \label{grx144c}\\
 &=\sum_{k=0}^\infty p^{(n)}_k 
 \bbE\left[\sum_{i=1}^k X^{(n)}_i 
\mathbbm{1}\{X^{(n)}_i\leq\delta\}\right] \label{grx144d}\\
&\leq \sum_{k=0}^\infty p^{(n)}_k k \delta \label{grx144e}\\
&=\theta^{(n)}\delta \label{grx144f}
\end{align}
This implies \eqref{e:TightPCF}.
\end{proof}

We are now in a position to prove Theorem (\ref{THM:INFDIV:BRAN}).
\begin{proof}[Proof of Theorem \ref{THM:INFDIV:BRAN}]\label{p.THM:INFDIV:BRAN}
 Let $n \in \N$. Suppose that $\mfV \sim \pi$ and $\mfV_n$ takes values in $\U(h)^\sqcup$ such that $L_\mfV(\Phi) = (L_{\mfV_n}(\Phi))^n$ for all $\Phi \in \CA(\Pi_h)$.
 We know that the kernel $Q_t$ has the branching property, i.e.~fulfills \eqref{e.tr223}.
 By Theorem (\ref{p.trLap}) this is equivalent to a relation for Laplace transforms.
 Using this statement we obtain for $\Phi \in \CA(\Pi_{t+h})$:
 \begin{align}\label{e2124}
  L_{\mfU}(\Phi) & = \int \P(\mfV \in \dx \mfv) \, \int Q_t(\mfv, \dx \mfu) \, \exp(-\Phi(\mfu)) & \\
   & = \int \P(\mfV_n \in \dx \mfv_1) \cdots \int \P(\mfV_n \in \dx \mfv_n) \, \int Q_t(\mfv_1 \sqcup^h \dots \sqcup^h \mfv_n, \dx \mfu) \, \exp(-\Phi(\mfu)) & \label{e2124b}\\
   & = \int \P(\mfV_n \in \dx \mfv_1)  \, \int Q_t(\mfv_1, \dx \mfu_1) \, \exp(-\Phi(\mfu_1)) \cdots \int \P(\mfV_n \in \dx \mfv_n)& \label{e2124cc} \\
 & \qquad \int Q_t(\mfv_1, \dx \mfu_n) \, \exp(-\Phi(\mfu_n)) & \label{e2124c}\\
   & = \left( \int \P(\mfV_n \in \dx \mfv_1)  \, \int Q_t(\mfv_1, \dx \mfu_1) \, \exp(-\Phi(\mfu_1)) \right)^n & \label{e2124d}\\
   & = \left( L_{\mfU_n}(\Phi) \right)^n \,& , \label{e2124e}
 \end{align}
 if we set $\mfU_n$ as the um-space in $\U(t+h)^\sqcup$ with distribution given by $\int \P(\mfV_n \in \dx \mfv_1)  \, \int Q_t(\mfv_1, \dx \cdot)$.
\end{proof}

\begin{proof}[Proof of Proposition \ref{p.IV.4.1}]\label{p.p.IV.4.1}
To be self-contained we give the proof of this statement on infinitely  divisible measures.
 Let $P,Q\in \mcM_1(\U(h)^\sqcup)$ be infinitely divisible with L\'{e}vy measures 
 $\lambda_h^P$ and $\lambda_h^Q$. Then for all $\Phi \in \Pi_+ $,
\begin{align}\label{e2125}
 &-\log( P\ast^h Q)[ \exp(-\Phi_h(\cdot))]\\
 &=-\log \int_{\U(h)^\sqcup} P(\dx\mfu)\int_{\U(h)^\sqcup} 
Q(\dx\mfv)\exp(-\Phi_h(\mfu\sqcup^h\mfv)) \label{e2125b}\\
 &=-\log \int_{\U(h)^\sqcup} P(\dx\mfu)\exp(-\Phi_h(\mfu))\int_{\U(h)^\sqcup} 
Q(\dx\mfv)\exp(-\Phi_h(\mfv)) \label{e2125c}\\
 &=-\log \int P(\dx\mfu)\exp(-\Phi_h(\mfu))
 -\log\int Q(\dx\mfv)\exp(-\Phi_h(\mfv)) \label{e2125d}\\
 &=\int_{\U(h)^\sqcup \setminus \{0\}} (1-e^{-\Phi_{\hu}(\mfu)}) 
\lambda^P_h (\dx \mfu) +
 \int_{\U(h)^\sqcup \setminus \{0\}} (1-e^{-\Phi_{\hu}(\mfu)}) 
\lambda^Q_h (\dx \mfu) \label{e2125e}\\
&=\int_{\U(h)^\sqcup \setminus \{0\}} (1-e^{-\Phi_{\hu}(\mfu)}) 
(\lambda^P_h+ \lambda^P_h)(\dx \mfu)\,.\label{e2125f}
\end{align}
Hence $P\ast^h Q$ is infinitely divisible with $h$-L\'{e}vy measure 
$\lambda^P_h+ \lambda^P_h$.
\end{proof}

Finally, we prove Proposition (\ref{t:branchorder}).

\begin{proof}[Proof of Proposition \ref{t:branchorder}]\label{p.t:branchorder}
By assumption there is $\mfw\in\bbU(h)^\sqcup$ such that $\mfv=\mfu\sqcup^h \mfw$. 
Denote the semigroup by $Q_t$. Then for all $f\in\mathrm{bm}\mcB(\bbU)$ (bounded and $ \mcB-$measurable), 
\begin{align}\label{e2126}
 \bbE_\mfv[f(\mfV_t)]
 &=\int Q_t(\mfv,\dx\tilde{\mfv})f(\tilde{\mfv})
 =\int Q_t(\mfu\sqcup^h \mfw,\dx\tilde{\mfv})f(\tilde{\mfv})\\
 &=\int (Q_t(\mfu,\cdot)\ast^h Q_t(\mfw,\cdot))(\dx\tilde{\mfv})f(\tilde{\mfv}) \label{e2126b}\\
 &=\int Q_t(\mfu,\dx\tilde{\mfu})\int Q_t(\mfw,\dx\tilde{\mfw}) 
f(\tilde{\mfu}\sqcup^h\tilde{\mfw}) \label{e2126c}\\
 &\geq \int Q_t(\mfu,\dx\tilde{\mfu})\int Q_t(\mfw,\dx\tilde{\mfw}) f(\tilde{\mfu})\label{e2126d}\\
 &=\int Q_t(\mfu,\dx\tilde{\mfu})f(\tilde{\mfu})\label{e2126e}\\
 &=\bbE_\mfu[f(\mfU_t)]\,.\label{e2126f}
\end{align}
This shows that $\mfU_t \preccurlyeq^h \mfV_t$. Now, the claim follows from Theorem 1 in \cite{kamae1977}.
\end{proof}

\subsection{Proof of extensions}\label{ss.proext}

We have already seen in Section~\ref{ss.marked} that the basic concepts we use in this paper, the \emph{concatenation} and the \emph{truncation} carry over to spatial models and we need now to verify that these operations have the same algebraic and topological structure.
We have to prove first that we get topological semigroups.
Then the second point is to prove the \Levy-Khintchine formula.
For that we have to show that all properties needed to obtain the \Levy-Khintchine representation hold to then verify the formula.
Altogether we proceed in two subsubsections.

\subsubsection{Proof of Proposition~\ref{p.tr2}}
In order to obtain the case where we have now \emph{marked} metric measure spaces consider first the case where $ \mu $ is a {\em finite}  measure and later we generalize this to the general case the argument  based on the finite one.  

In the finite measure case, note first the $ h$-truncation, the concept of $ h-$marked forests and trees refers to the genealogical part only and not the mark, we have only replaced $(U,r)$ by $(U \times V,r \otimes r_V)$ and $\mu \in \CM(U,\CB(\U))$ by $\nu \in \CM(U \times V, \CB (U \times V))$ in the definitions.
Furthermore concatenation involves the marks only via the fact that now two measures on $ U \times V $ rather than $ U $ are in focus, where however $V$ is a \emph{fixed} object for all elements of our state space $\U^V$.
Using that measures on a fixed measure space are a topological group and that the projections on $\U$ are topological groups one verifies in a straight forward way that we have again a topological semigroup.
Therefore the strategy for the $\U$-valued case can be used to get the marked case of the proposition, where essentially (a) has to be proven.
Namely we have to return to Section~\ref{ss.concsemi}, where Theorem~\ref{p.delphic}, the non-spatial version of our present claim is proved and see by inspection that we can repeat these arguments with the given observations for the lifted objects.

We need here only one extra information, namely a \emph{marked} compactness criterion in proving the conditions for a Delphic semigroup.
The compactness in the marked case requires in addition that the measures of the subset of $\U$ in question projected on the marks are tight, see Theorem 3 in \cite{DGP11}.
Here we talk about measures being all bounded by one given finite measure.
Hence tightness is immediate.

This means that {\em all} arguments carry over if we modify the statements on the concatenation as indicated in Section~\ref{ss.marked}.

\begin{remark}\label{r.3531}
Once we have the proof of the Proposition~\ref{p.delphic}, then we can see that the decomposition is also obtained via \emph{lifting}.
Namely apply the kernel $\kappa$ to the $\mu_i$ arising in the genealogical decomposition of $[U,r,\mu]$ we obtained on $\U$, where $\kappa$ is the transition probability from $U$ to $V$ induced by $\nu$ on $U \times V$ by $\mu_i=\mu |_{U_i} \otimes \kappa$.
Once we have proved the marked version of our statement, then we know we have exactly this representation of the decomposition.
\end{remark}

In the case where we consider $ \mu $ which are not finite on $U \times V $, we have assumed $ \mu $ is boundedly finite, i.e. finite on sets $ U \times A $, with $ A $ being a bounded mark set.
In particular do we have the following approximation with elements of $ \U $ with $ \mu $ finite.
We consider the restriction of the state $ \mfU $ to $ U \times V_n $ denoted $ \mfU_n $, with $ V_n \uparrow V $ and $ V_n $ finite resp. bounded.
Then the restricted random states $ \mfU_n $ fit our theory as explained in the previous paragraph.

Since the object in $\U^V$ can be identified with convergent sequences of the elements in $\U^{V_n}$ namely the restriction of the set $U \times V$ to elements $(u,v)$ with $u \in U$ and $v \in V_n$, the result carries then over using as well the definition of the convergence (convergence of polynomials).

\subsubsection{Proof of Theorem~\ref{T.resmark}}
We first have to argue first for the Propositions preparing the \Levy-Khintchine representation.

We have here the Propositions~\ref{p.trLap},~\ref{T.LK},~\ref{t:LimInfDiv} which collect the properties of truncated polynomials and the corresponding properties of the Laplace-transform.
Note we consider here polynomials based on $\varphi \cdot \chi$ where $\varphi$ depends on the distances and $\chi$ on the marks.
Due to this product structure, we can use the results we have on the non-spatial case to lift them for $\chi \equiv const$ to the marked case and similarly for $\varphi \equiv const$, the well known statements for measure-valued states in $\CM(V,\CB(V))$ can be lifted to our situation.
Hence we have to argue that the extension can be done for the \emph{joint} distribution of marks and distances.

Here we note that the $h$-truncation we consider here is affecting only the distances and not the marks, which are locations or types.
Therefore transferring the propositions to the marked case is straight forward and suppressed here.

\begin{remark}\label{r.3536}
Occasionally it is useful to use marks which depend on the ultrametric structure explicitly (as for example \emph{ancestral path} of individuals).
For example marking individuals by ancestral path, which depend explicitly on genealogical information.
This interesting but complicated situation is not touched here, as it would require a different form of concatenation and truncation.
This will be coming up in \cite{ggr_GeneralBranching}.
\end{remark}

Now we turn to the \Levy-Khintchine representation and we start with elements from $\U^V$ for $V$ a \emph{bounded} set and hence we have \emph{finite} measures $\nu$.

We can now consider the projection of $\mfU$ onto $\U$ by $[U \times V,r \otimes r_V,\mu \otimes \kappa] \to [U,r,\mu]$.
For the image we obtain from our results the \Levy-Khintchine representation via a measure $\lambda_h$ resp. $\lambda_\infty$ on $\U(h) \setminus \{0\}$ respectively $\U \setminus \{0\}$.
This representation we have to lift now to $\U^V$.

What is the additional structure we have to deal with?
We have to bring into play the infinite divisibility of the sampling measure on $U \times V$.
Clearly the projection onto $V$ leads to a random measure which is infinitely divisible for all $h$-truncations (the projection is the same for all $h$) and has a \Levy-Khintchine representation by the classical theory of random measures see \cite{Kall83}, but we have to obtain the \emph{joint} distribution of marks and genealogy.
However we can follow the steps of our proof on $\U$ also here very closely.

Return to the proof in Section~\ref{ss.levyform} and go through the argument. 
We just replace polynomials, Laplace-transforms on $\U$ by the ones on $\U^V$ and the measure $nP^n_h$, from which we obtained $\lambda_h$ as a limit for $n \to \infty$ before, is now a measure in $\CM^\#(\U^V \setminus \{0\})$.
The tightness statements require now, as we saw above as additional criterion a condition on the mark component, namely the projections of the measures on the marks need to be tight.
However for finite measure we have the same structure of addition and order the same arguments work here as in Section~\ref{ss.concsemi}.
Note also for random measures this representation is well known, see \cite{Kall83}.
Since we consider here as mark space finite or bounded sets which are fixed and the measures finite this condition is satisfied.
Then there is no obstacle to repeat the proof step by step replacing $U$ by $U \times V$ and $\U \setminus \{0\}$ by $\U^V \setminus \{0\}$.
We leave further details to the reader.

Having the \Levy-Khintchine representation for this case we can continue with the general case.
Next we note that by construction of $ \U^V $ in the case of infinite sampling measures a state is nothing else than the \emph{sequence} of all its \emph{restrictions} to a sequence of bounded mark  spaces exhausting the full space in fact they converge to the state $ \mfU $, which allows to handle the remaining claim, the \Levy -Khintchine formula.

The restriction is defined by mapping
\begin{equation}\label{e3533}
\left[U \times V, r \otimes r_V,\nu \right] \to \left[U \times V,r \otimes r_V,\nu |_{V_n}\right].
\end{equation}
The image can be mapped $1-1$ and isometric and measure preserving onto 
\begin{equation}\label{e3625}
\left[U \times V_n, r \otimes r_{V_n},\nu_n \right]
\end{equation}
where $r_{V_n}$ is defined as restriction of $r_V$ to $V_n \times V_n$ and $\nu_n$ is the image measure of $\nu|_{V_n}$.
For the latter we can \emph{apply} the previous results on the marked case with bounded mark sets and \emph{finite} sampling measures.

Namely by the very definition of the topology, the $ \mfU_n $ approximate the $ \mfU $, see Section 1.2. in \cite{GSW}.
Therefore consider the corresponding \Levy - measures $ \lambda_h^n $, more precisely corresponding to the populations $ \mfU_n $ in the finite measure spaces to corresponding $U \times V_n $.
They give elements of $\U^{V_n}$.
They can be extended to elements in $\U^V$(we embed for that the $n$-population from $U_n \times V_n$ in $U_n \times V$).
We have to show that these elements converge as $ n \to \infty $ to a limit measure $ \lambda^\infty_h $ on the space $ \U^V \setminus \{0\}$.
The convergence is w.r.t. the weak$^\#-$topology of measures on $ \U^V \setminus \{0\}$.

This convergence takes place since $ (\lambda^n_h)_{n \in \N} $ form a \emph{projective family} on $ \U^V $.
Consequently both sides of the \Levy - Khintchine representations for given $ n $ converge to a limit which gives the \Levy - Khintchine representation of the $ \U^V-$valued random variable via the limit measure $ \lambda^\infty_h $.

%


\newpage

\appendix
\appendixpage        
\addappheadtotoc     %

\section{Ultrametric measure spaces}\label{s.umspaces}

Our random variables take values in the space of metric measure spaces see \cite{GPW09}, later also in marked metric measure spaces, first introduced in \cite{DGP11} based on \cite{GPW09} and generalized in \cite{GSW}. 

We briefly review definitions and topological facts used. 

\begin{definition}[Topology of state space]\label{pkg:d:230413:1741}
  
  \begin{enumerate}[(i)]
    \item We call $(X,r,\mu)$ a \emph{metric measure space} \emph{(mm space)} if
    \begin{enumerate}[(a)]
      \item $(X,r)$ is a complete separable metric space,
      \item $\mu$ is a finite measure on the Borel subsets of $X$. 
    \end{enumerate}
    \item We define an equivalence relation on the collection of mmm spaces as follows: two mm spaces $(X,r_X,\mu_X)$ and $(Y,r_Y,\mu_Y)$ are \emph{equivalent} if and only if there exists a measurable map $\vphi:X\to Y$ such that
        \begin{equation}\label{a.3429}
          r_X(x_1,x_2)=r_Y(\vphi(x_1),\vphi(x_2))\quad \forall x_1,x_2\in \supp(\mu_X)
        \end{equation}
        and
        \begin{equation}\label{a3430}
          \mu_Y=\mu_X\circ \vphi^{-1} \, ,
        \end{equation}
  i.e.~$\vphi$ restricted to $\supp(\mu_X)$ is an isometry onto its image and  $\vphi$ is measure preserving.

        We denote the equivalence class of an mm space $(X,r_X,\mu)$ by $[X,r_X,\mu]$,
	if it does not create confusion we
	refer to $[X,r_X,\mu]$ as an mm space too.
        \item We denote the collection of (equivalence classes of) mm spaces by
        \be{gl1}
          \bbM:=\{[X,r,\mu]:(X,r,\mu)\text{ is mm space}\}.
        \end{equation}
        We use Gothic type letters $\mfx,\mfy,\dotsc$ to denote generic elements of $\bbM$.
  \end{enumerate}
\end{definition}

If $(X,r,\mu)$ is an mm space, we interpret $X$ as the set of individuals, $r$ as genealogical distance and, after normalization,
 $\hat{\mu} := \bar{\mfu}^{-1} \mu$ as the sampling measure.

%
%

We next define a topology on $\bbM$. The main idea is to extend the Gromov-weak topology from \cite{GPW09} 
to the setting of finite measures on $X$, see Section 2.4. in \cite{Gl12} for details.

\beD{D.gromov}{(Gromov-weak-topology)}
%

  We say that a sequence $(\mfx_n)_{n\in\bbN}$ of elements from $\bbM$ converges to $\mfx\in\bbM$ in the \emph{Gromov-weak topology} if and only if
  \begin{equation}\label{e:pkg:e:230413:2232}
    \Phi(\mfx_n) \to \Phi(\mfx)
  \end{equation}
  for any $\Phi \in \Pi$, defined above \eqref{rg8}.
\end{definition}

\begin{remark}\label{ar.3478}
 The topology of Gromov-weak convergence is equivalent to the convergence of the distance measures and can be metrized by the Gromov-Prokhorov metric for finite sampling measures. This metric is given by
 \begin{equation}\label{e.dGPr}
  d_{\text{GPr}}(\mfx,\mfy) = \inf_{\phi_X,\phi_Y,Z} d_{\text{Pr}}^Z((\phi_X)_\ast \mu_X, \phi_Y)_\ast \mu_Y)) \, ,
 \end{equation}
 where $d_{\text{Pr}}^Z$ is the Prokhorov distance of finite measures on the metric space $Z$.
 The metric space $(\bbM, d_{\text{GPr}})$ is complete and separable.
 This can be shown as in Section 5 of \cite{GPW09} and several variants of the topology can be found in \cite{ALW14a}, \cite{Gl12}.
\end{remark}
We are especially interested in the set of ultrametric measure spaces.
\begin{definition}\label{ad.3488}
  We say that the metric measure space $(U,r,\mu)$ is \emph{ultrametric} if $r(u, w) \leq r(u, v)  \vee r(v, w)$ for all $u,v,w \in U$ except on a set of $\mu$-measure $0$.
\end{definition}
Definition (\ref{D.umsp}) can now be rephrased as:
  \begin{equation}\label{e.tr53} \bbU =\{ \mfu \in \bbM :\, \mfu \text{ is ultrametric} \} \, .
  \end{equation}
The set $\bbU$ is a closed subset of $\bbM$ and therefore $\bbU$ is a Polish space, see Lemma 2.3 in \cite{GPWmp13}.

\section{Boundedly-finite measures in um-spaces: tightness, separation}\label{s.tightness}

Recall the metric $d_{\textup{GPr}}$ on $\U$ from \eqref{e.dGPr} and use the notation $\hat{\mfu} = [U,r,\bar{\mfu}^{-1} \mu]$ for the normalized space of the metric measure space $\mfu = [U,r,\mu]$.
The Polish space $\U \setminus \{0\}$ can be metrized by 
\begin{equation}\label{e.tr35} \tilde{d}(\mfu,\mfv) = d_{GPr}(\hat{\mfu},\hat{\mfv}) + | \bar{\mfu}^{-1} - \bar{\mfv}^{-1}| + \left( 1\wedge |\bar{\mfu} -\bar{\mfv}| \right) \end{equation}
such that the point $\mfu = 0$ is ``infinitely far away''.
Following Appendix 2.6 of \cite{DVJ03} we call $\mcM^{\#}(\U \setminus \{0\})$ the set of \emph{boundedly finite measures} on $\U \setminus \{0\}$, i.e.~$\mu \in {\mcM}^{\#}(\U \setminus \{0\})$ if $\mu (A) < \infty$ for all 
bounded measurable sets $A \subset \U \setminus \{0\}$, where boundedness is measured w.r.t.~$\tilde{d}$.

We say that a sequence $(\mu_n)_{n\in \N} \subset \mcM^{\#}( \U \setminus \{0\})$ \emph{converges in the weak}$^\#$-topology to $\mu \in \mcM^{\#}(\U \setminus \{0\})$ if $\mu_n(f) \to \mu(f)$ for all continuous $f:\U \setminus \{0\} \to \R$ with bounded support. 

We also need to recall the set of boundedly finite counting measures $\mcN^\#(\U \setminus \{0\})$ on $\U \setminus \{0\}$ which consists of those elements in $\mcM^\#(\U(h)\setminus\{0\})$ taking integer values on bounded sets.

\subsection{Tightness}\label{ss.tigthness}

As in \cite{GPW09} we want to establish a criterion for relatively compact subsets of $\U$ and afterwards a suitable tightness criterion for finite measures on $\U$ and boundedly finite measures on $\U\setminus \{0\}$.
The main difference to their article is that we are working with {\em finite} sampling measures instead of {\em probability}  measures.
Since the notation is consistent with that of \cite{GPW09} if we only consider probability measures on the metric spaces we chose not to change notation.
However, we extend the notation a bit in allowing an additional parameter $h$ in the next definition.

\begin{definition}[Modulus of mass distribution]\label{ad.modmass}
Let $\mfu=(U,r,\mu)\in \U$. For $\delta>0$, define the \emph{modulus of mass distribution} as 
\begin{equation} \label{modul}
 v_{\delta}(\mfu,h) =  \mu\left\{x\in U:\,
   \mu(B_{2h}(x))\leq \delta \right\}  \, .
\end{equation}
\end{definition}

Using the same steps as in \cite{GPW09}, it is not very difficult to establish the following analogue of Proposition 7.1 in this  article in the non-normalized setting, see also Remark 2.5 in \cite{DGP11}. 
\begin{proposition}[Relative compactness characterization] \label{p.pre-comp}
Let ${\Gamma} \subset \U$. The set $\Gamma$ is relatively compact in the Gromov-Prokhorov topology if and only if the family $\left\{ \nu^{2,\mfu};\, \mfu \in \Gamma \right\} \subset \mcM_f([0,\infty))$ is tight (in $ \CM_f $), and for all $h>0$
\begin{equation}\label{e.tr54}
   \sup_{\mfu\in\Gamma} v_\delta(\mfu,h) \xrightarrow{\delta \to 0} \, 0.
\end{equation}
Moreover $\Gamma \subset \U \setminus \{0\}$ is relatively compact w.r.t.~$\tilde{d}$ if $\; \Gamma$ is relatively compact w.r.t.~$d_{\text{GPr}}$ and $\sup_{\mfu \in \Gamma} \bar{\mfu}^{-1} < \infty$.
\end{proposition}
\begin{remark}\label{R.pre-comp}
 In the case we need to give a compact subset $A$ of $\U(h)^\sqcup$ for a fixed $h>0$, it suffices to verify that for any $h'\in (0,h),\, \eps >0$ we can find a $\delta(h',\eps)>0$ s.t.~$\sup_{\mfu\in A}v_{\delta}(\mfu,h') < \eps$ and that $\sup_{\mfu \in A} \bar{\mfu} < \infty$. The condition on the diameter is obsolete since it is bounded by $2h$.
\end{remark}

Now we give a {\em tightness characterization} for {\em boundedly finite measures} on $\U\setminus \{0\}$.
We write $m(f) $ for the integral w.r.t. the measure $m$ of the function $f$.

\begin{proposition}[Tightness of boundedly finite measures on $\CM^\#(\U \setminus\{0\})$ with respect to the Gromov-Prohorov topology]\label{p.tight-crit}
  A set $A \subseteq{\mcM}^{\#}(\U \setminus \{0\})$ is relatively compact in the weak$^\#$-topology with
  respect to $\tilde{d}$ on $\U\setminus \{0\}$
  if and only if for all $\eps>0$ and $h>0$ there exist $\delta>0$, $C>0$ and $M>0$ such that
 \begin{align}
\label{e:t-crit:1-e} & \sup_{m \in A} m [ 1-  \exp(- \nu^{2,\mfu}([C,\infty))) +1 -\exp ( -\mu^\mfu(x:\,\mu^\mfu(B_{2h}(x))<\delta) ) + \1( \bar{\mfu}>M) ] < \eps \,
\end{align}
and
\begin{align}
\label{e:t-crit:tmass} & \sup_{m \in A} m [ \mfu: \ \bar{\mfu} \geq \eps] \leq M \, . 
\end{align}
\end{proposition}

\begin{proof}\label{ap.3582}
 Let $S_n = \{\mfu \in \U:\, \bar{\mfu} \geq n^{-1}\}$.\\
 ``$\Rightarrow$'': By \cite[Proposition A2.6.IV]{DVJ03} we know that $\{m|_{S_n} : \, m \in A\}$ is tight as a family of finite measures on $S_n$.
 First, $\sup_{m\in A} m(S_n)< \infty$ for any $n\in \N$ which directly implies \eqref{e:t-crit:tmass}.
 Additionally, \eqref{e:t-crit:1-e} follows from a similar argument as in \cite{GPW09} and Proposition (\ref{p.pre-comp}).\\
 ``$\Leftarrow$'':
 Let $\eps >0$. For any $n\in \N$ choose $\delta_n,\, C_n$ and $M<\infty$ such that
 \begin{align}\label{a.3431}
  \sup_{m \in A} m [ 1- \exp(- \nu^{2,\mfu}([C,\infty))) ] <\eps^2 2^{-2n-1},\\
  \sup_{m \in A} m [ 1- \exp(- \mu^\mfu(x:\,\mu^\mfu(B_{2h}(x))<\delta_n)) ] <\eps^2 2^{-2n-1} \label{a.3431b},\\
  \sup_{m \in A} m [ \mfu: \, \bar{\mfu} > M] < 2^{-1} \eps \label{a.3431c},\\
  \sup_{m \in A} m [\mfu: \, \bar{\mfu} > \eps ] \leq M. \label{a.3431d}
 \end{align}
Then by Markov inequality
\begin{align}\label{a.3432}
 m[ \{ \nu^{2,\mfu}([C,\infty)) > \eps 2^{-n} \} ] & \leq m [ \{ 1- \exp ( - \nu^{2,\mfu}([C,\infty)) > 1- \exp(- \eps 2^{-n} ) \} ] \\
 & \leq \left( 1- \exp(- \eps 2^{-n}) \right)^{-1} m[ 1- \exp ( - \nu^{2,\mfu}([C,\infty)) ] \label{a.3432b}\\
 & \leq \left( \eps 2^{-n-1} \right)^{-1} \eps^2 2^{-2n-1} = \eps 2^{-n} \,\label{a.3432c}
\end{align}
for all $m \in A$. With the same proof:
\begin{equation}\label{a.3433}
 m[ \{  \mu^\mfu(x:\,\mu^\mfu(B_{2h}(x))<\delta_n) > \eps 2^{-n} \} ] \leq \eps 2^{-n} \, .
\end{equation}
Define the set
\begin{equation}\label{a.3434}
 \Gamma_\eps = \{\mfu: \, \bar{\mfu} \leq M \} \cap \bigcap_{n =1}^\infty \{\mfu: \, \nu^{2,\mfu}([C,\infty)) + \mu^\mfu(x:\,\mu^\mfu(B_{2h}(x))<\delta_n) \leq \eps 2^{-n} \} 
\end{equation}
which is a pre-compact set by Proposition~(\ref{p.pre-comp}), since $\nu^{2,\mfu}$ is uniformly bounded by $M^2$ on this set, for any $\eps' >0$ one can find a $n \in \N$ s.t.~$\eps 2^{-n} < \eps'$. Choosing $C' = C_n$ tells that $\nu^{2,\mfu}([C',\infty))<\eps'$ uniformly. For the same $n\in \N$ take $\delta = \delta_n$ to obtain uniformly
\begin{equation}\label{a.3435}
 v_{\delta'} (\mfu,h) = v_{\delta_n}(\mfu,h) \leq \inf \{\eps'' \geq \eps 2^{-n} : \, \mu^\mfu(x:\,\mu^\mfu(B_{2h}(x))<\delta_n) < \eps'' \} \leq \eps 2^{-n} < \eps' \, 
\end{equation}
uniformly in $\mfu$. Finally, we can show
\begin{align}\label{a.3436}
 m(\U \setminus \Gamma_\eps) & \leq m(\mfu: \, \bar{\mfu} > M) + m(\bigcup_{n=1}^\infty \{ w_\mfu([C_n,\infty)) + \mu^\mfu(x:\,\mu^\mfu(B_\eps(x))<\delta_n) > \eps 2^{-n} \} ) \\
 & \leq \eps/2 + \sum_{n=1}^\infty ( \eps 2^{-n} + \eps 2^{-n} ) =  3\eps/2 \, .\label{a.3436b}
\end{align}
So we have found a relatively compact subset of $\U$ such that most of the mass is concentrated on it. Thus it remains to show that the total masses of the restricted measures $\{m|_{S_n}: \, m \in A\}$ is uniformly bounded, where $S_n = \{\mfu: \, \bar{\mfu} \geq n^{-1}\}$ for any $n\in \N$. Let $\eps < n^{-1}$. Then:
\begin{equation}\label{a.3437}
 \sup_{m \in A} m(S_n) \leq \sup_{m \in A} m[\mfu: \bar{\mfu} \geq \eps ] <M \, .
\end{equation}
And therefore $A$ is relatively compact by \cite[Proposition A2.6.IV]{DVJ03}.
\end{proof}
\begin{remark}\label{ar.3624}
A similar tightness criterion without the boundedness assumption \eqref{e:t-crit:tmass} holds for probability measures on $\U \setminus \{0\}$ without \eqref{e:t-crit:tmass} ($M=1$ will always do it).
\end{remark}

\subsection{Separation of measures}\label{ss.sepmeas}

Boundedly-finite measures in ultrametric spaces appear in the L\'evy-Khintchine formula.
In that context the next proposition gives a helpful statement.

\begin{proposition}\label{c.tr2}
 The set of functions 
 \begin{equation}\label{a.3438} 
   \mcF = \textup{span} \left\{\mfu \mapsto 1- e^{-\Phi (\mfu )}:\, \Phi \in \mcA(\Pi_{h,+})   \right\}
 \end{equation}
 is convergence determining on $\mcM^{\#}(\U(h)^\sqcup \setminus \{0\})\cap \{\mu: \, \int \mu(\dx \mfu)\, (1\wedge \bar{\mfu}) < \infty\}$.
\end{proposition}
\begin{proof}\label{ap.3642}
 We check the three conditions of Theorem 2.3 in \cite{LR15}.
 The set $\mcF$ is closed under multiplication, so (T.1) holds.
 Moreover, $\mcF$ defines the topology on $\U(h)^\sqcup \setminus \{0\}$ by Theorem~(\ref{p.trunc.poly}), so (T.2) holds.
 Finally, if $A \subset \U(h)^\sqcup \setminus \{0\}$ is bounded, then there is $c>0$ with $\inf_{\mfu \in A} \bar{\mfu} \geq c$. But then we have for any polynomial $\Phi^{m,\phi}$ with $\phi \geq 1$ uniformly that $\inf_{\mfu\in A} (1-\exp(-\Phi(\cdot)))(\mfu) \geq c^m >0$. This establishes (T.3).
\end{proof}

\bibliography{infdiv}
\bibliographystyle{alpha}

\paragraph{Acknowledgement}
We thank two anonymous referees for the careful reading and many suggestions helping to improve the manuscript.

\end{document}